\documentclass[reqno]{amsart}

\usepackage{amssymb}
\usepackage{graphicx}
\usepackage{mathrsfs}
\usepackage{hyperref}
\usepackage{color}
\usepackage{url}
\usepackage{cite}

\usepackage{times}
\usepackage{microtype}
\usepackage{amssymb}
\usepackage{mathtools}
\usepackage{tikz}
\usepackage{wasysym}
\usepackage{centernot}
\usepackage{color}
\usepackage{setspace}
\usepackage{appendix}
\usepackage{booktabs}
\usepackage{multirow}

\def\bR{{\mathbb R}}

\def\sE{{\mathscr E}}
\def\sF{{\mathscr F}}
\def\sA{{\mathscr A}}
\def\sG{{\mathscr G}}

\def\re{{\mathrm{e}}}

\def\bH{\mathbf{H}}

\def\sN{\mathscr{N}}
\def\cH{\mathcal{H}}
\def\sL{\mathscr{L}}
\def\cD{\mathcal{D}}
\def\bH{\mathbf{H}}
\def\bD{\mathbf{D}}
\def\cH{\mathcal{H}}
\def\Tr{\text{Tr}}
\def\bn{\mathbf{n}}
\def\bS{\mathbf{S}}

\def\Tr{\mathop{\mathrm{Tr}}}

\def\${|\!|\!|}
\def\l|{\left|\!\left|\!\left|}
\def\r|{\right|\!\right|\!\right|}

\newtheorem{theorem}{Theorem}[section]
\newtheorem{lemma}[theorem]{Lemma}
\newtheorem{proposition}[theorem]{Proposition}
\newtheorem{corollary}[theorem]{Corollary}
\theoremstyle{definition}
\newtheorem{definition}[theorem]{Definition}
\newtheorem{example}[theorem]{Example}

\theoremstyle{remark}
\newtheorem{remark}[theorem]{Remark}
\numberwithin{equation}{section}

\begin{document}

\title[Generalizing DN operators]{Generalizing Dirichlet-to-Neumann operators}

%    Information for first author
\author{Liping Li}
%    Address of record for the research reported here
\address{RCSDS, HCMS, Academy of Mathematics and Systems Science, Chinese Academy of Sciences, Beijing, China.  }
\address{Bielefeld University,  Bielefeld, Germany.}
%    Current address
%\curraddr{Bielefeld University,  Bielefeld, Germany.}
\email{liliping@amss.ac.cn}
%    \thanks will become a 1st page footnote.
\thanks{The first named author is partially supported by NSFC (No. 11931004),  and Alexander von Humboldt Foundation in Germany.}

%    Information for second author
%\author{Wenjie Sun}
%\address{School of Mathematical Sciences, Tongji University,  China.}
%\email{wjsun@tongji.edu.cn}
%\thanks{Support information for the second author.}

%    General info

%\date{January 1, 2001 and, in revised form, June 22, 2001.}

%\dedicatory{This paper is dedicated to our advisors.}


\begin{abstract}
The aim of this paper is to study the  Dirichlet-to-Neumann operators in the context of Dirichlet forms and especially to figure out their probabilistic counterparts.  Regarding irreducible Dirichlet forms,  we will show that the  Dirichlet-to-Neumann operators for them are associated with the trace Dirichlet forms corresponding to the time changed processes on the boundary.  Furthermore,  the  Dirichlet-to-Neumann operators for perturbations of Dirichlet forms will be also explored.   It turns out that for typical cases such a  Dirichlet-to-Neumann operator corresponds to a quasi-regular positivity preserving (symmetric) coercive form,  so that there exists a family of Markov processes associated with it via Doob's $h$-transformations. 

%The main result shows that these operators are  %Special examples include the classical Dirichlet-to-Neumann operators and the Dirichlet-to-Neumann operators for certain L\'evy type operators.  
%Furthermore,   will be also explored.  This development is motivated by the celebrated Calder\'on's problem.  
\end{abstract}

\maketitle

\tableofcontents

\section{Introduction}

Let $\Omega\subset \bR^d$,  $d\geq 2$, be a bounded Lipschitz domain,  i.e. the boundary $\Gamma:=\partial \Omega$ is locally the graph of a Lipschitz function. 
The classical Dirichlet-to-Neumann operator (DN operator in abbreviation) $D$ over $\Omega$  is roughly defined as follows: For $\varphi\in L^2(\Gamma)$ such that there exists a (unique) harmonic function $u\in H^1(\Omega)$,  i.e.  $\Delta u=0$ in $\Omega$ in the weak sense,   with $\Tr(u)=\varphi$ and having a weak normal derivative $\partial_\bn u\in L^2(\Gamma)$,  $D\varphi$ is defined as $\frac{1}{2}\partial_\bn u$.  The domain $\cD(D)$ of $D$ is the totality of all such $\varphi\in L^2(\Gamma)$.  Here $\Tr(u)$ is the trace of $u$ on $\Gamma$ and the weak normal derivative  $\partial_\bn u$ is determined by the Green-Gauss formula \eqref{eq:GreenGauss}.  The rigorous definition will be stated in Definition~\ref{DEF31}. 

 A systemic introduction to DN operators is referred to,  e.g.,  the monograph of Taylor \cite[Section 12C]{T96}.  Analogical DN operators can be defined for other operators,  such as Schr\"odinger operators,  in place of the Laplacian; see,  e.g.,  \cite{AE20, AE14, BE17, EO14, EO19}.  Under mild conditions,  these DN operators are lower semi-bounded and self-adjoint on $L^2(\Gamma)$. There appear various analytic approaches to study them in recent years.  For example,  ter Elst and Ouhabaz \cite{EO14, EO19} showed that the strongly continuous semigroup corresponding to a DN operator is given  by a kernel which satisfies Poisson upper bounds.  Other properties like positivity and irreducibility for the semigroup are explored in,  e.g.,  \cite{AM12,  AE20}.   Based on an indirect ellipticity property called hidden compactness,  Arendt et al.  \cite{AE14} investigated DN operators that may be multi-valued.  However the literatures for probabilistic approach to DN operators are very few.  To our knowledge,  it first arose in \cite{BV17} that the classical DN operator is associated with certain Markov process on $\Gamma$.  
 
% The related problems have been widely studied by using different analytic approach in recent years.  The DN operators on rough domains are studied in \cite{AE11}.  Behrndt and ter Elst \cite{BE15} investigate them in the framework of linear relations in Hilbert space.  
%On the other hand, taking a suitable operator $\sL$ in place of $\Delta$,  one can analogically define the corresponding DN operator for $\sL$.  For example,  the DN operators for Schr\"odingder operators are explored in  Nonlocal type DN operators are under consideration in,  e.g.,  \cite{W15, W18}. 

It is our first aim in this paper to generalize DN operators to those for operators related to irreducible Dirichlet forms.  Terminologies and notations concerning Dirichlet forms are referred to,  e.g.,  \cite{FOT11, CF12}.  Let $\sL$ be a self-adjoint operator  that is the generator of a regular and irreducible Dirichlet form $(\sE,\sF)$ on $L^2(E,m)$.  The `boundary' and the underlying measure on the `boundary' are chosen as a quasi closed set $F\subset E$ and a positive Radon smooth measure $\mu$ with $\text{qsupp}[\mu]=F$,  where $\text{qsupp}[\mu]$ stands for the quasi support of $\mu$.  A function $u$ is called harmonic in the `interior'  $G:=E\setminus F$ with respect to $\sL$ if $u\in \sF_\re$ and 
$$\sE(u,g)=0,\quad \forall g\in \sF_\re^G,$$ 
where $\sF_\re$ is the extended Dirichlet space of $(\sE,\sF)$ and $\sF^G_\re:=\{u\in \sF_\re: u=0,\sE\text{-q.e. on }F\}$.  Then the  DN operator $\sN$ for $\sL$ (also called for $(\sE,\sF)$) is by definition mapping `the trace on the boundary' $\varphi:=u|_F$ of certain harmonic function $u$ to its `weak normal derivative on the boundary' $f$ determined by
\begin{equation}\label{eq:11}
	\sE(u,v)=\int_F fv|_F d\mu,\quad \text{for any } v\in \sF_\re \text{ with }v|_F\in L^2(F,\mu).  
\end{equation}
Here the `trace' $u|_F$ is the restriction of ($\sE$-quasi-continuous) $u$ to $F$.  We will show in Theorem~\ref{THM26} that $\sN$ is a well-defined positive self-adjoint operator on $L^2(F,\mu)$.  More significantly,  $-\sN$ is the $L^2$-generator of the trace Dirichlet form of $(\sE,\sF)$ on $L^2(F,\mu)$,  whose associated Markov process is the time changed process of $X$ by the positive continuous additive functional (PCAF in abbreviation) corresponding to the Revuz measure $\mu$.  Here $X$ is the Markov process   associated with $(\sE,\sF)$.  This new approach leads to a rich family of DN operators containing those for self-adjoint operators appearing in the literatures like \cite{V21, BV17,  W15, W18}.  Particularly,  the classical DN operator $D$ can be recovered as follows: Consider $(\sE,\sF)=(\frac{1}{2}\mathbf{D},H^1(\Omega))$ on $L^2(\bar{\Omega})$,  where $\bar{\Omega}$ is the closure of $\Omega$ and
\[
	\mathbf{D}(u,v):=\int_\Omega \nabla u(x) \nabla v(x)dx,\quad u,v\in H^1(\Omega),  
\] 
whose is  associated  with the reflected Brownian motion on $\bar{\Omega}$.  Take $F=\Gamma$ and $\mu=\sigma$,  the surface measure on $\Gamma$.  Then $D$ is identified with the  DN operator for $(\sE,\sF)$ on $L^2(\Gamma)$;%  see \S\ref{SEC4}. 

% Let  be a  ,  $F\subset E$ be a quasi closed set and $\mu$ be a positive Radon smooth measure with respect to $\sE$ whose quasi support is $F$.  The extended Dirichlet space and the generator of $(\sE,\sF)$ are denoted by $\sF_\re$ and $\sL$ respectively. Then the generalized DN operator $\sN$ related to $F$ and $\mu$ is defined as follows: Its domain $\cD(\sN)$ is by definition the family of all $\varphi\in L^2(F,\mu)$ satisfying that there exist $f\in L^2(F,\mu)$ and a harmonic function $u\in \sF_\re$ with respect to $\sL$ in $G:=E\setminus F$ such that $u|_F=\varphi$ and

%Meanwhile $\sN \varphi:=f$.  Here the harmonicity of $u$ means that 
% Note that the harmonic $u$ with $u|_F=\varphi $ is unique if exists,  due to the irreducibility of $(\sE,\sF)$,  and \eqref{eq:11} instead of \eqref{eq:GreenGauss} gives the analogue of the weak normal derivative for $u$.  

%One of the main results,  ,  figures out the probabilistic significance of the generalized DN operator $\sN$:  $-\sN$ is the generator of the trace Dirichlet form As far as we know,  a probabilistic approach to the classical DN operator first appeared in \cite{BV17}.  Then the DN operator for certain L\'evy type operator is formulated in \cite{V21} by means of the theory of Dirichlet forms.  We will see in \S\ref{SEC4} that all the main results in these literatures can be recovered by using the characterization involving trace Dirichlet forms.  

Nevertheless this new approach seems not to make us satisfied,  because the Schr\"odinger operators like $\frac{1}{2}\Delta-V$ with $V\in L^\infty(\Omega)$ are hardly associated with a Dirichlet form.  To overcome this lack,  we further introduce the  DN operators for the perturbation of $(\sE,\sF)$.  Let $\bS$ be the family of all positive smooth measures with respect to $\sE$ and take $\kappa=\kappa^+-\kappa^-\in \bS-\bS$, i.e. $\kappa^\pm\in \bS$ and $\kappa^+\perp \kappa^-$.  Then
\[
\begin{aligned}
	\sF^\kappa:=\sF\cap L^2(E,|\kappa|),  \quad \sE^\kappa(u,v):=\sE(u,v)+\int_E uvd\kappa,\; u,v\in \sF^\kappa
\end{aligned}\]
is called the perturbation of $(\sE,\sF)$ by $\kappa$.  Under suitable conditions $(\sE^\kappa,\sF^\kappa)$ is a lower bounded symmetric closed form,  whose $L^2$-generator $\sL_\kappa$ is upper semi-bounded and self-adjoint.  Particularly,  $\frac{1}{2}\Delta-V$ corresponds to the perturbation of $(\frac{1}{2}\mathbf{D},H^1(\Omega))$ by $V(x)dx$.  The  DN operator $\sN_\kappa$ for $\sL_\kappa$ is defined by a similar way to that of $\sN$;  see Definition~\ref{DEF41}.  Then one of the main results in this part, Theorem~\ref{THM44},  states a useful sufficient condition for the self-adjointness of $\sN_\kappa$ on $L^2(F,\mu)$ by an argument involving certain compact embedding property. 

A problem of great interest to us is to figure out the probabilistic counterparts of  DN operators.  This is completely solved for the  DN operators for Dirichlet forms.  However we are stuck in the perturbation case because the $L^2$-semigroup associated with $\sN_\kappa$ is usually not Markovian.  Instead the main result for this case,  Theorem~\ref{THM312},  states that under certain conditions  $\sN_\kappa$ corresponds to a quasi-regular positivity preserving (symmetric) coercive form,  so that there exists a family of Markov processes that are associated with $\sN_\kappa$ via Doob's $h$-transformations.  Several concrete perturbations including constant perturbation of the Laplacian,  perturbations of uniformly elliptic operators and perturbations supported on boundary will be paid special attention to.  For most of them,  the $L^2$-semigroup associated with $\sN_\kappa$ is shown to be irreducible.  As a result, $\sN_\kappa$ has a (unique) ground state,  i.e.  the smallest element in the spectrum $\sigma(\sN_\kappa)$ is a simple eigenvalue admitting a strictly positive eigenfunction,  called the ground state,  and no other eigenvalues admit strictly positive eigenfunctions.  More significantly,  this ground state corresponds to the unique $h$-transformation such that the Markov process obtained by $h$-transformation associated with $\sN_\kappa$ is recurrent;  see Example~\ref{EXA65}, Theorems~\ref{THM412} and \ref{THM55}. 

%A problem of great interest to us is to figure out the probabilistic counterparts of  DN operators.  This is completely solved for the  DN operators for Dirichlet forms.  However we are stuck in the perturbation case because the $L^2$-semigroup associated with $\sN_\kappa$ is usually not Markovian. 

%For either case it turns out that the quadratic form induced by $\sN_\kappa$ is a 
%The second aim of this paper is to introduce and to study the DN operator for the perturbation of $(\sE,\sF)$.  Note that $(\sE^\kappa,\sF^\kappa)$ is not necessarily a (lower bounded) closed form.  Replacing $(\sE,\sF)$ by $(\sE^\kappa,\sF^\kappa)$,  we can introduce a  DN operator $\sN_\kappa$ relative to the same $F$ and $\mu$ by a similar definition to that of $\sN$.  We should emphasize that  for a given $\varphi$ on $F$,  its harmonic extensions  for $\sE^\kappa$ (i.e.  $u$ is harmonic in $G$ for $\sE^\kappa$ and $u|_F=\varphi$) may be not unique.  Hence $\sN_\kappa$ is well defined only when a mild condition is imposed;  see Definition~\ref{DEF41}.  Examples of  DN operator for perturbation include particularly the DN operators for Schr\"odinger operators,  for which $\kappa(x)=V(x)m(dx)$.  We will also pay special attention to the perturbation on the boundary $F$ as well as the perturbation of certain non-local Dirichlet form.  %Particularly,  it turns out that under suitable conditions,  $\sN_\kappa$ is usually a lower bounded self-adjoint operator on $L^2(F,\mu)$.  

The developments of DN operators for perturbations are also motivated by the so-called Calder\'on's problem initiated in a pioneer contribution \cite{C80}.  This is an inverse problem considering whether one can determine the electrical conductivity of a medium by making voltage and current measurements at the boundary of the medium.  A systemic survey is referred to \cite{U12}.  To be more mathematical,  take two potentials $V_1,V_2\in L^\infty(\Omega)$ and let $D_{V_i}$ be the DN operator for the Schr\"odinger operator $\frac{1}{2}\Delta-V_i$.  Then the classical Calder\'on's problem is to ask if $D_{V_1}=D_{V_2}$ implies $V_1=V_2$.  The case of dimension $d\geq 3$ was solved in a seminal paper \cite{SU87} and the answer is positive.  Then the case of dimension $2$ is considered in, e.g., \cite{B08}.  Other related literatures include,  e.g.,   \cite{SU91,  BGU21, GSU20}.  In this paper we will also give some remarks on the Calder\'on's problem in the context of Dirichlet forms. 

%and present some clues for a possible probabilistic approach to a  Calder\'on's problem related to the perturbation of certain non-local Dirichlet form.

The rest of this paper is organized as follows.  The section \S\ref{SEC2}  is devoted to introducing and studying the  DN operator for a regular and irreducible Dirichlet form.  The classical DN operator will be recovered in \S\ref{SEC4}.   In \S\ref{SEC3},  we will explore DN operators for perturbations.  %Some new examples are raised in \S\ref{SEC33}.  
Their probabilistic counterparts are figure out in \S\ref{SEC34}. In \S\ref{Revisiting},  the DN operator $D_\lambda$ for $\frac{1}{2}\Delta-\lambda$ with a constant $\lambda\in \bR$ is explored.  The quadratic form induced by $D_\lambda$ is formulated in Theorem~\ref{THM44-2}.  As a corollary,  we obtain the irreducibility of the $L^2$-semigroup associated with $D_\lambda$ in Corollary~\ref{COR410} for $\lambda>\lambda^\text{D}_1/2$,  where $\lambda^\text{D}_1$ is the first eigenvalue of the Dirichlet Laplacian.  Then Theorem~\ref{THM412} states more properties about the Markov processes obtained by $h$-transformations associated with $D_\lambda$.  Perturbations of uniformly elliptic operators and perturbations supported on boundary are studied in \S\ref{SEC5} and \S\ref{SEC7} respectively.  

%Several examples will be raised. The results concerning the perturbation on $F$ will be presented in \S\ref{SEC42} and \ref{SEC43-2}.   Then the case of Schr\"odinger operator will be explored in \S\ref{SEC5} and the perturbation of certain non-local Dirichlet form will be studied in \S\ref{SEC7}.  

\subsection*{Notations}

We prepare notations that will be frequently used for handy reference.   Given a topological space $E$,  $\mathcal{B}(E)$ represents the family of all Borel measurable functions on $E$.  For a family $\sG$ of certain functions,  $p\sG$ (resp. $b\sG$) stands for the subfamily consisting of positive (resp. bounded) functions in $\sG$.   Given a measure $m$ or a function $u$ on $E$ and $F\subset E$,  $m|_F$ or $u|_F$ stands for the restriction of $m$ or $u$ to $F$.  For $u,v\in L^2(E,m)$,  $(u,v)_m:=\int_E uv dm$.  Given an operator $\sL$ on a Hilbert space,  $\cD(\sL)$ stands for its domain and $\sigma(\sL)$ stands for its spectrum. 

The symbol $\lesssim$ (resp. $\gtrsim$) means that the left (resp. right) term is bounded by the right (resp. left) term multiplying a non-essential constant.

For a regular or quasi-regular Dirichlet form $(\sE,\sF)$ on $L^2(E,m)$,  its extended Dirichlet space is denoted by $\sF_\re$.  For convenience,  every function in $\sF_\re$ is taken to be its $\sE$-quasi continuous $m$-version if without other statements.  For $\alpha\geq 0$,  $\sE_\alpha(u,v):=\sE(u,v)+\alpha (u,v)_m$ and $\|u\|_{\sE_\alpha}:=\sE_\alpha(u,u)^{1/2}$ for all $u,v\in\sF$.  

Let $\Omega\subset \bR^d$,  $d\geq 2$,  be a domain.  Set $$H^1(\Omega)=W^{1,2}(\Omega)=\{u\in L^2(\Omega): \partial_{x_i}u\in L^2(\Omega), 1\leq i\leq d\}$$and for $u,v\in H^1(\Omega)$,  $\mathbf{D}(u,v):=\int_\Omega \nabla u(x)\nabla v(x)dx$.  The notation $H^1_0(\Omega)$ stands for the closure of $C_c^\infty(\Omega)$ in $H^1(\Omega)$,  where $C_c^\infty(\Omega)$ is the family of all smooth functions  with compact support in $\Omega$.  Similarly,  $H^s(\Omega)=W^{s,2}(\Omega)$ denotes the Sobolev space of order $s\geq 0$.  To be more general,  $W^{m,p}(\Omega)$,  $m\in \mathbb{N},  p\geq 1$,  is the usual Sobolev space over $\Omega$.

\section{DN operators for irreducible Dirichlet forms}\label{SEC2}

\subsection{Basic setting}

What we are concerned with is a regular and irreducible Dirichlet form $(\sE,\sF)$ on $L^2(E,m)$   associated with an $m$-symmetric Markov process $X=(X_t, \mathbf{P}_x)$ on $E_\partial:=E\cup \{\partial\}$, where $E$ is a locally compact separable metric space attached by a trap $\partial$ and $m$ is a fully supported Radon measure on $E$.  The irreducibility particularly implies the recurrence or transience of $(\sE,\sF)$.  Denote by $\mathring{\mathbf{S}}$ the totality of positive Radon measures on $E$ charging no $\sE$-polar sets.  Take a non-zero $\mu\in \mathring{\bS}$ and set $F:=\text{qsupp}[\mu]$ (see \cite[Definition~3.3.4]{CF12}). Note that $F$ is quasi closed.  For convenience, we always take quasi closed (resp. quasi open) set to be a nearly Borel and finely closed (resp. finely open) $\sE$-q.e.  version,  so that its first hitting time,  e.g., $\sigma_F$ in Lemma~\ref{LM21},  is well defined. 
Let $G:=E\setminus F$.   Denote by $(\sE^G,\sF^G)$ the part Dirichlet form of $(\sE,\sF)$ on $G$, i.e. 
\[
	\sF^G=\{u\in \sF: u=0, \sE\text{-q.e. on }F\},\quad \sE^G(u,v)=\sE(u,v),\; u,v\in \sF^G. 
\]
Then $(\sE^G,\sF^G)$ is quasi-regular on $L^2(G,m|_G)$   associated with the part process of $X$ on $G$,  whose extended Dirichlet space is
\[
	\sF_{\re}^{G}:=\{u\in \sF_\re: u=0, \sE\text{-q.e. on }F\};
\]
see \cite[Theorems~3.3.8 and 3.4.9]{CF12}.  Write
\[
	\mathcal{H}_F:=\{u\in \sF_\re: \sE(u,v)=0,\forall v\in \sF_{\re}^{G}\}.  
\]
Note that $\sF^G_\re\cap \cH_F=\{0\}$.  
The following decomposition is elementary due to the irreducibility.

\begin{lemma}\label{LM21}
Every $u\in \sF_\re$ can be written as a sum
\[
	u=u_1+u_2,
\]
where $u_1\in \sF_{\re}^{G}$ and $u_2\in \mathcal{H}_F$.  This decomposition is unique,  and indeed
\[
	u_1=u-\mathbf{H}_F u,\quad u_2=\mathbf{H}_F u,
\]
where $\mathbf{H}_F u(x):=\mathbf{E}_x\left[u(X_{\sigma_F}); \sigma_F<\infty\right]$ for $x\in E$ and $\sigma_F:=\inf\{t>0:X_t\in F\}$.  
\end{lemma}
\begin{proof}
Note that $(\sE,\sF)$ is recurrent or transient due to the irreducibility.  When $(\sE,\sF)$ is transient,  the assertions have been concluded in, e.g., \cite[Theorem~3.4.2]{CF12}.  When $(\sE,\sF)$ is recurrent,  the above decomposition still holds in view of \cite[Theorem~3.4.8]{CF12},  and the uniqueness can be easily obtained by means of \cite[Theorem~5.2.16]{CF12}. 
\end{proof}
\begin{remark}
In a little abuse of notation,  we may write
\begin{equation}\label{eq:Fe}
	\sF_\re = \sF_{\re}^G\oplus \mathcal{H}_F.
\end{equation}
When $(\sE,\sF)$ is transient,  $\sF_\re$ is a Hilbert space with the inner product $\sE$ and $\sF_{\re}^G$ is clearly a closed subspace of $\sF_\re$.  Then $\mathcal{H}_F$ is the orthogonal complement of $\sF_{\re}^{G}$ and \eqref{eq:Fe} is a real direct sum decomposition.  However when $(\sE,\sF)$ is recurrent,  $\sF_\re$ is not a Hilbert space because $1\in \sF_\re, \sE(1,1)=0$; see, e.g.,  \cite[Theorem~1.6.3]{FOT11}. 
\end{remark}

%Denote by $\mathring{\mathbf{S}}$ (resp. $\mathbf{S}$) the totality of positive Radon measures on $E$ charging no $\sE$-polar sets (resp. positive smooth measures on $E$),  and let
%\[
%	\mathring{\mathbf{S}}_F:=\{\mu\in \mathring{\mathbf{S}}: \text{the quasi support of }\mu\text{ is }F\}.  
%\]
%Clearly every measure in $\mathring{\mathbf{S}}_F$ is smooth.  

\subsection{DN operators}

%Let $(\sL_0,\cD(\sL_0))$ be the generator of $(\sE^G,\sF^G)$ on $L^2(G,m|_G)$.  

%Set an auxiliary linear operator $(L, \mathcal{D}(L))$ as follows: $\mathcal{D}(L)$ is the totality of $u\in \sF_{\re}$ such that there exists $w\in L^2(E,m)$ satisfying
%\begin{itemize}
%\item[(L1)] $\sE(u,v)=-(w,v)_m$ for any $v\in \sF_\re^G$,
%\item[(L2)] $(w,v)_m=0$ for any $v\in \mathcal{H}_F$;
%\end{itemize}
%and $Lu:=w$ for $u, w$ as above.  It is easy to verify that $L$ is well defined and
%\begin{equation}\label{eq:HF}
%\cH_F\subset \cD(L),\quad Lu=0,\forall u\in \cH_F. 
%\end{equation}
For $\varphi\in L^2(F,\mu)$, $u$ is called the \emph{harmonic extension} of $\varphi$ if $u\in \cH_F$ and $u|_F=\varphi$.  
Note that the harmonic extension of $\varphi$ is unique if exists.   
We introduce the following  DN operator for $(\sE,\sF)$.  
%\begin{lemma}
%It holds $\cD(L)=\cD(\sL_0)\oplus \cH_F$,  and for $u=u_1+u_2$ with $u_1\in \cD(\sL_0)$ and $u_2\in \cH_F$, 
%\[
%	Lu=\sL_0u_1. 
%\]
%Particularly, $Lu=0$ for any $u\in\cH_F$.  
%\end{lemma}
%\begin{proof}
%For $u=u_1+u_2\in \cD(\sL_0)\oplus \cH_F$,  set $w:=\sL_0u_1$.  
%\end{proof}

\begin{definition}\label{DEF23}
Let $\mu\in \mathring{\bS}$ and $F=\text{qsupp}[\mu]$.  The following operator
\begin{equation}\label{eq:DN2}
\begin{aligned}
	\cD(\sN)=&\big\{\varphi\in L^2(F,\mu): \exists\,u\in \cH_F\text{ and }f\in L^2(F,\mu)\text{ such that }u|_F=\varphi, \\
	&\qquad\qquad \sE(u,v)=\int_{F} f v|_F d\mu \text{ for any }v\in \sF_\re\text{ with }v|_F\in L^2(F,\mu)\big\},\\
\sN \varphi=& f,\quad \varphi\in \mathcal{D}(\sN)
\end{aligned}\end{equation}
 is called the \emph{DN operator for $(\sE,\sF)$ on $L^2(F,\mu)$}.   
\end{definition}
%\begin{remark}
%We should point out that,  on account of \cite[Theorem~3.3.5]{CF12},  $u|_F$ is $\mu$-a.e. defined for any $\sE$-quasi-continuous function $u$.  Hence $u|_F=\varphi$ in the definition of $\mathcal{D}(\sN)$ is well defined. 
%\end{remark}

In what follows,  we figure out the probabilistic counterpart of $\sN$.  Note that $\mu$ induces a PCAF $A=(A_t)_{t\geq 0}$, whose support is $F$; see, e.g.,  \cite[Theorem~5.1.5]{FOT11}.  Let $\zeta:=\inf\{t>0:X_t=\partial\}$ be the lifetime of $X$ and we denote $F\cup \{\partial\}$ by $F_\partial$ regarding it as a topological subspace of $E_\partial$.  The right continuous inverse $\tau_t$ of $A$ is defined as
\[
	\tau_t:=\left\lbrace
	\begin{aligned}
		&\inf\{s>0: A_s>t\},\quad t<A_{\zeta-}, \\
		&\infty,\qquad\qquad\qquad \qquad\;\, t\geq A_{\zeta-}. 
	\end{aligned} \right.
\]
Set $\check{X}_t:=X_{\tau_t}$ for $t\geq 0$ and $\check{\zeta}:=A_{\zeta-}$.
Then $\check{X}=(\check{X}_t, \check{\zeta}, \{\mathbf{P}_x\}_{x\in F_\partial})$ is the so-called \emph{time changed process} of $X$ by the PCAF $A$.  It is known that $\check{X}$ is a right process   associated with the quasi-regular Dirichlet form $(\check{\sE},\check{\sF})$ on $L^2(F,\mu)$: 
\begin{equation}\label{eq:traceDirichletform}
\begin{aligned}
	\check{\sF}&=\sF_\re|_F\cap L^2(F,\mu),\\
	\check{\sE}(\varphi, \phi)&=\sE(\bH_F \varphi, \bH_F\phi),\quad \varphi, \phi\in \check{\sF},
\end{aligned}
\end{equation}
where $\mathbf{H}_F \varphi(x):=\mathbf{E}_x\left[\varphi(X_{\sigma_F}); \sigma_F<\infty\right]$ for $x\in E$;  see,  e.g., \cite[Theorem~5.2.7]{CF12}.  The Dirichlet form $(\check{\sE},\check{\sF})$ is called the \emph{trace Dirichlet form} of $(\sE,\sF)$ on $L^2(F,\mu)$.  Let $F^*$ be the topological support of $\mu$,  i.e. the smallest closed set such that $\mu(E\setminus F^*)=0$.  Then $F\subset F^*$,  $\sE$-q.e.  (but $F^*\setminus F$ is not necessarily $\sE$-polar).  It is worth noting that $(\check{\sE},\check{\sF})$ can be realized as a regular Dirichlet form on $L^2(F^*,\mu)$ ($=L^2(F,\mu)$);  see,  e.g.,  \cite[Theorem~5.2.13]{CF12}. 

\begin{theorem}\label{THM26}
Let $\sN$ be the DN operator for $(\sE,\sF)$ on $L^2(F,\mu)$  and $(\check{\sE},\check{\sF})$ be the trace Dirichlet form of $(\sE,\sF)$ on $L^2(F,\mu)$.  Then $-\sN$ is identified with the generator of $(\check{\sE},\check{\sF})$ on $L^2(F,\mu)$.  Particularly,  $\sN$ is a positive and self-adjoint operator on $L^2(F,\mu)$ corresponding to a strongly continuous Markovian semigroup.  
\end{theorem}
\begin{proof}
It suffices to show that $\varphi\in \cD(\sN)$,  $\sN \varphi=f$,  if and only if $\varphi\in \check{\sF}$ and $\check{\sE}(\varphi,  \phi)=(f, \phi)_\mu$ for any $\phi\in \check{\sF}$.  To do this,  we first take $\varphi\in \cD(\sN)$ with $\sN\varphi=f$.  Then there exists $u\in \cH_F\subset \sF_\re$ such that $\varphi=u|_F\in L^2(F,\mu)$ and 
\begin{equation}\label{eq:29}
 \sE(u,v)=\int_{F} f \cdot v|_F d\mu
\end{equation}
for any $v\in \sF_\re$ with $v|_F\in L^2(F,\mu)$.  Particularly $\varphi\in \sF_\re|_F\cap L^2(F,\mu)=\check{\sF}$.  Note that
\[
	\sE(u,v)=\sE(u,\bH_Fv)=\sE(\bH_Fu,\bH_Fv).
\]
It follows from \eqref{eq:29} that for any $\phi\in \check{\sF}$ with $\phi=v|_F$ and $v\in \sF_\re$,  
\[
\check{\sE}(\varphi, \phi)=\sE(\bH_Fu,\bH_F v)=(f, \phi)_\mu. 
\]

To the contrary,  let $\varphi\in \check{\sF}$ and $f\in L^2(F,\mu)$ such that $\check{\sE}(\varphi,  \phi)=(f, \phi)_\mu$ for any $\phi\in \check{\sF}$.  Then $u:=\bH_F \varphi\in \cH_F$ is the harmonic extension of $\varphi$.  We assert that $u,f$ satisfy the condition in \eqref{eq:DN2}.  To do this,  set $\phi:=v|_F\in \check{\sF}$ for $v\in \sF_\re$ with $v|_F\in L^2(F,\mu)$.  By means of Lemma~\ref{LM21} and \eqref{eq:traceDirichletform},  we can obtain that
\[
	(f,v|_F)_\mu=\check{\sE}(\varphi, \phi)=\sE(\bH_F \varphi, \bH_F\phi)=\sE(u,v).  
\]
Hence \eqref{eq:29} is concluded.  That completes the proof. 
\end{proof}

\subsection{Recovering classical DN operators}\label{SEC4}
%\subsection{Classical DN operators on bounded Lipschitz domains}\label{SEC31}

Let $\Omega\subset \mathbb{R}^d$,  $d\geq 2$,  be a bounded Lipschitz domain.  Let $H^1(\Omega):=\{u\in L^2(\Omega): \partial_{x_i}u\in L^2(\Omega), 1\leq i\leq d\}$ and $L^2(\Gamma)$ be the $L^2$-space on $\Gamma$ with respect to the surface measure $\sigma$,  i.e.  the restriction of $d-1$ Hausdorff measure to $\Gamma$.  Similarly,  $H^s(\Omega)$ denotes the Sobolev space of order $s\geq 0$.  In addition,  we can define the Sobolev spaces $H^s(\Gamma)$ for $0\leq s\leq 1$ in the usual way using local coordinate representations of $\Gamma$; see,  e.g.,  \cite[\S2.4]{SS11}. 

Since $\Omega$ is Lipschitz,  there is a unique \emph{trace operator} $\text{Tr}: H^1(\Omega)\rightarrow L^2(\Gamma)$ such that $\Tr(u)=u|_{\Gamma}$ for $u\in H^1(\Omega)\cap C(\bar{\Omega})$.  The \emph{weak normal derivative} is defined as follows.  For $u\in H^1(\Omega)$, we say $\Delta u\in L^2(\Omega)$ if there exists $f\in L^2(\Omega)$ such that $\int_\Omega \nabla u\nabla vdx+\int_\Omega fv=0$ for any $v\in H^1_0(\Omega)$.  In this case $\Delta u:=f$.  For $u\in H^1(\Omega)$ with $\Delta u\in L^2(\Omega)$,  we say $u$ has a weak normal derivative in $L^2(\Gamma)$ provided that there exists $f\in L^2(\Gamma)$ such that
\begin{equation}\label{eq:GreenGauss}
	\int_\Omega \nabla u\nabla v dx +\int_\Omega \Delta u vdx=\int_\Gamma f \Tr(v)d\sigma,\quad \forall v\in H^1(\Omega).  
\end{equation}
Meanwhile we denote by $\partial_{\bn}u:=f$ the weak normal derivative of $u$.  Note that \eqref{eq:GreenGauss} holds for every function in
\begin{equation}\label{eq:35}
\begin{aligned}
	\{u\in H^1(\Omega): &\Delta u\in L^2(\Omega),  u\text{ has a weak normal derivative }\partial_\bn u\in L^2(\Gamma)\} \\
	 &=H^{3/2}_\Delta(\Omega):=\{u\in H^{3/2}(\Omega): \Delta u\in L^2(\Omega)\} \\
	 &=W^1:=\{u\in H^1(\Omega): \Delta u\in L^2(\Omega), \Tr(u)\in H^1(\Gamma)\};  
\end{aligned}\end{equation}
see,  e.g.,  \cite[Lemma~2.2]{BV17}.   Following,  e.g., \cite{AM12,  AE11, EO14, BV17},  we present the DN operator $D_\lambda$ with a parameter $\lambda\geq 0$ on $L^2(\Gamma)$,  which maps the trace of certain harmonic $u\in H^1(\Omega)$ to its weak normal derivative,  in the following way. 

\begin{definition}\label{DEF31}
For $\lambda\geq 0$,  the operator
\[
\begin{aligned}
	\cD(D_\lambda)&:=\big\{\varphi\in L^2(\Gamma): \exists u\in H^1(\Omega) \text{ such that }\frac{1}{2}\Delta u=\lambda u, \Tr(u)=\varphi, \\
		&\qquad \qquad \qquad \qquad\qquad \text{ and }u\text{ has a weak normal derivative }\partial_\bn u\in L^2(\Gamma)\big\},  \\
		D_\lambda \varphi&:=\frac{1}{2}\partial_\bn u,\quad \varphi\in \cD(D_\lambda) \text{ and } u\text{ as above}
\end{aligned}
\]
is called the DN operator on $L^2(\Gamma)$.  Write $D:=D_0$ for the sake of brevity.  
\end{definition}

The main result Theorem~3.3 of \cite{BV17} concludes that $-D_\lambda$ is the $L^2$-generator of the time changed process of the $\lambda$-subprocess of the reflected Brownian motion on $\bar{\Omega}$.  In what follows, we will recover it as a special case of Theorem~\ref{THM26}.   To do this, consider the Dirichlet form $(\frac{1}{2}\bD, H^1(\Omega))$, where $\bD(u,v):=\int_\Omega \nabla u\nabla vdx$ for $u,v\in H^1(\Omega)$.  Since $\Omega$ is Lipschitz,  $(\frac{1}{2}\bD, H^1(\Omega))$ is a regular Dirichlet form on $L^2(\bar{\Omega})$, which is   associated with the reflected Brownian motion on $\bar{\Omega}$.  Let 
\begin{equation}\label{eq:Brownian}
	\sF:=H^1(\Omega),\quad \sE(u,v):=\frac{1}{2}\bD(u,v)+\lambda\cdot (u,v)_m, \; u,v\in \sF,
\end{equation}
where $\lambda\geq 0$ and $m$ is the Lebesgue measure on $\bar{\Omega}$.  Then $(\sE,\sF)$ is clearly a regular and irreducible Dirichlet form on $L^2(\bar{\Omega})$.  When $\lambda>0$, the   associated Markov process $X$ is the $\lambda$-subprocess of the reflected Brownian motion on $\bar\Omega$.  Take $\mu:=\sigma\in \mathring{\bS}$.
%Consider $\mu=\sigma\in \mathring{\bS}$,  and set
%\[
%	\Gamma_\sigma:=\text{qsupp}[\sigma].  
%\]
%Note that $\Gamma_\sigma\subset \Gamma$,  $\sE$-q.e.  and $\sigma(\Gamma\setminus \Gamma_\sigma)=0$ (although $\Gamma\setminus \Gamma_\sigma$ is not necessarily $\sE$-polar).  Particularly,  when $\Gamma$ is smoother (e.g.,  $C^3$ as in \cite[\S5.8 ($1^\circ$)]{CF12}),  one has $\Gamma_\sigma=\Gamma$.  Clearly,  $\Gamma$ is the topological support of $\sigma$.  Set $F=\Gamma_\sigma$ and $G=\bar{\Omega}\setminus F$.  
The following lemma is crucial to applying Theorem~\ref{THM26}.  

\begin{lemma}\label{LM32}
Let $(\sE,\sF)$ be given by \eqref{eq:Brownian}.  Then the following hold:
\begin{itemize}
\item[(1)] $\sF_\re=H^1(\Omega)$.  
\item[(2)] For any $\sE$-quasi-continuous $u\in \sF_\re$,  $u|_\Gamma=\Tr(u)$,  $\sigma$-a.e. 
\item[(3)] $\text{qsupp}[\sigma]=\Gamma$, $\sE$-q.e.,  and the topological support of $\sigma$ is also $\Gamma$. 
\item[(4)] $\sF_{\re}^{\Omega}=H^1_0(\Omega)$ and $\cH_\Gamma=\{u\in H^1(\Omega): \frac{1}{2}\Delta u=\lambda u\}$.  
\end{itemize}
\end{lemma}
\begin{proof}
\begin{itemize}
\item[(1)] It suffices to consider the case $\lambda=0$.  Since $\Omega$ is bounded and Lipschitz,  the Poincar\'e inequality 
\[
	\| f-\bar f\|_{L^2(\Omega)}^2\leq C\mathbf{D}(f,f),\quad f\in H^1(\Omega)
\]
is known to be true for some constant $C>0$ depending only on $\Omega$,  where $\bar f:=\frac{1}{m(\Omega)}\int_\Omega f(x)dx$.   Let $u\in \sF_\re$ and $\{u_n:n\geq 1\}\subset H^1(\Omega)$ be an approximating sequence for $u$.  Then $u_n\rightarrow u$ a.e.,  and the Poincar\'e inequality yields that $\{u_n-\bar{u}_n\}$ is a Cauchy sequence in $L^2(\Omega)$.  Hence there exists $v\in L^2(\Omega)$ such that $u_n-\bar{u}_n\rightarrow v$ in $L^2(\Omega)$.  Particularly,  taking a subsequence if necessary,  we have $u_n-\bar{u}_n\rightarrow v$, a.e.  Since $u_n\rightarrow u$, a.e.,  it follows that $u-v=\lim_{n\rightarrow\infty} \bar{u}_n$ is a.e. constant.  Therefore $u\in L^2(\Omega)$ because $v$ and constant functions are in $L^2(\Omega)$.  Eventually we can conclude that $\sF_\re=\sF_\re\cap L^2(\Omega)=\sF=H^1(\Omega)$.  
\item[(2)] For $u\in \sF_\re=H^1(\Omega)$,  we can find $\{u_n\}\subset H^1(\Omega)\cap C(\bar{\Omega})$ such that $u_n\rightarrow u$ in $H^1(\Omega)$.  On account of \cite[Theorem~2.3.4]{CF12},  taking a subsequence if necessary,  we get that $u_n$ converges to $u$,  $\sE$-q.e.  Since $\sigma$ charges no $\sE$-polar sets,  it follows that $u_n|_\Gamma \rightarrow u|_\Gamma$,  $\sigma$-a.e.  Note that $u_n|_\Gamma\rightarrow \Tr(u)$ in $L^2(\Gamma)$.  We can conclude that $u|_\Gamma=\Tr(u)$,  $\sigma$-a.e. 
\item[(3)] Clearly,  the topological support of $\sigma$ is $\Gamma$.  Denote $\Gamma_\sigma:=\text{qsupp}[\sigma]$.  Note that $\Gamma_\sigma\subset \Gamma$,  $\sE$-q.e.  It suffices to show that $\Gamma\setminus \Gamma_\sigma$ is $\sE$-polar.  Without loss of generality assume that $\Gamma_\sigma$ is nearly Borel and finely closed.  It follows from \cite[Theorem~3.3.5]{CF12} and $\sigma(\Gamma\setminus \Gamma_\sigma)=0$ that the part Dirichlet space of $(\sE,\sF)$ on $\Gamma_\sigma^c:=\bar{\Omega}\setminus \Gamma_\sigma$ is
\[
\begin{aligned}
	\sF^{\Gamma_\sigma^c}&=\{u\in \sF: u=0, \sE\text{-q.e. on }\Gamma_\sigma\} \\
	&=\{u\in \sF: u=0, \sigma\text{-a.e. on }\Gamma_\sigma\} \\
	&=\{u\in H^1(\Omega): \text{Tr}(u)=0\}.  
\end{aligned}\]
Using the trivial traces theorem over a Lipschitz domain (e.g.,  \cite[Theorem~4]{H20}),  we obtain that $\sF^{\Gamma_\sigma^c}$ is identified with $H^1_0(\Omega)$,   the part Dirichlet space of $(\sE,\sF)$ on $\Omega$.  On account of \cite[Theorem~3.3.8~(iii)]{CF12},  $\Gamma^c_\sigma\setminus \Omega=\Gamma\setminus \Gamma_\sigma$ is $\sE$-polar.  
\item[(4)] They are obvious by the first assertion.
\end{itemize}
That completes the proof. 
\end{proof}

In view of Lemma~\ref{LM32}~(2), we will not distinguish $u|_\Gamma$ and $\Tr(u)$ for $u\in H^1(\Omega)$ hereafter.  Denote the corresponding PACF of $\sigma$ by $L^\sigma:=(L^\sigma_t)_{t\geq 0}$,  also called the local time (of $X$) on $\Gamma$.  Let $(\check{\sE},\check{\sF})$ be the trace Dirichlet form of $(\sE,\sF)$ defined as \eqref{eq:traceDirichletform} with $F=\Gamma$ and $\mu=\sigma$.  %For emphasis,  we denote $(\check{\sE},\check{\sF})$ as $(\sE^*,\sF^*)$ when viewed as a Dirichlet form on $L^2(\Gamma)$ ($=L^2(\Gamma_\sigma)$).   Then $(\sE^*,\sF^*)$ is regular on $L^2(\Gamma)$ and $\Gamma\setminus \Gamma_\sigma$ is $\sE^*$-polar. 
Now we present the following result to recover \cite[Theorem~3.3]{BV17}.  

\begin{corollary}\label{COR28}
For $\lambda\geq 0$,  let $D_\lambda$ be the DN operator in Definition~\ref{DEF31}.  Then the following hold:
\begin{itemize}
\item[(1)] $\cD(D_\lambda)=H^1(\Gamma)$ and $D_\lambda \varphi=\frac{1}{2}\partial_\bn \bH_\Gamma \varphi$ for $\varphi \in H^1(\Gamma)$. 
\item[(2)] $-D_\lambda$ is the generator of the trace Dirichlet form $(\check\sE,\check\sF)$ on $L^2(\Gamma)$, which is associated with the time changed process of $X$ by the local time $L^\sigma$.  Furthermore,  $(\check\sE,\check\sF)$ is regular on $L^2(\Gamma)$.  
\end{itemize}
\end{corollary}
\begin{proof}
\begin{itemize}
\item[(1)] The case $\lambda=1$ has been considered in \cite[Lemma~2.2]{BV17}.  Now consider any $\lambda\geq 0$.  It suffices to show $\cD(D_\lambda)=H^1(\Gamma)$.  Take $\varphi \in \cD(D_\lambda)$.  It follows from \eqref{eq:35} that  $u:=\bH_\Gamma \varphi\in W^1$ and hence $\varphi=u|_\Gamma\in H^1(\Gamma)$.  To the contrary,  let $\varphi\in H^1(\Gamma)$.  Then there exists $u'\in H^1(\Omega)$ such that $u'|_\Gamma=\varphi$ due to the trace theorem.  Set $u:=\bH_\Gamma u' \in \sF_\re=H^1(\Omega)$ by means of Lemma~\ref{LM32}~(1).  We get from Lemma~\ref{LM32}~(4) that $\frac{1}{2}\Delta u=\lambda u\in L^2(\Omega)$.  Since $\Tr(u)=\varphi\in H^1(\Gamma)$, it follows from \eqref{eq:35} that $u$ has a weak normal derivative $\partial_\bn u\in L^2(\Gamma)$.  This implies $\varphi\in \cD(D_\lambda)$.  
\item[(2)] It can be straightforwardly verified by applying Theorem~\ref{THM26} and Lemma~\ref{LM32} with $F=\Gamma$ and $\mu=\sigma$.  The regularity of $(\check{\sE},\check{\sF})$ on $L^2(\Gamma)$ is due to \cite[Theorem~5.2.3]{CF12} and Lemma~\ref{LM32}~(3). 
\end{itemize}
That completes the proof. 
\end{proof}

%In fact,  $\check{\sF}=H^{1/2}(\Gamma)$ and $\|\cdot\|_{\check\sE_1}$ is equivalent to the norm of $H^{1/2}(\Gamma)$; see,  e.g.,  \cite[Theorems 2.6.8 and 2.6.11]{SS11}.  
When $d=2$, $\lambda=0$ and $\Omega=\mathbb{D}:=\{x: |x|<1\}$,  $\check{\sE}$ is identified with the celebrated \emph{Douglas integral};  see \cite{D31}.  To be precise,  $\Gamma=\partial \mathbb{D}=\{\theta:0\leq \theta<2\pi\}$ and
\[
\begin{aligned}
	&\check{\sE}(\varphi,\varphi)=\frac{1}{16\pi}\int_0^{2\pi}\int_0^{2\pi} \left(\varphi(\theta)-\varphi(\theta')\right)^2\sin^{-2}\left(\frac{\theta-\theta'}{2}\right)d\theta d\theta,\\
	&\check{\sF}=\{\varphi\in L^2(\partial \mathbb{D}): \check{\sE}(\varphi,\varphi)<\infty\}.  
\end{aligned}
\]
Particularly,  the DN operator $D$ corresponds to the Cauchy process on $\partial \mathbb{D}$.  

\section{DN operators for perturbations of Dirichlet forms}\label{SEC3}

%\subsection{DN operators for perturbation of Dirichlet forms}

Let $(\sE,\sF)$ be a regular and irreducible Dirichlet form on $L^2(E,m)$.  
Take $\kappa=\kappa^+-\kappa^-\in \bS-\bS$ and let 
$$\sF^\kappa=\sF\cap L^2(E,|\kappa|),\quad \sE^\kappa(u,v)=\sE(u,v)+\int_E uv d\kappa,\; u,v\in \sF^\kappa$$ be the perturbation of $(\sE,\sF)$ by $\kappa$,  where $|\kappa|=\kappa^++\kappa^-$; see Definition~\ref{DEFB1}.  Set
\[
	\sF^\kappa_\re:=\sF_\re\cap L^2(E,|\kappa|).
\]
We impose $\kappa^-\neq 0$ unless otherwise specified.  It is worth pointing out that $\kappa^-\neq 0$ may lead to the failure of closedness or Markovian property for $(\sE^\kappa,\sF^\kappa)$ that can not be straightforwardly linked with a certain Markov process. 

As reviewed in Appendix~\ref{APPB}, if $\kappa^-$ is $\sE^{\kappa^+}$-form bounded,  then $(\sE^\kappa,\sF^\kappa)$ is a lower bounded symmetric closed form on $L^2(E,m)$.  When $\kappa\in \bS$,  it becomes a quasi-regular Dirichlet form,  called the perturbed Dirichlet form of $(\sE,\sF)$ by $\kappa$.  Meanwhile $\sF_\re^\kappa$ is the extended Dirichlet space of $(\sE^\kappa,\sF^\kappa)$ enjoying the same quasi notions as $(\sE,\sF)$; see \cite[Proposition~5.1.9]{CF12}.  

\subsection{Definition}

Let $\mu\in \mathring{\bS}$ with $F:=\text{qsupp}[\mu]$.  Write $G:=E\setminus F$.  
Set 
\[
	\sF^{\kappa,G}_\re:=\{u\in \sF^\kappa_\re:  u=0,  \sE\text{-q.e.  on }F\}
\]
and 
\[
	\cH^\kappa_F:=\{u\in \sF^\kappa_\re: \sE^\kappa(u,v)=0,\forall v\in \sF^{\kappa,G}_\re\}.  
\]
Given $\varphi \in L^2(F,\mu)$,  $u$ is called a $\kappa$-\emph{harmonic extension} of $\varphi$ provided that $u\in \cH_F^\kappa$ and $u|_F=\varphi$.   Note that the condition
\begin{equation}\label{eq:DEF41}
	\sF^{\kappa,G}_\re\cap \cH^\kappa_F=\{0\}
\end{equation}
 implies that the $\kappa$-harmonic extension,  denoted by $\bH^\kappa_F \varphi$, of $\varphi$ is unique if exists. 
As an analogue of Definition~\ref{DEF23},  we introduce the following. 

\begin{definition}\label{DEF41}
Assume \eqref{eq:DEF41}.  The following operator 
\[
\begin{aligned}
	\cD(\sN_\kappa)=&\big\{\varphi\in L^2(F,\mu): \exists\,u\in \cH^\kappa_F\text{ and }f\in L^2(F,\mu)\text{ such that }u|_F=\varphi, \\
	&\qquad\qquad \sE^\kappa(u,v)=\int_{F} f v|_F d\mu \text{ for any }v\in \sF^\kappa_\re\text{ with }v|_F\in L^2(F,\mu)\big\},   \\
	\sN_\kappa \varphi=&f,\quad \varphi\in \cD(\sN_\kappa)
\end{aligned}
\]
is called the DN operator for $(\sE^\kappa,\sF^\kappa)$ on $L^2(F,\mu)$.  
\end{definition}

In most cases we will impose a stronger assumption:
\begin{equation}\label{eq:42-3}
\sF_\re^\kappa=\sF^{\kappa,G}_\re\oplus \cH^\kappa_F,
\end{equation}
i.e.  for any $u\in \sF^\kappa_\re$,  there exists a unique pair $(u_1,u_2)\in \sF^{\kappa,G}_\re\times \cH^\kappa_F$ such that $u=u_1+u_2$.  
See Appendix~\ref{APP2} for some remarks on this assumption. Meanwhile set 
\begin{equation}\label{eq:45}
	\check{\sF}^\kappa:=\sF^\kappa_\re|_F\cap L^2(F,\mu)=\{u|_F\in L^2(F,\mu): u\in \sF^\kappa_\re\},
\end{equation}
and because of the first assertion in the following lemma,  define 
\begin{equation}\label{eq:46}
	\check{\sE}^\kappa(\varphi, \phi):=\sE^\kappa(\bH^\kappa_F \varphi, \bH^\kappa_F \phi),\quad \varphi,\phi\in\check{\sF}^\kappa.
	\end{equation}
We call $(\check{\sE}^\kappa,\check{\sF}^\kappa)$ the \emph{trace form of $(\sE^\kappa,\sF^\kappa)$ on $L^2(F,\mu)$.} 

\begin{lemma}\label{LM43}
Assume \eqref{eq:42-3}.  The following hold:
\begin{itemize}
\item[(1)] For any $\varphi\in \check{\sF}^\kappa$,  the $\kappa$-harmonic extension of $\varphi$ exists uniquely. 
\item[(2)] $\varphi\in \cD(\sN_\kappa)$ with $f=\sN_\kappa \varphi$,  if and only if $\varphi \in \check{\sF}^\kappa,  f\in L^2(F,\mu)$ and 
\begin{equation}\label{eq:42-2}
	\check{\sE}^\kappa(\varphi, \phi)=\int_F f\phi d\mu,\quad \forall \phi\in \check{\sF}^\kappa.  
\end{equation}
\end{itemize}
\end{lemma}
\begin{proof}
The first assertion is obvious.  We only prove the second one.  Take $\varphi\in \cD(\sN_\kappa)$ with $f=\sN_\kappa \varphi$.  Then $u$ appearing in Definition~\ref{DEF41} is the $\kappa$-harmonic extension of $\varphi$,  and particularly $\varphi\in \check{\sF}^\kappa$.  In addition, for any $\phi\in \check{\sF}^\kappa$,  it holds that $\bH^\kappa_F \phi \in \sF^\kappa_\re$ and $\bH^\kappa_F \phi|_F=\phi\in L^2(F,\mu)$.  Hence
\[
	\check{\sE}^\kappa(\varphi, \phi)=\sE^\kappa(u,\bH^\kappa_F \phi)=\int_F f\phi d\mu.  
\]
To the contrary,  take $\varphi\in \check{\sF}^\kappa$ satisfying \eqref{eq:42-2} for some $f\in L^2(F,\mu)$.  Then $u:=\bH^\kappa_F \varphi \in \cH^\kappa_F$ and $u|_F=\varphi$.   For any $v\in \sF^\kappa_\re$ with $\phi:=v|_F\in L^2(F,\mu)$,  \eqref{eq:42-2} yields that
\[
	\int_F f\phi d\mu=\check{\sE}^\kappa(\varphi, \phi)=\sE^\kappa(u, \bH^\kappa_F \phi)=\sE^\kappa(u,v),
\]
because $v-\bH^\kappa_F \phi \in \sF^{\kappa,G}_\re$.  
That completes the proof.  
\end{proof}

%Now assume that $\sF_\re^\kappa=\sF^{\kappa,G}_\re\oplus \cH^\kappa_F$,  which particularly implies 

\subsection{Self-adjointness of DN operators}\label{SEC41}

It is of course interesting to ask if $\sN_\kappa$ is self-adjoint on $L^2(F,\mu)$.  Lemma~\ref{LM43} tells us that if $(\check{\sE}^\kappa,\check{\sF}^\kappa)$ is a lower bounded closed form,  then $\sN_\kappa$ is lower semi-bounded and self-adjoint.  

\begin{definition}\label{DEF34}
Let $\kappa=\kappa^+-\kappa^-, \mu$ and $F$ be as above.  Then $\kappa^-$ is called \emph{$\sE^{\kappa^+}$-form bounded on trace},  if there exist some constants $0<\delta_0<1$ and $C_{\delta_0}>0$ such that
\begin{equation}\label{eq:47}
	\int_E u^2d\kappa^-\leq \delta_0\cdot \sE^{\kappa^+}(u,u)+C_{\delta_0}\cdot  \int_F (u|_F)^2 d\mu,\quad \forall u\in \cH^\kappa_F,
\end{equation}
where $\sE^{\kappa^+}(u,u)=\sE(u,u)+\int_E u^2d\kappa^+$. 
\end{definition}

A useful sufficient condition for the self-adjointness of $\sN_\kappa$ is presented in the following.

\begin{theorem}\label{THM44}
Assume \eqref{eq:42-3} and that $\kappa^-$ is $\sE^{\kappa^+}$-form bounded on trace.  Then $\sN_\kappa$ is lower semi-bounded and self-adjoint on $L^2(F,\mu)$.  
\end{theorem}
\begin{proof}
%If $\kappa^-=0$,  then the conclusion trivially holds.  We only treat the case $\kappa^-\neq 0$.  
It suffices to prove that $(\check{\sE}^\kappa,\check{\sF}^\kappa)$ is a lower bounded closed form on $L^2(F,\mu)$.  For $\varphi\in \check{\sF}^\kappa$,  set $u:=\bH^\kappa_F \varphi$.  Taking $\varepsilon>0$ such that $(1+\varepsilon){\delta_0}<1$,  $\alpha_0:=(1+\varepsilon)C_{\delta_0}$ and using \eqref{eq:47},  we get 
\begin{equation}\label{eq:47-1}
\begin{aligned}
	\check{\sE}^\kappa(\varphi, \varphi)&+\alpha_0 \int_F \varphi^2d\mu \\ &=\sE^{\kappa^+}(u,u)+\varepsilon \int_E u^2 d\kappa^--(1+\varepsilon)\int_E u^2 d\kappa^- +\alpha_0 \int_F (u|_F)^2 d\mu \\
	&\geq\tilde{{\delta}}_0\cdot \sE^{\kappa^+}(u,u)+\varepsilon \int_E u^2 d\kappa^-\geq 0,
\end{aligned}\end{equation}
where $\tilde{{\delta}}_0=1-(1+\varepsilon){\delta_0}>0$.  
Hence $(\check{\sE}^\kappa_{\alpha_0}, \check{\sF}^\kappa)$ is a non-negative symmetric quadratic form.  

  Next we show that $\check{\sF}^\kappa$ is a Hilbert space under the inner product $\check{\sE}^\kappa_\alpha$ for $\alpha>\alpha_0$.  To do this, take an $\check{\sE}^\kappa_\alpha$-Cauchy sequence $\{\varphi_n\}\subset \check{\sF^\kappa}$.  Set $u_n:=\bH^\kappa_F \varphi_n$.  It follows from \eqref{eq:47-1} that $\{u_n\}$ forms an $\sE^{|\kappa|}$-Cauchy sequence in $\sF^\kappa_\re$.  Note that $\sF^\kappa_\re$ is a Hilbert space under the inner product $\sE^{|\kappa|}$.  Thus there exists $u\in \sF^\kappa_\re$ such that $\sE^{|\kappa|}(u_n-u,u_n-u)\rightarrow 0$.  Taking a subsequence if necessary,  we have that $u_n$ converges to $u$,  $\sE^{|\kappa|}$-q.e.  as well as $\sE$-q.e.   To claim $u\in \cH^\kappa_F$,  note that for any $v\in \sF^{\kappa, G}_\re$,  
  \begin{equation}\label{eq:48}
  	|\sE^\kappa(u_n-u,v)|\lesssim \sE^{|\kappa|}(u_n-u,u_n-u)^{1/2}\cdot \sE^{|\kappa|}(v,v)^{1/2}\rightarrow 0. 
  \end{equation}
 Consequently $\sE^\kappa(u,v)=\lim_{n\rightarrow \infty}\sE^\kappa(u_n,v)=0$.  This yields $u\in \cH^\kappa_F$.  
  On the other hand,  $\varphi_n=u_n|_F$ is Cauchy in $L^2(F,\mu)$.  Hence $\varphi_n\rightarrow \varphi$ in $L^2(F,\mu)$ for some $\varphi\in L^2(F,\mu)$.  Taking a subsequence if necessary we get that $\varphi_n\rightarrow \varphi$,  $\mu$-a.e. on $F$.  Since $\text{qsupp}[\mu]=F$,  we obtain that $u|_F=\varphi$ and particularly,  $\varphi \in \check{\sF}^\kappa$.  In addition, 
  \[
  	|\check{\sE}^\kappa(\varphi_n-\varphi,\varphi_n-\varphi)|=|\sE^\kappa(u_n-u,u_n-u)|\leq \sE^{|\kappa|}(u_n-u,u_n-u)\rightarrow 0.  
  	  \]
 Therefore $\check{\sE}^\kappa_\alpha(\varphi_n-\varphi,\varphi_n-\varphi)\rightarrow 0$.  That completes the proof. 
\end{proof}

Now we turn to give some remarks on the condition \eqref{eq:47}.  Note that $\sF^\kappa_\re$ is a Hilbert space under the inner product $\sE^{|\kappa|}$,  and on account of \eqref{eq:48},  $\cH^\kappa_F$ is a closed subspace of $\sF^\kappa_\re$.  In other words,  $\cH^\kappa_F$ is also a Hilbert space under the inner product $\sE^{|\kappa|}$.  The argument in the following lemma is based on the abstract Ehrling's lemma; see,  e.g.,  \cite[Chapter I,  Theorem~7.3]{W87}.  

\begin{lemma}\label{LM45}
Assume \eqref{eq:42-3}.  If $(\cH^\kappa_F, \sE^{|\kappa|})$ is compactly embedded in $L^2(E,\kappa^-)$,  i.e.  any $\{u_n\in \cH^\kappa_F:n\geq 1\}$ with $\sup_{n\geq 1}\sE^{|\kappa|}(u_n,u_n)<\infty$ forms a relatively compact sequence in $L^2(E,\kappa^-)$,  then for any $\delta>0$,  there exists $C_\delta>0$ such that
\begin{equation}\label{eq:49-2}
	\int_E u^2d\kappa^-\leq \delta\cdot \sE^{|\kappa|}(u,u)+C_\delta\cdot  \int_F (u|_F)^2 d\mu,\quad \forall u\in \cH^\kappa_F.  
\end{equation}
Particularly,  if $(\sF^\kappa_\re, \sE^{|\kappa|})$ is compactly embedded in $L^2(E,\kappa^-)$,  then for any $\delta>0$,  there exists $C_\delta>0$ such that \eqref{eq:49-2} holds. 
\end{lemma}
\begin{proof}
Argue by contradiction.  Suppose that for some $\delta>0$,  and any $n\in \mathbb{N}$,  there exists $u_n\in \cH^\kappa_F$ such that $\sE^{|\kappa|}(u_n,u_n)=1$ and
\begin{equation}\label{eq:4110}
	\int_E u_n^2 d\kappa^-> \delta + n \int_F (u_n|_F)^2 d\mu. 
\end{equation}
Note that $\sup_n \int_E u^2_nd\kappa^-\leq  \sup_n \sE^{|\kappa|}(u_n,u_n)=1$ leading to $\int_F (u_n|_F)^2d\mu\rightarrow 0$.  Using the relative compactness of $\{u_n\}$ in $L^2(E,\kappa^-)$,  we may (and do) assume that $u_n$ converges to some $v\in L^2(E,\kappa^-)$ both strongly in $L^2(E,\kappa^-)$  and $\kappa^-$-a.e.  
Taking a subsequence of $\{u_n\}$ if necessary,  we get that $f_N:=\frac{1}{N}\sum_{n=1}^N u_n$ converges to some $u$ strongly in $\cH^\kappa_F$ under the norm $\|\cdot\|_{\sE^{|\kappa|}}$ as $N\rightarrow \infty$,  and $f_N|_F:=\frac{1}{N}\sum_{n=1}^N u_n|_F$ converges to $0$ strongly in $L^2(F,\mu)$.  Since $\sF^\kappa_\re$ is the extended Dirichlet space of the Dirichlet form $(\sE^{|\kappa|}, \sF^\kappa)$,  it follows that a subsequence of $\{f_N\}$,  still denoted  by $\{f_N\}$,  converges to $u$,  $\sE^{|\kappa|}$-q.e. as well as $\sE$-q.e.  Since $f_N|_F\rightarrow 0$ in $L^2(F,\mu)$, taking a subsequence if necessary,  we may (and do) assume that $f_N|_F$ converges to $0$,  $\mu$-a.e.  Since $\text{qsupp}[\mu]=F$ and $\kappa^-$ charges no $\sE$-polar sets,  we obtain that $u|_F=0$,  $\sE$-q.e.,  and $u=v$,  $\kappa^-$-a.e.  As a result,  $u\in \cH^\kappa_F$ implies that $u=0$.  Particularly,  $\int_E u^2_n d\kappa^-\rightarrow 0$ leading to a contraction of \eqref{eq:4110}.  That completes the proof.   
\end{proof}

The condition \eqref{eq:49-2} is stronger than \eqref{eq:47}.  In fact,  taking $\delta<1/2$ in \eqref{eq:49-2} and letting $\delta_0:=\delta/(1-\delta),  C_{\delta_0}:=C_\delta/(1-\delta)$,  we arrive at \eqref{eq:47}.  Hence the following corollary holds.

\begin{corollary}\label{COR46}
Assume \eqref{eq:42-3} and that $(\sF^\kappa_\re, \sE^{|\kappa|})$ is compactly embedded in $L^2(E,\kappa^-)$.  Then $\sN_\kappa$ is lower semi-bounded and self-adjoint on $L^2(F,\mu)$.  
\end{corollary}

A special case of great interest is $\kappa^-=V^-\cdot m$ with $V^-\in L^\infty(E,m)$.  In this case $(\sE^\kappa,\sF^\kappa)$ is clearly a lower bounded closed form on $L^2(E,m)$.   %Another corollary below provides a sufficient condition for it.

\begin{corollary}\label{COR47}
Assume \eqref{eq:42-3} and that $\sF_\re=\sF$ (endowed with the norm $\|\cdot\|_{\sE_1}$) is compactly embedded in $L^2(E,m)$.  Consider $\kappa^-=V^-\cdot m$ with $V^-\in L^\infty(E,m)$.   Then $\sN_\kappa$ is lower semi-bounded and self-adjoint on $L^2(F,\mu)$.  
\end{corollary} 
\begin{proof}
Mimicking the proof of Lemma~\ref{LM45},  we can obtain that for any $\delta>0$,  there exists $C_\delta>0$ such that 
\[
	\int_E u^2 dm\leq \delta\cdot \sE_1(u,u)+ C_\delta \int_F (u|_F)^2d\mu,\quad \forall u\in \cH^\kappa_F\subset \sF.  
\]
This implies that for $0<\delta<1$, 
\[
	\int_E u^2 dm\leq \frac{\delta}{1-\delta}\cdot \sE(u,u)+ \frac{C_\delta}{1-\delta} \int_F (u|_F)^2d\mu,\quad \forall u\in \cH^\kappa_F.  
\]
Letting $\|V^-\|_\infty:=\|V^-\|_{L^\infty(E,m)}$ and taking $0<\delta<1$ such that $\delta_0:=\|V^-\|_\infty \cdot \delta/(1-\delta)<1$,  we get
\[
	\int_E u^2d\kappa^-=\int_E u^2 V^-dm\leq \delta_0\cdot \sE(u,u)+\frac{C_\delta \|V^-\|_\infty}{1-\delta} \int_F (u|_F)^2d\mu,\quad \forall u\in \cH^\kappa_F.  
\]
Eventually the assertion is concluded by applying Theorem~\ref{THM44}. 
\end{proof}

\subsection{Examples}\label{SEC33}

We give two examples in this short subsection. 

\begin{example}\label{EXA48}
Let $\Omega\subset \bR^d$ be a bounded Lipschitz domain,  and $0<\alpha\leq 2$.  Consider the Dirichlet form
\[
\begin{aligned}
	\sF&=\{ u\in L^2(\Omega): \sE(u,u)<\infty\}, \\
	\sE(u,v)&=\frac{c(d,\alpha)}{2}\iint_{\Omega\times \Omega\setminus \mathsf{d}_\Omega} \frac{(u(x)-u(y))(v(x)-v(y))}{|x-y|^{d+\alpha}}dxdy+\lambda \int_\Omega uvdx,\quad u,v\in \sF, 
\end{aligned}
\]
where the given constant $\lambda>0$ for $0<\alpha<2$ and $\lambda\geq 0$ for $\alpha=2$,  $c(d,\alpha)$ is the constant appearing in,  e.g.,  \cite[(1.4.27)]{FOT11},  and $\mathsf{d}_\Omega$ is the diagonal of $\Omega\times \Omega$.  Then $(\sE,\sF)$ is a regular and irreducible Dirichlet form on $L^2(\bar{\Omega})$. When $\alpha=2$,  $(\sE,\sF)$ is identified with \eqref{eq:Brownian}   associated with the $\lambda$-subprocess of the reflected Brownian motion on $\bar{\Omega}$.  When $0<\alpha<2$,  it is   associated with the $\lambda$-subprocess of the reflected $\alpha$-stable process on $\bar{\Omega}$; see \cite{BBC03}.  
Note that $\sF_\re=\sF=H^{\alpha/2}(\Omega)$,  the Sobolev space of order $\alpha/2$,  and $\|\cdot\|_{\sE_1}$ is equivalent to the norm of $H^{\alpha/2}(\Omega)$.  In view of \cite[Theorem~2.5.5]{SS11},  $\sF=H^{\alpha/2}(\Omega)$ is compactly embedded in $L^2(\Omega)$.  

Let $\mu=\sigma$, the surface measure on $\Gamma=\partial \Omega$.  In the case $\alpha=2$,   $\text{qsupp}[\mu]=\Gamma$ has been mentioned in Lemma~\ref{LM32}~(3).  For the case $0<\alpha<2$,  under a slightly stronger condition that $\Omega$ is $C^{1,1}$ (see, e.g.,  page 67 of \cite{K03}),  the same result can be obtained by virtue of \cite[Theorem~3.14]{K03} and \cite[Lemma~5.2.9]{CF12}.  

Take $\kappa=\kappa^+-\kappa^-\in \bS-\bS$ satisfying \eqref{eq:42-3}.  Let $\sN_\kappa$ be the DN operator for $(\sE^\kappa,\sF^\kappa)$ on  $L^2(\Gamma)$.  On account of Corollary~\ref{COR47},  if $\kappa^-(dx)=V^-(x)dx$ with $V^-\in L^\infty(\Omega)$,  then $\sN_\kappa$ is lower semi-bounded and self-adjoint on $L^2(\Gamma)$.  %Special examples include that $\kappa(dx)=V(x)dx$ with $V\in L^\infty(\Omega)$.  
\end{example} 

\begin{example}
In this example let $d\in \mathbb{N}$ and $0<\alpha\leq 2$.  Consider the Dirichlet form 
\[
\begin{aligned}
	\sF&=\{ u\in L^2(\bR^d): \sE(u,u)<\infty\}, \\
	\sE(u,v)&=\frac{c(d,\alpha)}{2}\iint_{\bR^d \times \bR^d\setminus \mathsf{d}} \frac{(u(x)-u(y))(v(x)-v(y))}{|x-y|^{d+\alpha}}dxdy,\quad u,v\in \sF, 
\end{aligned}
\]
where $c(d,\alpha)$ is the same constant as in Example~\ref{EXA48} and $\mathsf{d}$ is the diagonal of $\bR^d\times \bR^d$.  Clearly $(\sE,\sF)$ is a regular and irreducible Dirichlet form on $L^2(\bR^d)$.  When $0<\alpha<2$,  it is   associated with the isotropic $\alpha$-stable process on $\bR^d$.  When $\alpha=2$,  it is   associated with the Brownian motion on $\bR^d$.  The associated process is denoted by $X$.  

Let $\mu\in \mathring{\bS}$ with $F=\text{qsupp}[\mu]$.  Take $\kappa=\kappa^+-\kappa^-\in \bS-\bS$ satisfying \eqref{eq:42-3}.  Denote by $\sN_\kappa$ the DN operator for $(\sE^\kappa,\sF^\kappa)$ on $L^2(F,\mu)$.  

We first consider the transient case $\alpha<d$.  Denote by $\bS_\infty$ the subfamily of $\bS_K$,  the Kato class,  consisting of all Green-tight smooth measures with respect to $X$;  see, e.g., \cite[Definition~2.1]{T07}.  Note that $$\{V(x)dx: V\in L^\infty(\bR^d)\cap L^1(\bR^d), V\geq 0\}\subset \bS_\infty.$$ Assume that
\[
\kappa\in \bS-\bS_\infty.
\]
By virtue of \cite[Theorem~3.4]{T07},  we obtain that $(\sF^\kappa_\re,\sE^{|\kappa|})$ is compactly embedded in $L^2(\bR^d, \kappa^-)$.  Therefore Corollary~\ref{COR46} yields that $\sN_\kappa$ is lower semi-bounded and self-adjoint on $L^2(F,\mu)$.  

Next we treat the recurrent case $\alpha\geq d=1$.  For $\nu\in \bS$,  denote by $X^\nu$ the subprocess of $X$ killed by the PCAF corresponding to $\nu$.  Let $\bS^\nu_\infty$ be the family of all Green-tight smooth measures with respect to $X^\nu$;  see \cite[(2.8)]{T16}.    Assume that $$\kappa=\kappa^+-\kappa^-\in \bS_K-\bS^{\kappa^+}_\infty.$$  Using \cite[Proposition~6.3]{T16},  we get that $(\sF^\kappa_\re, \sE^{|\kappa|})$ is compactly embedded in $L^2(\bR^d,\kappa^-)$ for the case $\alpha>1$.  If $\alpha=1$,  the same compact embedding holds when certain condition on $\kappa^-$ is added;  see \cite[Remark~6.4]{T16}.  Eventually $\sN_\kappa$ is lower semi-bounded and self-adjoint on $L^2(F,\mu)$ by means of Corollary~\ref{COR46}. 
\end{example}

\subsection{Markov processes $h$-associated with DN operators}\label{SEC34}

Let us turn to figure out the probabilistic counterpart of $\sN_\kappa$.  
Adopt the assumptions of Theorem~\ref{THM44} and assume further that $\kappa^-$ is $\sE^{\kappa^+}$-form bounded.  These mean that both $(\sE^\kappa,\sF^\kappa)$ and $(\check{\sE}^\kappa,\check{\sF}^\kappa)$ are lower bounded closed forms.  %Particularly there exists a constant $\alpha_0\geq 0$ such that $(\check{\sE}^\kappa_{\alpha_0},\check{\sF}^\kappa)$ is a non-negative closed form. 
Set
\[
	\sF^{\kappa,G}:=\{u\in \sF^\kappa: u=0,\sE\text{-q.e. on }F\},\quad \sE^{\kappa,G}(u,v):=\sE^\kappa(u,v),\; u,v\in \sF^{\kappa,G}.  
\]
Then $(\sE^{\kappa,G},\sF^{\kappa,G})$ is a lower bounded closed form on $L^2(G,m|_G)$,  whose generator is denoted by $\sL_{\kappa,G}$. 

\begin{lemma}\label{LM311}
The following are equivalent:
\begin{itemize}
\item[(a)] $-\sL_{\kappa,G}$ is positive in the sense that $(-\sL_{\kappa,G}u,u)_m\geq 0$ for any $u\in \cD(\sL_{\kappa,G})$. 
\item[(b)] $\sE^{\kappa,G}(u,u)\geq 0$ for any $u\in \sF^{\kappa,G}$.
\item[(c)] $\sE^{\kappa,G}(u,u)\geq 0$ for any $u\in \sF^{\kappa,G}_\re$.  
\end{itemize}
\end{lemma}
\begin{proof}
Clearly (c) implies (a).  Suppose (a).  Since $(\sE^{\kappa,G}_{\alpha_0}, \sF^{\kappa,G})$ is a non-negative closed form for some $\alpha_0\geq 0$,  it follows that for $\alpha>\alpha_0$,
\begin{equation}\label{eq:3.11}
	\sE^{\kappa,G}_\alpha (u,u)\geq \alpha (u,u)_m,\quad \forall u\in \cD(\sL_{\kappa, G}).
\end{equation}
Note that $\cD(\sL_{\kappa,G})$ is $\|\cdot\|_{\sE^{\kappa,G}_\alpha}$-dense in $\sF^{\kappa,G}$.  For $u\in \sF^{\kappa,G}$,  one can take a sequence $u_n\in \cD(\sL_{\kappa,G})$ such that $\|u_n-u\|_{\sE^{\kappa,G}_\alpha}\rightarrow 0$ and $(u_n,u_n)_m\rightarrow (u,u)_m$.  Applying \eqref{eq:3.11} to $u_n$ and letting $n\rightarrow 0$,  we get (b).  Finally suppose (b) and we are to derive (c).  Note that $\sF^{\kappa,G}_\re=\sF^{\kappa^+,G}_\re$ is the extended Dirichlet space of the Dirichlet form $(\sE^{\kappa^+,G},\sF^{\kappa^+,G})$,  where $\sE^{\kappa^+,G}(u,v)=\sE(u,v)+\int uv d\kappa^+$ for $u,v\in \sF^{\kappa^+,G}$.  For $u\in \sF^{\kappa,G}_\re=\sF^{\kappa^+,G}_\re$, take its approximating sequence $\{u_n\}\subset \sF^{\kappa^+,G}$, and on account of \cite[Theorem~2.3.4]{CF12},  we may and do assume that $\{u_n\}$ is $\sE^{\kappa^+,G}$-Cauchy and $u_n$ converges to $u$,  $\sE^{\kappa^+,G}$-q.e.  Particularly,  $u_n$ converges to $u$,  $\kappa^-$-a.e.    Looking at (b),  we have
\begin{equation}\label{eq:3.12}
	\sE^{\kappa^+,G}(u_n,u_n)\geq \int u_n^2 d\kappa^-.
\end{equation}
Since $\{u_n\}$ is $\sE^{\kappa^+,G}$-Cauchy, it follows that $\{u_n\}$ is $L^2(E,\kappa^-)$-Cauchy and $u_n$ converges to $u$ in $L^2(E,\kappa^-)$.  Letting $n\uparrow\infty$ in \eqref{eq:3.12},  we arrive at (c).  That completes the proof. 
\end{proof}
\begin{remark}
Either condition in this lemma admits a non-zero $\kappa^-$.  For example,  in Lemma~\ref{LM41},  they are also equivalent to $\lambda>\lambda^\text{D}_1/2$ due the Poincar\'e's inequality. 
\end{remark}

The conception of quasi-regular (non-negative) positivity preserving (symmetric) coercive form is reviewed in Appendix~\ref{APPD}.  Note that $(\check{\sE}^\kappa_{\alpha_0},\check{\sF}^\kappa)$ is a non-negative closed form for some $\alpha_0\geq 0$.  The quasi notions like nest, polar set and quasi-continuous function for $\check\sE^\kappa$ are by definition those for $\check{\sE}^\kappa_{\alpha_0}$.  We say $(\check{\sE}^\kappa,\check{\sF}^\kappa)$ is a quasi-regular positive preserving (symmetric) coercive form if so is $(\check{\sE}^\kappa_{\alpha_0},\check{\sF}^\kappa)$. 
%The following theorem states that there exists a family of Markov processes associated with $\sN_\kappa$ via the Doob's $h$-transformations. 

\begin{theorem}\label{THM312}
Assume that either of the conditions in Lemma~\ref{LM311} holds.  Then $(\check{\sE}^\kappa,\check{\sF}^\kappa)$ is a quasi-regular positivity preserving (symmetric) coercive form on $L^2(F,\mu)$.  
\end{theorem}
\begin{proof}
%It is easy to verify that the positivity of $-\sL_{\kappa,G}$ amounts to that $\sE^{\kappa,G}(u,u)\geq 0$ for any $u\in \sF^{\kappa,G}$. 
We first show that $(\check{\sE}^\kappa,\check{\sF}^\kappa)$  is positivity preserving.  Take $\varphi\in \check{\sF}^\kappa$, and set $\varphi^+:=\varphi \vee 0,  \varphi^-:=\varphi^+-\varphi$.  Clearly $\varphi^\pm \in \check{\sF}^\kappa$.  Let $u:=\bH^\kappa_F \varphi$ and $u^+:=u\vee 0, u^-:=u^+-u$.  Since $u=\bH^\kappa_F \varphi^+-\bH^\kappa_F\varphi^-$,  it follows that
\[
	u^+-\bH^\kappa_F\varphi^+=u^--\bH^\kappa_F\varphi^-=:u_0\in \sF^\kappa_\re.  
\]
Note that $u_0|_F=0$,  $\mu$-a.e.  and hence $\sE$-q.e.  due to $\text{qsupp}[\mu]=F$. This implies that $u_0\in \sF^{\kappa,G}_\re$.  By the definition of $\check{\sE}^\kappa$, one have that
\[
\begin{aligned}
	\check{\sE}&^\kappa(\varphi^+,\varphi^-)\\ &=\sE^\kappa(\bH^\kappa_F \varphi^+,\bH^\kappa_F\varphi^-)=\sE^\kappa(u^+-u_0,u^--u_0)=\sE(u^+,u^-)-\sE^\kappa(u_0,u_0). 
\end{aligned}\]
By means of (c) in Lemma~\ref{LM311} and the positivity preserving property of $(\sE,\sF)$,  we get that
\[
	\check{\sE}^\kappa(\varphi^+,\varphi^-)\leq \sE(u^+,u^-)\leq 0. 
\]
Hence $(\check{\sE}^\kappa,\check{\sF}^\kappa)$  is positivity preserving.  

Denote by $(\check{\sE}^{|\kappa|},\check{\sF}^{|\kappa|})$ the trace Dirichlet form of $(\sE^{|\kappa|},\sF^{|\kappa|})$ on $L^2(F,\mu)$.  Clearly,   $(\check{\sE}^{|\kappa|},\check{\sF}^{|\kappa|})$ is quasi-regular on $L^2(F,\mu)$.  Note that
\[
	\check{\sF}^{|\kappa|}=\sF^{|\kappa|}_\re|_F\cap L^2(F,\mu)=\sF^\kappa_\re|_F\cap L^2(F,\mu)=\check{\sF}^\kappa. 
\]
For $\varphi\in \check{\sF}^\kappa$,  $\bH^\kappa_F\varphi, \bH^{|\kappa|}_F\varphi \in \sF^\kappa_\re$ and $\bH^\kappa_F\varphi-\bH^{|\kappa|}_F\varphi\in \sF^{\kappa,G}_\re$.  By means of (c) in Lemma~\ref{LM311},  we have that
\[
\begin{aligned}
	\check{\sE}^{|\kappa|}(\varphi,\varphi)&=\sE^{|\kappa|}\left(\bH^{|\kappa|}_F\varphi,\bH^{|\kappa|}_F\varphi\right)\geq \sE^\kappa\left(\bH^{|\kappa|}_F\varphi,\bH^{|\kappa|}_F\varphi \right) \\
	&=\sE^\kappa\left(\bH^\kappa_F\varphi+(\bH^{|\kappa|}_F\varphi-\bH^\kappa_F\varphi), \bH^\kappa_F\varphi+(\bH^{|\kappa|}_F\varphi-\bH^\kappa_F\varphi)\right).   \\
\end{aligned}\]
On account of $\bH^\kappa_F\varphi-\bH^{|\kappa|}_F\varphi\in \sF^{\kappa,G}_\re$,  the last term is equal to
\[
\begin{aligned}
	&=\sE^\kappa(\bH^\kappa_F\varphi, \bH^\kappa_F\varphi)+\sE^\kappa\left(\bH^{|\kappa|}_F\varphi-\bH^\kappa_F\varphi,\bH^{|\kappa|}_F\varphi-\bH^\kappa_F\varphi \right)\\
	&\geq \sE^\kappa(\bH^\kappa_F\varphi, \bH^\kappa_F\varphi)=\check{\sE}^\kappa(\varphi,\varphi).  
\end{aligned}\]
In view of Lemma~\ref{LMD3} we eventually obtain the quasi-regularity of $(\check{\sE}^\kappa,\check{\sF}^\kappa)$.  That completes the proof. 
\end{proof}

Take a constant $\alpha_0\geq 0$ such that $\check{\sE}^\kappa_{\alpha_0}$ is non-negative.  
Denote by $\check{T}^\kappa=(\check{T}^\kappa_t)_{t\geq 0}$ and $(\check{G}^\kappa_\alpha)_{\alpha>\alpha_0}$ the $L^2$-semigroup and $L^2$-resolvent of $(\check{\sE}^\kappa,\check{\sF}^\kappa)$ respectively.  For $\alpha\in \bR$,  $\varphi$ is called \emph{$\alpha$-excessive} if $\varphi\in pL^2(F,\mu)$ and $\re^{-\alpha t}\check{T}^\kappa_t \varphi\leq \varphi$,  $\mu$-a.e. for all $t>0$.  Define
\[
\mathbf{E}^+_\alpha:=\{\varphi \text{ is } \alpha\text{-excessive}: \varphi\in \check{\sF}^\kappa,  \varphi>0, \mu\text{-a.e.}\}.  
\]
Note that $\mathbf{E}^+_\alpha\subset \mathbf{E}^+_{\alpha'}$ in case $\alpha<\alpha'$,  and for $\alpha>\alpha_0$, 
\[
	\{\check{G}^\kappa_\alpha f: f\in L^2(F,\mu),  f>0, \mu\text{-a.e.}\}\subset \mathbf{E}_\alpha^+;
\] 
see \cite[Lemma~3.6]{MR95}.  For $\alpha\in \bR$ and $h\in L^2(F,\mu),  h>0$, $\mu$-a.e.,  set
\[
\begin{aligned}
	&\check\sF^{\kappa,h}:=\{\varphi\in L^2(F,h^2\cdot \mu): \varphi h\in \check\sF^{\kappa}\}, \\
		&\check\sE^{\kappa,h}_\alpha(\varphi,\phi):=\check\sE^{\kappa}_\alpha(\varphi h,\phi h),\quad \varphi,\phi\in \check\sF^{\kappa,h},
\end{aligned}
\]
called the $h$-transform of $(\check{\sE}^\kappa_\alpha,\check{\sF}^\kappa)$.  Write $\check{\sE}^{\kappa,h}:=\check{\sE}^{\kappa,h}_0$.  Clearly,  $(\check{\sE}^{\kappa,h}_\alpha,\check{\sF}^{\kappa,h})$ is a lower bounded closed form on $L^2(F,h^2\cdot \mu)$,  whose $L^2$-semigroup and $L^2$-resolvent are
\[
	\check{T}^{\kappa,h,\alpha}_t\varphi=\frac{\re^{-\alpha t}\check{T}^\kappa_t\left(\varphi h\right)}{h},\quad \check{G}^{\kappa,h,\alpha}_{\alpha'}\varphi=\frac{1}{h}\int_0^\infty \re^{-(\alpha+\alpha') t} \check{T}^{\kappa}_t (\varphi h)dt
\]
for $t\geq 0$ and $\alpha'>(\alpha_0-\alpha)\vee 0$.  Its $L^2$-generator is 
\[
\begin{aligned}
	&\cD\left(\check{\sL}^{\kappa,h}_\alpha\right)=\{\varphi \in L^2(F,h^2\cdot \mu): \varphi h\in \cD(\sN_\kappa)\},  \\
	&\check{\sL}^{\kappa,h}_\alpha \varphi =-\frac{\sN_\kappa (\varphi h)}{h}-\alpha \varphi,\quad \varphi \in \cD\left(\check{\sL}^{\kappa,h}_\alpha \right).
\end{aligned}\]
When $h\in \mathbf{E}^+_\alpha$,  $\check{T}^{\kappa, h,\alpha}_t$ is Markovian,  i.e.  $0\leq \check{T}^{\kappa, h,\alpha}_t \varphi\leq 1$ for $\varphi\in L^2(F,h^2\cdot \mu)$ with $0\leq \varphi\leq 1$.  Particularly,  the resolvent can be extended to a Markovian one with parameter $\alpha'>0$ on $L^\infty(F,\mu)$,  i.e.  
\[
	\check{G}^{\kappa,h,\alpha}_{\alpha'} \varphi=\frac{1}{h}\int_0^\infty \re^{-(\alpha+\alpha') t} \check{T}^{\kappa}_t (\varphi h)dt,\quad \alpha'>0, \varphi \in L^\infty(F,\mu)
\]
forms a Markovian resolvent on $L^\infty(F,\mu)$;  see,  e.g., \cite[page 8]{O13}.  

\begin{corollary}
Adopt the same assumptions in Theorem~\ref{THM312}.  For any $\alpha\in \bR$ and $h\in \mathbf{E}^+_\alpha$,  the $h$-transform $(\check{\sE}^{\kappa,h}_\alpha,\check{\sF}^{\kappa,h})$ is a quasi-regular lower bounded symmetric Dirichlet form on $L^2(F,h^2\cdot \mu)$.  Furthermore,  there exists an $h^2\cdot \mu$-symmetric Markov process $\check{X}^{\kappa,  h, \alpha}$ such that for $\alpha'>0$ and $\varphi\in L^\infty(F,\mu)$,
\[
	\check{R}^{\kappa,h,\alpha}_{\alpha'}\varphi(\cdot):=\mathbf{E}_\cdot \left[\int_0^\infty \re^{-\alpha' t} \varphi(\check{X}^{\kappa, h,\alpha}_t)dt\right]
\]
is an $\check{\sE}^{\kappa,h}_\alpha$-quasi-continuous $\mu$-version of $\check{G}^{\kappa, h,\alpha}_{\alpha'} \varphi$.  
\end{corollary}
\begin{proof}
Clearly $(\check{\sE}^{\kappa,h}_\alpha,\check{\sF}^{\kappa,h})$ is a lower bounded symmetric Dirichlet form.  Note that $h\in \mathbf{E}^+_\alpha \subset \mathbf{E}^+_{\alpha \vee (\alpha_0+1)}$.  On account of Theorem~\ref{THMD2},  we obtain the quasi-regularity of $(\check{\sE}^{\kappa,h}_\alpha,\check{\sF}^{\kappa,h})$.  Using \cite[Theorem~3.3.4]{O13} and quasi-homeomorphism, one can conclude the existence of $\check{X}^{\kappa,h,\alpha}$ by a standard argument.  That completes the proof.
\end{proof}

The Markov process $\check{X}^{\kappa,  h, \alpha}$ is called $(\alpha, h)$-associated with $\sN_\kappa$.  When $\alpha=0$,  it is called $h$-associated with $\sN_\kappa$. 

\section{Revisiting classical DN operators}\label{Revisiting}

Let us turn to revisit the classical DN operators $D_\lambda$ defined as Definition~\ref{DEF31} but with possibly negative parameter $\lambda$.  For simplicity assume that $\Omega\subset \bR^d$,  $d\geq 2$,  is a bounded domain with smooth boundary (though most results in this section hold for less regular boundary like a Lipschitz one).  Let $\Delta^\text{D}$ be the Dirichlet Laplacian,  i.e.  $\Delta^\text{D} u=\Delta u$ with domain $\cD(\Delta^\text{D}):=\{u\in H^1_0(\Omega): \Delta u\in L^2(\Omega)\}$.  It is known that the spectrum $\sigma(\Delta^\text{D})$ consists of a decreasing sequence of eigenvalues:
\[
	\cdots \leq \lambda^\text{D}_3\leq \lambda^\text{D}_2<\lambda^\text{D}_1<0,
\]
where $\lambda^\text{D}_1<0$ is called the first eigenvalue of $\Delta^\text{D}$.  

Throughout this section,  let $(\sE,\sF)=(\frac{1}{2}\mathbf{D}, H^1(\Omega))$ be the Dirichlet form on $L^2(\bar{\Omega})$ associated with the reflected Brownian motion $X$ on $\bar{\Omega}$.  Denote by $(\check{\sE},\check{\sF})$ the trace Dirichlet form of $(\sE,\sF)$ on $L^2(\Gamma)$ and by $\check{X}$ its   associated Markov process on $\Gamma$.  

\subsection{Form representation for DN operators}

The DN operator $D_\lambda$ for $\lambda\in \bR$ is defined as the same way as Definition~\ref{DEF31}.  But we should point out that it is well defined,  i.e.  not multi-valued,  if and only if $2\lambda\notin \{\lambda^\text{D}_n: n\geq 1\}$;  see \cite[Proposition~4.11]{AE14}.  Another way to obtain $D_\lambda$ is to use the method formulated in \S\ref{SEC3}.  Take $\mu=\sigma$,  the surface measure on $\Gamma=\partial \Omega$ with $\text{qsupp}[\sigma]=\Gamma$,  and $\kappa=\lambda \cdot m\in \bS-\bS$.  

\begin{lemma}\label{LM41}
Let $(\sE,\sF)=(\frac{1}{2}\mathbf{D}, H^1(\Omega))$,  $\mu=\sigma$ and $\kappa=\lambda\cdot m$ for $\lambda\in \bR$.  Then the following are equivalent:
\begin{itemize}
\item[(a)] \eqref{eq:DEF41} holds;
\item[(b)] \eqref{eq:42-3} holds;
\item[(c)] $2\lambda\notin \{\lambda^\text{D}_n: n\geq 1\}$. 
\end{itemize}
In addition,  for $\lambda\neq \lambda^\text{D}_n/2$,  the following are true:
\begin{itemize}
\item[(1)] $D_\lambda$ is identified with the DN operator for $(\sE^\kappa, \sF^\kappa)$ on $L^2(\Gamma)$.  
\item[(2)] $D_\lambda$ is lower semi-bounded and self-adjoint on $L^2(\Gamma)$.  
\end{itemize} 
\end{lemma}
\begin{proof}
Clearly (b) implies (a),  and (c) implies (b) by means of Lemma~\ref{LMB1}; see also \cite[Lemma~2.2]{AM12}.  If $\lambda=\lambda^\text{D}_n/2$ for some $n$,  then $D_\lambda$ is not well defined because of \cite[Proposition~4.11]{AE14}.  This means that there exist $u_1, u_2\in \cH^\kappa_\Gamma$ such that $u_1|_\Gamma=u_2|_\Gamma$ but $u_1\neq u_2$.  Consequently $0\neq u:=u_1-u_2\in \cH^\kappa_\Gamma\cap \sF^{\kappa,  \Omega}_\re$, and \eqref{eq:DEF41} fails.  

For $\lambda\neq \lambda^\text{D}_n/2$,  it is easy to verify the first assertion by using Lemma~\ref{LM32}.  The second assertion is the consequence of Corollary~\ref{COR47}. That completes the proof.  
\end{proof}

From now on we assume that $\lambda\neq \lambda^\text{D}_n/2$.  
As mentioned in Lemma~\ref{LM43},  $-D_\lambda$ is the $L^2$-generator of the trace form $(\check{\sE}^\kappa,\check{\sF}^\kappa)$.  Fix a constant $\alpha_0\geq 0$ such that $\check{\sE}^\kappa_{\alpha_0}$ is non-negative.  The main purpose of this subsection is to formulate the expression of $(\check{\sE}^\kappa,\check{\sF}^\kappa)$.  For emphasis, write
\[
\begin{aligned}
	&(\sE^\lambda,\sF^\lambda):=(\sE^\kappa,\sF^\kappa),  \quad (\check\sE^\lambda,\check\sF^\lambda):=(\check\sE^\kappa,\check\sF^\kappa),\\  &\cH^\lambda_\Gamma:=\cH^\kappa_\Gamma,\quad \bH^\lambda_\Gamma \varphi:=\bH^\kappa_\Gamma \varphi, \text{ for }\varphi \in \check{\sF}^\lambda.  
\end{aligned}\]
Write $\bH_\Gamma\varphi:=\bH^0_\Gamma \varphi$.  Clearly $\sF^\lambda=H^1(\Omega)$ and $\check{\sF}^\lambda=H^{1/2}(\Gamma)$.  Set
\begin{equation}\label{eq:4.6}
	\mathscr{U}_\lambda(\varphi\otimes \phi):=\lambda \int_\Omega \bH^\lambda_\Gamma \varphi (x) \bH_\Gamma \phi(x)dx,\quad \varphi, \phi\in H^{1/2}(\Gamma).  
\end{equation}
Recall that $(\sE^\Omega,\sF^\Omega)$ is the part Dirichlet form of $(\sE,\sF)$ on $L^2(\Omega)$ associated with the killed Brownian motion $X^\Omega$ on $\Omega$.  Denote its $L^2$-generator by $\sL_\Omega$ with domain $\cD(\sL_\Omega)$ and set $G^\Omega:=(-\sL_\Omega)^{-1}$ to be its $0$-resolvent operator,  i.e. the Green operator on $L^2(\Omega)$.  Let $p^\Omega_t(x,y)$ be the transition density of $X^\Omega$,  i.e.  $\mathbf{E}_x\left[f(X^\Omega_t) \right]=\int_\Omega p^\Omega_t(x,y)f(y)dy$. 

%Denote by $X^0$ the part process of $X$ on $\Omega$,  i.e the absorbing Brownian motion on $\Omega$.  Let $R^0$ be the $0$-resolvent of $X^0$,  i.e.  $R^0 f(x):=\mathbf{E}_x\left[ \int_0^\infty f(X^0_t)dt\right]$ for $f\in b\mathcal{B}(\Omega)$, and $

% for $\varphi \in b\mathcal{B}_+(\Gamma)$,  
%,\quad \bH_\Gamma \varphi(x):=\mathbf{E}_x\left[\varphi(X_\tau);\tau<\infty \right],
%\]
%  Write

\begin{lemma}\label{LM42}
Assume that $\lambda\neq \lambda^\text{D}_n/2$.  Then the following hold:
\begin{itemize}
\item[(1)] For $\varphi\in H^{1/2}(\Gamma)$,  
\begin{equation}\label{eq:4.2}
	\bH_\Gamma \varphi-\bH^\lambda_\Gamma \varphi-\lambda G^\Omega (\bH^\lambda_\Gamma \varphi)=0.
\end{equation}
\item[(2)] Assume $\lambda>\lambda^\text{D}_1/2$.   For $\varphi \in b\mathcal{B}(\Gamma)\cap H^{1/2}(\Gamma)$,  $\bH^\lambda_\Gamma\varphi$ admits the following probabilistic representation:
\begin{equation}\label{eq:4.4}
\bH^\lambda_\Gamma \varphi(x)=\mathbf{E}_x\left[\re^{-\lambda \tau}\varphi(X_\tau);\tau<\infty \right],
\end{equation}
where $\tau:=\inf\{t>0: X_t\notin \Omega\}$ and $\varphi$ on the right hand side takes its $\check{\sE}$-quasi-continuous $\sigma$-version.  Furthermore,  for $\varphi, \phi\in b\mathcal{B}(\Gamma)\cap H^{1/2}(\Gamma)$,  
\begin{equation}\label{eq:4.3}
\mathscr{U}_\lambda(\varphi\otimes \phi)=\int_{\Gamma \times \Gamma} \varphi(\xi)\phi(\eta)U_\lambda(\xi,\eta)\sigma(d\xi)\sigma(d\eta),
\end{equation}
where
\begin{equation}\label{eq:4.5-2}
	U_\lambda(\xi,\eta)=\frac{1}{4}\int_0^\infty \left(1-\re^{-\lambda t}\right) \frac{\partial^2 p^\Omega_t(\xi,\eta)}{\partial \bn_\xi \partial \bn_\eta}dt
\end{equation}
is symmetric and $\bn_\xi$ denotes the inward normal vector at $\xi$.  
\end{itemize}
\end{lemma}
\begin{proof}
\begin{itemize}
\item[(1)]  Note that $\bH_\Gamma \varphi-\bH^\lambda_\Gamma \varphi \in H^1(\Omega)$ and $(\bH_\Gamma \varphi-\bH^\lambda_\Gamma \varphi)|_\Gamma=0$.  This implies $\bH_\Gamma \varphi-\bH^\lambda_\Gamma \varphi\in H^1_0(\Omega)$.  Take $v\in H^1_0(\Omega)$.  Then
\[
	\sE^\Omega(\bH_\Gamma \varphi-\bH^\lambda_\Gamma \varphi, v)=\sE(\bH_\Gamma\varphi, v)-\sE(\bH^\lambda_\Gamma\varphi,v)=\lambda \int_\Omega \bH^\lambda_\Gamma \varphi vdx.  
\]
Thus $\bH_\Gamma\varphi-\bH^\lambda_\Gamma \varphi\in \cD(\sL_\Omega)$ and $\sL_\Omega(\bH_\Gamma\varphi-\bH^\lambda_\Gamma \varphi)=-\lambda\bH^\lambda_\Gamma \varphi$.  As a result,  \eqref{eq:4.2} holds true. 
\item[(2)]  Denote the right hand side of \eqref{eq:4.4} by $\bH^\lambda_\tau \varphi$ and write $\bH_\tau \varphi:=\mathbf{E}_x\left[\varphi(X_\tau);\tau<\infty\right]$.  Then $\bH_\tau \varphi=\bH_\Gamma \varphi$.  Let $R^\Omega$ be the (probabilistic) $0$-resolvent of $X^\Omega$,  i.e.  
\[
	R^\Omega f(x):=\mathbf{E}_x\left[ \int_0^\infty f(X^\Omega_t)dt\right],\quad f\in b\mathcal{B}(\Omega). 
\]
We assert that when $\lambda>\lambda^\text{D}_1/2$,  for $\varphi\in b\mathcal{B}(\Gamma)\cap H^{1/2}(\Gamma)$, 
\begin{equation}\label{eq:4.5}
	\bH_\tau \varphi(x)-\bH^\lambda_\tau \varphi(x)-\lambda R^\Omega\left(\bH^\lambda_\tau \varphi\right)(x)=0,\quad x\in \Omega. 
\end{equation}
Note that $\lambda>\lambda_1^\text{D}/2$ amounts to that $(\Omega,  -\lambda)$ is gaugeable,  i.e.  $x\mapsto \mathbf{E}_x\left[\re^{-\lambda \tau} \right]$ is finite and bounded;  see \cite[Theorem~4.19]{CZ95}.  Particularly $\bH^\lambda_\tau \varphi$ is finite and bounded.  
Since $X^\Omega$ is transient,  it follows that $R^\Omega(\bH^\lambda_\tau \varphi)<\infty$.  In addition,  for $\lambda\neq 0$,  
\[
\begin{aligned}
	\frac{1}{\lambda}\left(\bH_\tau \varphi(x)-\bH^\lambda_\tau \varphi(x)\right)&=\mathbf{E}_x \left[\int_0^\tau \re^{-\lambda t} \varphi(X_\tau)dt \right]  \\
	&=\mathbf{E}_x \left[\int_0^\infty 1_{\{t<\tau\}} \re^{-\lambda \tau\circ \theta_t} \varphi(X_\tau\circ \theta_t)dt  \right]\\
	&=\mathbf{E}_x\left[\int_0^\infty 1_{\{t<\tau\}}\mathbf{E}_{X_t}\left[\re^{-\lambda \tau} \varphi(X_\tau) \right] dt \right] \\
	&=R^\Omega(\bH^\lambda_\tau \varphi)(x),
\end{aligned}
\]
where $\theta_t$ is the translation operator of $X$,  and the third identity is due to the strong Markov property of $X$.  Hence we arrive at \eqref{eq:4.5}. 

Since $\sF^{\kappa,\Omega}_\re=H^1_0(\Omega)$,  one can get that $\bH^\lambda_\Gamma \varphi$ is the unique weak solution in $H^1(\Omega)$ to the following equation:
\begin{equation}\label{eq:4.1}
	\frac{1}{2}\Delta u- \lambda u=0\text{ in }\Omega
\end{equation}
with the boundary condition $u|_\Gamma=\varphi$.  
On the other hand,  on account of \cite[Theorem~4.7]{CZ95},  $\bH_\tau^\lambda \varphi$ is also a weak solution to \eqref{eq:4.1}.  In addition,  for $\check{\sE}$-q.e.  $x\in \Gamma$,
\[
	\bH^\lambda_\tau \varphi(x)=\mathbf{E}_x\left[\varphi(X_\tau)\right]=\bH_\tau \varphi(x)=\varphi(x).  
\]
It suffices to show $\bH^\lambda_\tau \varphi\in H^1(\Omega)$,  so that $\bH^\lambda_\Gamma \varphi=\bH^\lambda_\tau \varphi$.  In fact,  write $\varphi^+:=\varphi \vee 0$ and $\varphi^-:=\varphi^+-\varphi$.  Then \eqref{eq:4.5} implies that
\[
\int_\Omega \bH^\lambda_\tau \varphi^\pm(x) R^\Omega(\bH^\lambda_\tau \varphi^\pm)(x)dx<\infty.  
\]	
In view of \cite[Theorem~1.5.4]{FOT11},  $R^\Omega(\bH^\lambda_\tau \varphi^\pm)\in \sF^\Omega_\re=H^1_0(\Omega)$. Hence 
\[
	R^\Omega(\bH^\lambda_\tau \varphi)=R^\Omega(\bH^\lambda_\tau \varphi^+)-R^\Omega(\bH^\lambda_\tau \varphi^-)\in H^1_0(\Omega).  
\]
Since $\bH_\tau\varphi \in \sF_\re=H^1(\Omega)$,  it follows from \eqref{eq:4.2} that $\bH^\lambda_\tau \varphi\in H^1(\Omega)$.  Eventually we obtain that $\bH^\lambda_\Gamma \varphi=\bH^\lambda_\tau\varphi$.  

 As formulated in \cite[(5.8.1)]{CF12},  
\begin{equation}\label{eq:4.7-2}
	\mathbf{P}_x\left(\tau\in ds,  X_\tau\in d\xi \right)=\frac{1}{2}\frac{\partial p^\Omega_s(x,\xi)}{\partial {\bn_\xi}}ds\sigma(d\xi),\quad x\in \Omega, \xi\in \Gamma.  
\end{equation}
Then a straightforward computation yields \eqref{eq:4.3}.  Note that  \eqref{eq:4.2} and \eqref{eq:4.4} imply
\[
\begin{aligned}
	\mathscr U_\lambda(\varphi\otimes \phi)&=\lambda \int_\Omega \bH^\lambda_\tau \varphi(x) \bH^\lambda_\tau \phi(x)dx+\lambda^2 \int_\Omega \bH^\lambda_\tau \varphi(x)R^\Omega(\bH^\lambda_\tau \phi)(x)dx \\
	&=\lambda \int_\Omega \bH^\lambda_\tau \varphi(x) \bH^\lambda_\tau \phi(x)dx+\lambda^2 \int_\Omega R^\Omega(\bH^\lambda_\tau \varphi)(x)\bH^\lambda_\tau \phi(x)dx \\
	&=\mathscr{U}_\lambda(\phi\otimes \varphi).  
\end{aligned}\]
Therefore $U_\lambda(\xi,\eta)=U_\lambda(\eta,\xi)$.  
\end{itemize}
That completes the proof.  
\end{proof}
\begin{remark}\label{RM43}
\eqref{eq:4.2} reads as $\bH^\lambda_\Gamma \varphi=-\lambda G^\Omega(\bH^\lambda_\Gamma \varphi)+ \bH_\Gamma \varphi$,  where $-\lambda G^\Omega(\bH^\lambda_\Gamma \varphi)\in \sF^{\Omega}_\re=H^1_0(\Omega)$ and $\bH_\Gamma \varphi\in \cH_\Gamma$.  This is actually the orthogonal decomposition of $\bH^\lambda_\Gamma \varphi\in H^1(\Omega)$ in $H^1_0(\Omega)\oplus \cH_\Gamma$.  
%We should emphasis that $\lambda>\lambda^\text{D}_1/2$ is equivalent to the gaugeability of $(\Omega,  -\lambda)$,  under which the solution to the Dirichlet problem \eqref{eq:4.1}  admits a probabilistic representation.  Particularly, \eqref{eq:4.4} holds true.  

Note that $U_\lambda(\xi,\eta)\geq 0$ when $\lambda\geq 0$,  while $U_\lambda(\xi,\eta)< 0$ when $\lambda^\text{D}_1/2<\lambda<0$.  For $\lambda>\lambda^\text{D}_1/2$,  $U_\lambda(\xi,\eta)\sigma(\xi)\sigma(\eta)$ forms a finite symmetric measure on $\Gamma\times \Gamma$,  because the gaugeability of $(\Omega, -\lambda)$ implies
\[
	\left|\iint_{\Gamma\times \Gamma}U_\lambda(\xi,\eta)\sigma(\xi)\sigma(\eta)\right|=\left| \mathscr{U}_\lambda(1\otimes 1)\right|\leq |\lambda| \int_\Omega \mathbf{E}_x\left[\re^{-\lambda \tau} \right]dx<\infty.  
\]
In addition,
\[
	\lambda\mapsto U_\lambda(\xi,\eta)
\]
is increasing and the limit $U(\xi,\eta):=\lim_{\lambda\uparrow\infty} U_\lambda(\xi,\eta)$ exists for any $\xi,\eta\in \Gamma$.  Particularly,  $(\check{\sE},\check{\sF})$ enjoys the following Beurling-Deny representation:
\begin{equation}\label{eq:4.7}
	\check{\sE}(\varphi,\varphi)=\frac{1}{2}\iint_{\Gamma\times \Gamma} \left(\varphi(\xi)-\varphi(\eta)\right)^2U(\xi,\eta)\sigma(d\xi)\sigma(d\eta),\quad \forall \varphi\in \check{\sF}=H^{1/2}(\Gamma);
\end{equation}
see \cite[(5.8.4]{CF12}.  
\end{remark}

Now we have a position to formulate the trace form $(\check{\sE}^\lambda, \check{\sF}^\lambda)$ associated with the DN operator $D_\lambda$.  

\begin{theorem}\label{THM44-2}
Assume that $\lambda\neq \lambda^\text{D}_n/2$.  Then for any $\varphi\in \check{\sF}^\lambda=H^{1/2}(\Gamma)$,  
\begin{equation}\label{eq:4.8}
	\check{\sE}^\lambda(\varphi,\varphi)=\frac{1}{2}\iint_{\Gamma\times \Gamma} \left(\varphi(\xi)-\varphi(\eta)\right)^2U(\xi,\eta)\sigma(d\xi)\sigma(d\eta)+\mathscr{U}_\lambda(\varphi\otimes \varphi),
\end{equation}
where $U(\xi,\eta)=\lim_{\lambda\uparrow \infty}U_\lambda(\xi,\eta)$ appears in Remark~\ref{RM43} and $\mathscr{U}_\lambda$ is defined as \eqref{eq:4.6}.  Furthermore,  for either $\lambda\geq 0,  \varphi\in H^{1/2}(\Gamma)$ or $\lambda>\lambda^\text{D}_1/2,  \varphi\in b\mathcal{B}(\Gamma)\cap H^{1/2}(\Gamma)$,  
\begin{equation}\label{eq:4.9}
\begin{aligned}
	\check{\sE}&^\lambda(\varphi,\varphi)\\
	&=\frac{1}{2}\iint_{\Gamma\times \Gamma}\left(\varphi(\xi)-\varphi(\eta)\right)^2\left(U(\xi,\eta)-U_\lambda(\xi,\eta)\right)\sigma(d\xi)\sigma(d\eta)+\int_\Gamma \varphi(\xi)^2V_\lambda(\xi)\sigma(d\xi),
\end{aligned}\end{equation}
where $V_\lambda(\xi)=\int_\Gamma U_\lambda(\xi,\eta)\sigma(d\eta)$.  
\end{theorem}
\begin{proof}
It follows from \eqref{eq:4.2} that for $\varphi\in H^{1/2}(\Gamma)$,  
\[
\begin{aligned}
	\sE^\lambda(\bH^\lambda_\Gamma \varphi, \bH^\lambda_\Gamma \varphi)&=\sE(\bH^\lambda_\Gamma \varphi, \bH^\lambda_\Gamma \varphi)+\lambda \int_\Omega \left(\bH^\lambda_\Gamma \varphi\right)^2dx \\
	&=\sE(\bH_\Gamma \varphi, \bH_\Gamma \varphi)+\lambda^2 \sE(G^\Omega(\bH^\lambda_\Gamma \varphi),  G^\Omega(\bH^\lambda_\Gamma \varphi))+\lambda \int_\Omega \left(\bH^\lambda_\Gamma \varphi\right)^2dx.
\end{aligned}\]
Since $G^\Omega(\bH^\lambda_\Gamma\varphi)\in \cD(\sL_\Omega)\subset H^1_0(\Omega)$ and $-\sL_\Omega G^\Omega(\bH^\lambda_\Gamma \varphi)=\bH^\lambda_\Gamma \varphi$,  $\sE^\lambda(\bH^\lambda_\Gamma \varphi, \bH^\lambda_\Gamma \varphi)$ is equal to
\[
\sE(\bH_\Gamma \varphi, \bH_\Gamma \varphi)+\lambda^2\int_\Omega \bH^\lambda_\Gamma \varphi G^\Omega(\bH^\lambda_\Gamma\varphi) dx +\lambda \int_\Omega \left(\bH^\lambda_\Gamma \varphi\right)^2dx. 
\]
Using \eqref{eq:4.2},  we get that
\[
\sE^\lambda(\bH^\lambda_\Gamma \varphi, \bH^\lambda_\Gamma \varphi)=\sE(\bH_\Gamma \varphi, \bH_\Gamma \varphi)+\mathscr U_\lambda (\varphi\otimes \varphi).  
\]
In view of the definition of $\check{\sE}^\lambda$ as well as $\check{\sE}$ and  \eqref{eq:4.7},  the formula \eqref{eq:4.8} can be concluded.  

When $\lambda\geq 0$,  \eqref{eq:4.4} holds true for all $\varphi\in H^{1/2}(\Gamma)$.  Hence \eqref{eq:4.3} also holds true for $\varphi\in H^{1/2}(\Gamma)$.  Since $0\leq U_\lambda(\xi, \eta)\leq U(\xi,\eta)$,   it follows that
\[
\begin{aligned}
	\iint_{\Gamma\times \Gamma} \left(\varphi(\xi)-\varphi(\eta)\right)^2&U_\lambda(\xi,\eta)\sigma(d\xi)\sigma(d\eta)\\
	&\leq \iint_{\Gamma\times \Gamma} \left(\varphi(\xi)-\varphi(\eta)\right)^2U(\xi,\eta)\sigma(d\xi)\sigma(d\eta)<\infty.  
\end{aligned}\]
This implies that
\[
	\int_\Gamma \varphi(\xi)^2V_\lambda(\xi)\sigma(d\xi)=\frac{1}{2}\iint_{\Gamma\times \Gamma} \left(\varphi(\xi)-\varphi(\eta)\right)^2U_\lambda(\xi,\eta)\sigma(d\xi)\sigma(d\eta)+\mathscr{U}_\lambda(\varphi\otimes \varphi)<\infty.  
\]
Particularly,  \eqref{eq:4.9} holds true.  When $\lambda^\text{D}_1/2<\lambda<0$,  we note that $U_\lambda(\xi,\eta)\sigma(d\xi)\sigma(d\eta)$ is a finite measure on $\Gamma\times \Gamma$ and hence $V_\lambda(\xi)\sigma(\xi)$ is a finite measure on $\Gamma$.  Then the same representation \eqref{eq:4.9} holds for $\varphi\in b\mathcal{B}(\Gamma)\cap H^{1/2}(\Gamma)$  by means of Lemma~\ref{LM42}~(2).  That completes the proof.  
\end{proof}
\begin{remark}\label{RM45}
We point out that
\begin{itemize}
\item[(a)] For $\lambda>\lambda^\text{D}_1/2$,  it holds that $\check{\sE}^\lambda(\varphi,\varphi)\lesssim \|\varphi\|_{H^{1/2}(\Gamma)}^2$ for all $\varphi\in H^{1/2}(\Gamma)$;
\item[(b)] For $\lambda>0$,  it holds that $\|\varphi\|_{H^{1/2}(\Gamma)}^2\lesssim \check{\sE}^\lambda(\varphi,\varphi)$ for all  $\varphi\in H^{1/2}(\Gamma)$;
\item[(c)] For $\lambda\leq 0$ with $\lambda\neq \lambda^\text{D}_n/2$,  it holds that $$\|\varphi\|_{H^{1/2}(\Gamma)}^2\lesssim \check{\sE}^\lambda(\varphi,\varphi)+(1-\lambda)\int_\Omega \bH^\lambda_\Gamma\varphi(x)^2dx,\quad \forall \varphi\in H^{1/2}(\Gamma).$$ 
\end{itemize}
To show (a),  note that there exists a bounded extension operator $Z: H^{1/2}(\Gamma)\rightarrow H^1(\Omega)$ such that $Z\varphi|_\Gamma=\varphi$ for $\varphi\in H^{1/2}(\Gamma)$;  see,  e.g.,  \cite[Theorem~2.6.11]{SS11}.  Then $Z\varphi-\bH^\lambda_\Gamma \varphi \in H^1_0(\Omega)$ and 
\[
\begin{aligned}
	\sE^\lambda(Z\varphi,Z\varphi)&=\sE^\lambda\left((Z\varphi-\bH^\lambda_\Gamma \varphi)+\bH^\lambda_\Gamma \varphi,  (Z\varphi-\bH^\lambda_\Gamma \varphi)+\bH^\lambda_\Gamma \varphi\right) \\ 
	&=\sE^\lambda(\bH^\lambda_\Gamma\varphi,\bH^\lambda_\Gamma \varphi)+\sE^\lambda(Z\varphi-\bH^\lambda_\Gamma \varphi,Z\varphi-\bH^\lambda_\Gamma \varphi) \\
	&\geq \sE^\lambda(\bH^\lambda_\Gamma\varphi,\bH^\lambda_\Gamma \varphi)=\check{\sE}^\lambda(\varphi,\varphi),
\end{aligned}	
\]
where the inequality is due to the Poincar\'e's inequality.  By means of the boundedness of $Z$,   one get that
\[
	\check{\sE}^\lambda(\varphi,\varphi)\leq \sE^\lambda(Z\varphi,Z\varphi)\lesssim \|Z\varphi\|_{H^1(\Omega)}^2\lesssim \|\varphi\|^2_{H^{1/2}(\Gamma)}.  
\]
To prove (b) and (c),  we have
\[
	\|\varphi\|_{H^{1/2}(\Gamma)}^2=\|\text{Tr}(\bH^\lambda_\Gamma \varphi)\|^2_{H^{1/2}(\Gamma)}\lesssim \|\bH^\lambda_\Gamma \varphi\|^2_{H^1(\Omega)},
\]
since the trace operator $\text{Tr}: H^1(\Omega)\rightarrow H^{1/2}(\Gamma)$ is bounded;  see,  e.g., \cite[Theorem~2.6.8]{SS11}.  When $\lambda>0$,  it holds that $\|\varphi\|_{H^{1/2}(\Gamma)}^2\lesssim \sE^\lambda(\bH^\lambda_\Gamma \varphi,\bH^\lambda_\Gamma \varphi)=\check{\sE}^\lambda(\varphi,\varphi)$ arriving at (b).  When $\lambda\leq 0$ with $\lambda\neq \lambda^\text{D}_n/2$,  since 
\[
\begin{aligned}
	\|\bH^\lambda_\Gamma \varphi\|^2_{H^1(\Omega)}&=\sE^\lambda(\bH^\lambda_\Gamma\varphi, \bH^\lambda_\Gamma\varphi)+(1-\lambda)\int_\Omega \bH^\lambda_\Gamma\varphi(x)^2dx\\
	&=\check{\sE}^\lambda(\varphi,\varphi)+(1-\lambda)\int_\Omega \bH^\lambda_\Gamma\varphi(x)^2dx,
\end{aligned}\]
we obtain (c). 
\end{remark}

Note that the representation \eqref{eq:4.9} is not established for $\lambda^\text{D}_1/2<\lambda<0$ and unbounded $\varphi\in H^{1/2}(\Gamma)$.  In fact, we have the following.

\begin{corollary}
Assume that $\lambda>\lambda^\text{D}_1/2$.  Then \eqref{eq:4.9} holds for $\varphi\in H^{1/2}(\Gamma)$ such that $\int_\Gamma \varphi(\xi)^2|V_\lambda(\xi)|\sigma(d\xi)<\infty$.  
\end{corollary}
\begin{proof}
Fix $\varphi\in H^{1/2}(\Gamma)$ such that $\int_\Gamma \varphi(\xi)^2|V_\lambda(\xi)|\sigma(d\xi)<\infty$.  Set $\varphi_n:=(-n)\vee \varphi \wedge n$.  Then $\varphi_n\in H^{1/2}(\Gamma)$ and 
\[
	\iint_{\Gamma\times \Gamma}\left(\varphi_n(\xi)-\varphi_n(\eta)\right)^2U(\xi,\eta)\sigma(d\xi)\sigma(d\eta)\rightarrow \iint_{\Gamma\times \Gamma}\left(\varphi(\xi)-\varphi(\eta)\right)^2U(\xi,\eta)\sigma(d\xi)\sigma(d\eta).
\]
Since $\iint_{\Gamma\times \Gamma}\left(\varphi(\xi)-\varphi(\eta)\right)^2|U_\lambda(\xi,\eta)|\sigma(d\xi)\sigma(d\eta)\leq 4\int_\Gamma \varphi(\xi)^2|V_\lambda(\xi)|\sigma(d\xi)<\infty$,  one can obtain that
\[
	\iint_{\Gamma\times \Gamma}\left(\varphi_n(\xi)-\varphi_n(\eta)\right)^2U_\lambda(\xi,\eta)\sigma(d\xi)\sigma(d\eta)\rightarrow \iint_{\Gamma\times \Gamma}\left(\varphi(\xi)-\varphi(\eta)\right)^2U_\lambda(\xi,\eta)\sigma(d\xi)\sigma(d\eta)
\]
and
\[
	\int_\Gamma \varphi_n(\xi)^2V_\lambda(\xi)\sigma(d\xi)\rightarrow \int_\Gamma \varphi(\xi)^2V_\lambda(\xi)\sigma(d\xi).  
\]
Therefore Remark~\ref{RM45}~(a) yields that $\check{\sE}^\lambda(\varphi,\varphi)=\lim_{n\rightarrow \infty}\check{\sE}^\lambda(\varphi_n,\varphi_n)$ is equal to the right hand side of \eqref{eq:4.9}. That completes the proof. 
\end{proof}

\begin{example}
Consider $\Omega=\mathbb{D}=\{x\in \bR^d: |x|<1\}$ and $\lambda>\lambda^\text{D}_1/2$.  Note that the Poisson kernel for $\mathbb{D}$ is 
\begin{equation}\label{eq:4.12-2}
	K(x,\xi)=\frac{1}{w_d}\frac{1-|x|^2}{|x-\xi|^d},\quad x\in \mathbb{D}, \xi\in \partial \mathbb{D},
\end{equation}
where $w_d$ is the area of $\mathbb{D}$.  A straightforward computation yields that
\[
	V_\lambda(\xi)=\lambda\int_\mathbb{D} K(x,\xi)\mathbf{E}_x\re^{-\lambda \tau}dx.  
\]
Since  $x\mapsto \mathbf{E}_x \re^{-\lambda\tau}$ is bounded,  it follows that $|V_\lambda(\xi)|\lesssim  |\lambda|\int_\mathbb{D} K(x,\xi)dx$.  By the explicit expression \eqref{eq:4.12-2} of $K$,  one have that $\xi\mapsto \int_\mathbb{D}K(x,\xi)dx$ is finite and constant.  Therefore $V_\lambda(\xi)$ is finite and \eqref{eq:4.9} holds for all $\varphi\in H^{1/2}(\Gamma)$.  

Recall that $\check{\sE}^\lambda_{\alpha_0}$ is non-negative.  By means of \eqref{eq:4.9},  Remark~\ref{RM45} and the boundedness of $V_\lambda$,  one can obtain that for $\alpha>\alpha_0$,  $\|\cdot\|_{\check{\sE}^\lambda_\alpha}$ is an equivalent norm on $H^{1/2}(\Gamma)$ to $\|\cdot\|_{H^{1/2}(\Gamma)}$.  
\end{example}

Looking at \eqref{eq:4.7} and \eqref{eq:4.8},  one can find the following exchangeability of perturbation and trace transformation on $(\sE,\sF)$: For $\lambda\neq \lambda^\text{D}_n/2$,   the trace form of $(\sE^\lambda,\sF^\lambda)$ is identified with the trace Dirichlet form $(\check{\sE},\check{\sF})$ `perturbed' by the bilinear functional $\mathscr{U}_\lambda$.  When $\lambda>\lambda^\text{D}_1/2$,  the perturbation term $\mathscr{U}_\lambda$ is determined by the symmetric measure $$\check\nu_\lambda(d\xi, d\eta):=U_\lambda(\xi,\eta)\sigma(d\xi)\sigma(d\eta).$$  Particularly,  if $\lambda\geq 0$,  then $\check\nu_\lambda$ is a bivariate smooth measure with respect to $\check{\sE}$ in the sense that $\bar{\nu}_\lambda(\cdot):=\check\nu_\lambda(\cdot, \Gamma)$ is smooth with respect to $\check{\sE}$ and $\check\nu_\lambda(d\xi, d\eta)\leq U(\xi,\eta)\sigma(d\xi)\sigma(d\eta)$;  see \cite{Y96-2}.  On account of \cite[Theorem~4.3]{Y96-2}, there exists a multiplicative functional $\check M^\lambda=(\check M^\lambda_t)_{t\geq 0}$ of $\check{X}$ such that $\check\nu_\lambda$ is the bivariate Revuz measure of $\check M^\lambda$; see \cite{Y96} for the definitions of multiplicative functional and its bivariate Revuz measure.  Note that the perturbation of $(\check{\sE},\check\sF)$ by $\check\nu_\lambda$ (more exactly, by $\mathscr{U}_\lambda$) corresponds to the killing transformation of $\check{X}$ by $\check M^\lambda$; see \cite[Theorem~3.10]{Y96}.  Therefore we have the following.

\begin{corollary}
Assume $\lambda\geq 0$.  Let $X^\lambda$ be the $\lambda$-subprocess of $X$ associated with $(\sE^\lambda,\sF^\lambda)$. Then there exists a multiplicative functional $\check M^\lambda$ of $\check{X}$ such that the time changed process of $X^\lambda$ by the local time on $\Gamma$ is identified with the subprocess of $\check{X}$ perturbed by $\check{M}^\lambda$.  
\end{corollary}

\subsection{Irreducibility and ground state}

Let $(\check{T}^\lambda_t)_{t\geq 0}:=(\re^{-tD_\lambda})_{t\geq 0}$ be the strongly continuous semigroup associated with $(\check{\sE}^\lambda,\check{\sF}^\lambda)$ on $L^2(\Gamma)$.  Denote its resolvent by $(\check{G}^\lambda_\alpha)_{\alpha>\alpha_0}$,  i.e.  $\check{G}^\lambda_{\alpha} f=\int_0^\infty \re^{-\alpha t}\check{T}^\lambda_t f dt$ for $f\in L^2(\Gamma)$.   %Write $\check{T}_t:=\check{T}^0_t,  \check{G}_\alpha:=\check{G}^0_\alpha$ for brevity.

\begin{lemma}\label{LM46}
Let $A\subset \Gamma$ be such that $1_A\in H^{1/2}(\Gamma)$ and $\int_{A\times (\Gamma\setminus A)}U(\xi, \eta)\sigma(d\xi)\sigma(d\eta)=0$.  Then $\sigma(A)=0$ or $\sigma(\Gamma\setminus A)=0$.  %Particularly,  $(\check{T}_t)_{t\geq 0}$ is irreducible (in the sense of Definition~\ref{DEFC1}).  
\end{lemma}
\begin{proof}
Since $1_A\in H^{1/2}(\Gamma)=\check{\sF}$,  take $1_{\tilde{A}}$ to be its $\check{\sE}$-quasi-continuous $\sigma$-version.  Then $\tilde{A}=A$,  $\sigma$-a.e., and $\tilde{A}$ is  nearly Borel and simultaneously finely open and finely closed.  

In view of \eqref{eq:4.5-2},  $U(\xi,\eta)=0$ if and only if $U_\lambda(\xi,\eta)=0$ for one (equivalently all) $\lambda>0$. 
Set $h(s,x,\xi):=\frac{1}{2}\frac{\partial p^\Omega_s(x,\xi)}{\partial {\bn_\xi}}$,  i.e.  $\mathbf{P}_x(\tau\in ds, X_\tau\in d\xi)=h(s,x,\xi)ds\sigma(d\xi)$;  see \eqref{eq:4.7-2}.  Then $U_\lambda(\xi,\eta)=\lambda\int_\Omega K_\lambda(x,\xi)K(x,\eta)dx$, where
\[
	K(x,\xi)=\int_0^\infty h(s,x,\xi)ds,\quad K_\lambda(x,\eta)=\int_0^\infty \re^{-\lambda s}h(s,x,\eta)ds.  
\]
Clearly $K(x,\xi)=0$ amounts to $K_\lambda(x,\xi)=0$.  Hence for $\lambda>0$,  $U_\lambda(\xi,\eta)=0$ is equivalent to $\int_\Omega K(x,\xi)K(x,\eta)dx=0$.  Since $\int_{A\times (\Gamma\setminus A)}U(\xi,\eta)\sigma(d\xi)\sigma(d\eta)=0$,  it follows that
\[
	\int_\Omega \mathbf{P}_x(X_\tau\in \tilde A)\mathbf{P}_x(X_\tau\in \Gamma \setminus \tilde A)dx=0.  
\]
Particularly,  $\mathbf{P}_x(X_\tau\in \tilde A)=0$ or $\mathbf{P}_x(X_\tau\in \Gamma\setminus \tilde A)=0$ for some $x\in \Omega$.  We only treat the former case $u(x):=\mathbf{P}_x (X_\tau\in \tilde A)=0$ for some $x\in \Omega$.  Note that $u\geq 0$ and $u$ is harmonic in $\Omega$ (see \cite[Theorem~1.23]{CZ95}).  Then the strong maximum principle implies that $u\equiv 0$ in $\Omega$.   Since $u$ is $\sE$-quasi-continuous,  it follows that $u=0$,  $\sE$-a.e. on $\bar{\Omega}$.  On account of \cite[Theorem~5.2.8]{CF12}, $1_{\tilde{A}}=u|_\Gamma=0$,  $\check{\sE}$-q.e.  Particularly,   $1_{\tilde{A}}=0$,  $\sigma$-a.e.  and hence $\sigma(A)=\sigma(\tilde{A})=0$.  That completes the proof. 
\end{proof}

For $\lambda>\lambda^\text{D}_1/2$,  the irreducibility of $(\check{T}^\lambda_t)_{t\geq 0}$ was first obtained in \cite[Theorem~4.2]{AM12} with the help of Krein-Rutman theorem.  It is worth pointing out that $\lambda>\lambda^D_1/2$ is equivalent to either condition in Lemma~\ref{LM311} (as well as the gaugeability of $(\Omega,-\lambda)$).
In what follows we give an alternative proof that based on the form representation of $D_\lambda$.  

\begin{corollary}\label{COR410}
For $\lambda>\lambda^\text{D}_1/2$,  $(\check{T}^\lambda_t)_{t\geq 0}$ is irreducible.  
\end{corollary}
\begin{proof}
The positivity of $(\check{T}^\lambda_t)_{t\geq 0}$ has been obtained in Theorem~\ref{THM312}.  
To show the irreducibility,  take an invariant set $A\subset \Gamma$ with respect to $(\check{T}^\lambda_t)_{t\geq 0}$.  Since $1_\Gamma\in H^{1/2}(\Gamma)$ and Theorem~\ref{THMC2},  we have that $1_A\in H^{1/2}(\Gamma)$ and 
\[
	0=\check{\sE}^\lambda(1_A, 1_{\Gamma\setminus A})=\iint_{A\times (\Gamma \setminus A)}\left(U(\xi,\eta)-U_\lambda(\xi,\eta)\right)\sigma(d\xi)\sigma(d\eta).  
\]
Clearly,  $U-U_\lambda\geq 0$.  
In view of \eqref{eq:4.5-2},  $U(\xi,\eta)-U_\lambda(\xi,\eta)=0$ if and only if $U(\xi,\eta)=0$.  Therefore Lemma~\ref{LM46} yields $\sigma(A)=0$ or $\sigma(\Gamma\setminus A)=0$.  That completes the proof. 
\end{proof}

Next we claim the existence of a special eigenfunction, called the ground state, of $D_\lambda$ by virtue of the Krein-Rutman theorem.  Though the derivation is classical,  we present a proof for readers' convenience. 

\begin{theorem}\label{THM411}
Assume that $\lambda\neq \lambda^\text{D}_n/2$.  Then the following hold:
\begin{itemize}
\item[(1)] The DN operator $D_\lambda$ has a purely discrete spectrum $\sigma(D_\lambda)$ in the sense that $\sigma(D_\lambda)$ consists only of eigenvalues of finite multiplicities which have no finite accumulation points. 
\item[(2)] Assume that $\lambda>\lambda^\text{D}_1/2$ and let $\lambda^\text{DN}_1$ be the smallest eigenvalue in $\sigma(D_\lambda)$.  Then $\lambda^\text{DN}_1$ is a simple eigenvalue admitting a strictly positive eigenfunction $h_\lambda\in L^2(\Gamma)$,  i.e.  $h_\lambda>0$,  $\sigma$-a.e.,  and $D_\lambda h_\lambda=\lambda^\text{DN}_1 h_\lambda$.  
No other eigenvalues admit strictly positive eigenfunctions. 
\end{itemize}  
\end{theorem}
\begin{proof}
\begin{itemize}
\item[(1)] Note that for $\alpha>\alpha_0$,  $\check{G}^\lambda_\alpha(L^2(\Gamma))\subset \cD(D_\lambda)\subset H^{1/2}(\Gamma)$ and $H^{1/2}(\Gamma)$ is compactly embedded in $L^2(\Gamma)$.  Hence $\check{G}^\lambda_\alpha$ is a compact operator on $L^2(\Gamma)$.  The first assertion is a result of \cite[Proposition~2.11]{S12}.  
\item[(2)] The existence of $\lambda^\text{DN}_1$ is due to the lower semi-boundedness of $D_\lambda$.  Note that $\beta \in \sigma(D_\lambda)$ if and only if $(\alpha+\beta)^{-1}$ is an eigenvalue of $\check{G}^\lambda_\alpha$,  and meanwhile for $\varphi\in L^2(\Gamma)$,  $D_\lambda \varphi=\beta \varphi$,  if and only if $\check{G}^\lambda_\alpha \varphi=(\alpha+\beta)^{-1}\varphi$.  
Particularly,  $(\alpha+\lambda^\text{DN}_1)^{-1}$ is equal to the spectral radius of $\check{G}^\lambda_\alpha$.  
It follows from the irreducibility of $(\check{T}^\lambda_t)_{t\geq 0}$ that $\check{G}^\lambda_\alpha \varphi>0$,  $\sigma$-a.e.,  for any non-zero $\varphi\in pL^2(\Gamma)$.  Applying the Krein-Rutman theorem (see,  e.g.,  \cite[Theorem~1.2]{D06}) to $\check{G}^\lambda_\alpha$,  we can obtain that $(\alpha+\lambda^\text{DN}_1)^{-1}$ is a simple and unique eigenvalue of $\check{G}^\lambda_\alpha$ admitting a strictly positive eigenfunction $h_\lambda$.  This amounts to that $\lambda^\text{DN}_1$ is a simple eigenvalue of $D_\lambda$ and the only eigenvalue admitting a strictly positive eigenfunction. 
\end{itemize}
 That completes the proof. 
\end{proof}

When $\lambda\neq \lambda^\text{D}_n/2$,  $\lambda^\text{DN}_1$ is called the \emph{first eigenvalue} of $D_\lambda$.  It is worth pointing out that if $\lambda<0$,  then $\lambda^\text{DN}_1<0$;  see \cite[Lemma~3.2]{AM12}.  
When $\lambda>\lambda^\text{D}_1/2$,  $\lambda^\text{DN}_1$ is a simple eigenvalue of $D_\lambda$ and we call the eigenfunction $h_\lambda$ in this theorem the \emph{ground state} of $D_\lambda$.  

\subsection{Markov processes $h$-associated with the DN operators}

We have obtained in Corollary~\ref{COR28} the probabilistic counterpart of $D_\lambda$ for $\lambda\geq 0$,  i.e.  the time changed process of the $\lambda$-subprocess of reflected Brownian motion on $\bar{\Omega}$ by the local time on $\Gamma$.  For $\lambda<0$,  $-D_\lambda$ is not straightforwardly identified with the $L^2$-generator of a certain Markov process,  as $V_\lambda<0$ in \eqref{eq:4.9} usually leads to the failure of Markovian property for $(\check{\sE}^\lambda,\check{\sF}^\lambda)$.   Instead,  another tactic involving $h$-transformations arises in \S\ref{SEC34}.  Note that either condition in Lemma~\ref{LM311} amounts to $\lambda>\lambda^\text{D}/2$.  For $\alpha\in \bR$ and $h\in \mathbf{E}^+_\alpha$,  let $(\check{\sE}^{\lambda, h}_\alpha,  \check{\sF}^{\lambda, h})$ be the $h$-transform of $(\check{\sE}^\lambda_\alpha, \check{\sF}^\lambda)$.  Then we have the following.

\begin{theorem}\label{THM412}
Assume that $\lambda>\lambda_1^\text{D}/2$.  Let $\alpha\in \bR$ and $h\in \mathbf{E}^+_\alpha$.  Then the following hold:
\begin{itemize}
\item[(1)]  The $h$-transform $(\check{\sE}^{\lambda,h}_\alpha, \check{\sF}^{\lambda,h})$ is a quasi-regular and irreducible lower bounded symmetric Dirichlet form on $L^2(\Gamma, h^2\cdot \sigma)$.  Particularly,  the $(\alpha,h)$-associated Markov process of $D_\lambda$ is irreducible. 
\item[(2)]  $(\check{\sE}^{\lambda,h}_\alpha, \check{\sF}^{\lambda,h})$ is non-negative if and only if $\alpha\geq -\lambda^\text{DN}_1$.  
\item[(3)] $(\check{\sE}^{\lambda,h}_\alpha, \check{\sF}^{\lambda,h})$ is recurrent,  if and only if $\alpha=-\lambda^\text{DN}_1$ and $h=c\cdot h_\lambda$ for some constant $c>0$,  where $h_\lambda$ is the ground state of $D_\lambda$.  
\end{itemize}
\end{theorem}
\begin{proof}
\begin{itemize}
\item[(1)] We only need to prove the irreducibility.  In fact,  the $L^2$-semigroup of $(\check{\sE}^{\lambda,h}_\alpha,\check{\sF}^{\lambda,h})$ is 
\begin{equation}\label{eq:4.13}
	\check{T}_t^{\lambda, h, \alpha}\varphi:=\frac{\re^{-\alpha t}\check{T}^\lambda_t(\varphi h)}{h},\quad \varphi \in L^2(\Gamma, h^2 \cdot \sigma). 
\end{equation}
Since $(\check{T}^\lambda_t)_{t\geq 0}$ is irreducible,  it follows that $(\check{T}_t^{\lambda, h, \alpha})_{t\geq 0}$ is also irreducible.  As a result of Theorem~\ref{THMC2},  $(\check{\sE}^{\lambda,h}_\alpha,\check{\sF}^{\lambda,h})$ is irreducible.  
\item[(2)] Note that $\lambda^\text{DN}_1=\inf\{(D_\lambda \varphi, \varphi)_\sigma: \varphi\in \cD(D_\lambda), \|\varphi\|_{L^2(\Gamma)}=1\}$ and $\check{\sE}^\lambda_\alpha(\varphi,\varphi)=(D_\lambda \varphi,\varphi)_\sigma+\alpha(\varphi, \varphi)_\sigma$ for $\varphi\in \cD(D_\lambda)$.  Hence $\check{\sE}^\lambda_\alpha$ is non-negative if and only if $\alpha\geq - \lambda^\text{DN}_1$.  This also amounts to that $(\check{\sE}^{\lambda,h}_\alpha, \check{\sF}^{\lambda,h})$ is non-negative.  
\item[(3)] Suppose $\alpha=-\lambda^\text{DN}_1$ and $h=c\cdot h_\lambda$.  Clearly $1\in \check{\sF}^{\lambda,h}$ and it follows from \eqref{eq:4.13} that $\check{T}^{\lambda, h,\alpha}_t 1=1$.  Hence $\check{\sE}^{\lambda, h}_\alpha(1,1)=0$.  As a result,  $(\check{\sE}^{\lambda,h}_\alpha, \check{\sF}^{\lambda,h})$ is recurrent.  To the contrary,  the recurrence of $(\check{\sE}^{\lambda,h}_\alpha, \check{\sF}^{\lambda,h})$ implies its conservativeness,  which particularly yields that $\check{T}^{\lambda, h,\alpha}_t 1=1$.  Thus $\check{T}^\lambda_t h=\re^{-\alpha t} h$.  By the Hille-Yosida theorem,  we get that $D_\lambda h=-\alpha h$.   In view of Theorem~\ref{THM411},  we  can eventually conclude that $\alpha=-\lambda^\text{DN}_1$ and $h=c\cdot h_\lambda$ for some $c>0$.  
\end{itemize}
That completes the proof.
\end{proof}
\begin{remark}
The irreducibility of $(\check{\sE}^{\lambda,h}_\alpha, \check{\sF}^{\lambda,h})$ implies that it is either recurrent or transient.  The proof of the third assertion indicates that if $(\check{\sE}^{\lambda,h}_\alpha, \check{\sF}^{\lambda,h})$ is transient,  then it is not conservative. 
\end{remark}

Finally let us formulate the Beurling-Deny decomposition for $(\check{\sE}^{\lambda, h}_{\alpha},\check{\sF}^{\lambda,h})$.  When $\lambda\geq 0$,  $(\check{\sE}^\lambda,\check{\sF}^\lambda)$ is already a regular Dirichlet form on $L^2(\Gamma)$.  A related consideration for this case is referred to \cite{Y98}. 

\begin{proposition}
For either $\lambda\geq 0,  \varphi\in b\check{\sF}^{\lambda,h}$ or $\lambda>\lambda^\text{D}_1/2,  \varphi\in b\check{\sF}^{\lambda,h}$ with $\varphi h\in bH^{1/2}(\Gamma)$,  $\check{\sE}^{\lambda,h}_{\alpha}(\varphi,\varphi)$ enjoys the following Beurling-Deny decomposition:
\[
\begin{aligned}
	\check{\sE}^{\lambda,h}_{\alpha}(\varphi,\varphi)&=\frac{1}{2}\iint_{\Gamma\times \Gamma}\left(\varphi(\xi)-\varphi(\xi)\right)^2 h(\xi)h(\eta)\left(U(\xi,\eta)-U_\lambda(\xi,\eta)\right)\sigma(d\xi)\sigma(d\eta)\\ &\qquad +\int_\Gamma \varphi^2(\xi)k^{h}(d\xi),
\end{aligned}\]
where $\int_\Gamma \varphi(\xi)^2k^{h}(d\xi)=\check{\sE}_\alpha^\lambda(\varphi^2 h,h)$.
\end{proposition}
\begin{proof}
%The representation for the first case $\lambda\geq 0$ has been obtained in \cite[Theorem~3.1]{Y98}.  We only treat the case $\lambda^\text{D}_1/2<\lambda<0$.  
Clearly,  $1\in \check{\sF}^{\lambda,h}$ and $$\int_\Gamma \varphi(\xi)^2k^{h}(d\xi)=\check{\sE}^{\lambda,h}_{\alpha}(\varphi^2, 1)=\check{\sE}^\lambda_\alpha(\varphi^2 h, h).$$
Note that  $\check{\sE}^{\lambda,h}_{\alpha}(\varphi,\varphi)=\check{\sE}^\lambda_\alpha(\varphi h, \varphi h)$,  and in view of Theorem~\ref{THM44-2},  $\check{\sE}^\lambda(\varphi h, \varphi h)$ is equal to $I_1+I_2$,  where
\[
\begin{aligned}
	&I_1:=\frac{1}{2}\iint\left(\varphi(\xi)h(\xi)-\varphi(\eta)h(\eta) \right)^2U(\xi,\eta)\sigma(d\xi)\sigma(d\eta),  \\
	&I_2:=\iint\varphi(\xi)h(\xi)\varphi(\eta)h(\eta)U_\lambda(\xi,\eta)\sigma(d\xi)\sigma(d\eta).  
\end{aligned}\]
A straightforward computation yields that $I_1=I_{11}+I_{12}$,  where
\[
\begin{aligned}
&I_{11}=\frac{1}{2}\iint \left(\varphi(\xi)-\varphi(\eta)\right)^2h(\xi)h(\eta)U(\xi,\eta)\sigma(d\xi)\sigma(d\eta),\\
&I_{12}=\frac{1}{2}\iint (\varphi^2(\xi)h(\xi)
	-\varphi^2(\eta)h(\eta))(h(\xi)-h(\eta))U(\xi,\eta)\sigma(d\xi)\sigma(d\eta). 
\end{aligned}
\]
On the other hand,  $\check{\sE}^\lambda(\varphi^2 h, h)=I_{12}+\mathscr{U}_\lambda(\varphi^2 h \otimes h)$.  Since $\varphi^2 h\in H^{1/2}(\Gamma)$ is bounded when $\lambda^D_1/2<\lambda<0$,  it follows from \eqref{eq:4.4} and $\bH_\Gamma h(x)=\mathbf{E}_x \tilde{h}(X_\tau)$,  where $\tilde{h}$ is the $\check{\sE}$-quasi-continuous $\sigma$-version of $h$,  that
\[
	\mathscr{U}_\lambda(\varphi^2 h\otimes h)=\int_\Gamma \varphi(\xi)^2h(\xi)h(\eta)U_\lambda(\xi,\eta)\sigma(d\xi)\sigma(d\eta).  
\]
Thus we get that
\[
\begin{aligned}
	\check{\sE}^{\lambda,h}_{\alpha}&(\varphi,\varphi)-\int_\Gamma \varphi^2(\xi)k^{h}(d\xi)\\
	&=I_{11}+I_2-\mathscr{U}_\lambda(\varphi^2h\otimes h) \\
&=\frac{1}{2}\iint_{\Gamma\times \Gamma}\left(\varphi(\xi)-\varphi(\xi)\right)^2 h(\xi)h(\eta)\left(U(\xi,\eta)-U_\lambda(\xi,\eta)\right)\sigma(d\xi)\sigma(d\eta).
\end{aligned}\]
That completes the proof.
\end{proof}

%   $U_\lambda$ can be extended to a finite symmetric bimeasure still denoted by $U_\lambda$ on $\mathcal{B}(\Gamma \times \Gamma)$,  which admits the following expression:
%\begin{equation}\label{eq:4.3}
%	U_\lambda(d\xi d\eta)=U_\lambda(\xi, \eta)\sigma(d\xi) \sigma(d\eta),
%\end{equation}

%Note that $|U_\lambda(1\otimes 1)|\leq |\lambda| \int_\Omega \mathbf{E}_x\left[\re^{-\lambda \tau} \right]dx<\infty$ and
%Hence the first assertion is the consequence of \cite[III: 74]{DM78}. 

\section{DN operators for Schr\"odinger operators}\label{SEC5}

Let $\Omega$,  $\Gamma=\partial \Omega$ and $\sigma$ be the same as those in \S\ref{SEC4}.  Further let $(a_{ij}(x))_{1\leq i,j\leq d}$ be a matrix function on $\Omega$ such that $a_{ij}=a_{ji}$ and there exists a constant $C>1$ such that for any $\xi=(\xi_1,\cdots, \xi_d)\in \bR^d$, 
\begin{equation}\label{eq:36}
	C^{-1} |\xi|^2\leq \sum_{i,j=1}^da_{i,j}(x)\xi_i\xi_j\leq C|\xi|^2,\quad \forall x\in \Omega.  
\end{equation}
In this section consider the Dirichlet form:
\begin{equation}\label{eq:52-2}
\begin{aligned}
		\sF&:=H^1(\Omega),\\
		\sE(u,v)&:=\frac{1}{2}\sum_{i,j=1}^d\int_\Omega a_{ij}(x)\partial_{x_i}u\partial_{x_j}vdx,\quad u,v\in \sF.  
\end{aligned}
\end{equation}
Clearly $(\sE,\sF)$ is a regular and irreducible Dirichlet form on $L^2(\bar\Omega)$ associated with a Markov process $X$ on $\bar{\Omega}$.  The irreducibility is due to,  e.g.,  \cite[Corollary~4.6.4]{FOT11}.  

Denote by $\mathbf{S}$ (resp. $\mathring{\mathbf{S}}$) the totality of positive smooth measures (positive Radon smooth measures) on $\bar\Omega$ with respect to $\sE$.  Note that $(\sE,\sF)$ enjoys the same quasi notations as \eqref{eq:Brownian}.  Hence $\mathbf{S}$ and $\mathring{\mathbf{S}}$ are also the same as those in \S\ref{SEC4}.  Particularly,  $\mu:=\sigma\in \mathring{\bS}$ with $\text{qsupp}[\sigma]=\Gamma$.  Denote by $(\check{\sE},\check{\sF})$ the trace Dirichlet form of $(\sE,\sF)$ on $L^2(\Gamma)$.  
%\subsection{Markov processes $h$-associated with DN operators for Sch\"odinger operators}
Take
\[
	\kappa(x)=V(x)dx\in \bS-\bS
\]
with $V\in L^\infty(\Omega)$.  
% and $(\sE^\kappa,\sF^\kappa)$ be the perturbation of $(\sE,\sF)$ by $\kappa$.  When 
% \eqref{eq:42-3} holds true with $G:=\Omega, F:=\Gamma$ and $\kappa^-$ is $\sE^{\kappa^+}$-form bounded on trace (in the sense of Definition~\ref{DEF34}),  the DN operator $\sN_\kappa$ for $(\sE^\kappa,\sF^\kappa)$ on $L^2(\Gamma)$ is a lower semi-bounded and self-adjoint operator on $L^2(\Gamma)$.  
%In what follows we will pay our attention to the special case
%\[
%	\kappa(dx):=V(x)dx,\quad \text{ with }V\in L^\infty(\Omega).
%\]
% and especially figure out the probabilistic counterpart of $\sN_\kappa$. 
For emphasis,  write 
\[
\begin{aligned}
	&(\sE^V,\sF^V):=(\sE^\kappa,\sF^\kappa),\\
	& \sF^{V,\Omega}:=H^1_0(\Omega),\quad \sE^{V,\Omega}(u,v):=\sE^V(u,v),\; u,v\in \sF^{V,\Omega}. 
\end{aligned}\]
Both $(\sE^V,\sF^V)$ and $(\sE^{V,\Omega},\sF^{V,\Omega})$ are lower bounded closed forms.  Denote their $L^2$-generators by $\sL_V:=\sL-V$ and $\sL_{V,\Omega}$ respectively,  where $\sL$ is the generator of $(\sE,\sF)$.   From Lemma~\ref{LM32}~(1),  we know that $\sF^V_\re:=\sF_\re \cap L^2(\Omega,  |V(x)|dx)=H^1(\Omega)$ and $\sF^{V,\Omega}_\re:=\sF^{\kappa,\Omega}_\re=H^1_0(\Omega)$.  Set $\cH^V_\Gamma:=\cH^\kappa_\Gamma=\{u\in H^1(\Omega): \sE^V(u,v)=0,\forall v\in H^1_0(\Omega)\}$. 

The DN operator for the Schr\"odinger operator $\sL_V$ defined as below has been widely studied in,  e.g.,  \cite{AE14, AE20,  EO14, EO19}.  Note that $0\notin \sigma(\sL_{V,\Omega})$ implies \eqref{eq:42-3};  see Lemma~\ref{LMB1}.

%When $\kappa^-\in \bS_{EK}$ (for example,  $\kappa^-=V^-(x)dx$ with $V^-\in L^\infty(\Omega)$),  $(\sE^\kappa,\sF^\kappa)$ is lower bounded and induces a so-called Feymman-Kac semigroup,  which is strongly continuous on $L^2(\bar\Omega)$.  In the meanwhile we denote its generator on $L^2(\bar\Omega)$ by $\sL_\kappa=\sL-\kappa$,
%where $\sL$ strands for the generator of $(\sE,\sF)$ on $L^2(\bar\Omega)$.  In particular if $\kappa=V(x)dx$ with $V\in L^\infty(\Omega)$,  then $\sL_\kappa=\sL_V:=\sL-V$.  

%It is easy to verify that $D_V$ in the following definition is identified with $\sN_\kappa$ in \S\ref{SEC4} for $\kappa(dx)=V(x)dx$. 

\begin{definition}\label{DEF45}
Let $V\in L^\infty(\Omega)$ such that $0\notin \sigma(\sL_{V,\Omega})$.  The DN operator $D_V$  for $\sL_V$ on $L^2(\Gamma)$ is defined as follows:
\[
	\begin{aligned}
	&\cD(D_V):=\bigg\{\varphi\in L^2(\Gamma): \exists u\in \cH^V_\Gamma, f\in L^2(\Gamma)\text{ s.t. }\Tr(u)=\varphi, \\&\qquad\qquad\qquad \qquad\ \qquad\qquad \sE^V(u,v)=\int_\Gamma f v|_\Gamma d\sigma\text{ for all }v\in H^1(\Omega)\bigg\},  \\
	&D_V\varphi:=f,
\end{aligned}\]
The function $u\in \cH^V_\Gamma$ with $\text{Tr}(u)=\varphi$ is called the \emph{$V$-harmonic extension} of $\varphi$.  
\end{definition}

It is easy to verify that $D_V$ is identified with the DN operator for $(\sE^V,\sF^V)$ on $L^2(\Gamma)$.  On account of Corollary~\ref{COR47},  $D_V$ is lower semi-bounded and self-adjoint.  Let $(\check{\sE}^V,\check{\sF}^V):=(\check{\sE}^\kappa,\check{\sF}^\kappa)$ be the trace form associated with $D_V$.  Then $\check{\sF}^V=\check{\sF}=H^{1/2}(\Gamma)$.  Denote by $\bH^V_\Gamma \varphi$ the $V$-harmonic extension of $\varphi \in H^{1/2}(\Gamma)$ and write $\bH_\Gamma \varphi:= \bH^0_\Gamma \varphi$.  Define
\[
\mathscr{U}_V(\varphi\otimes \phi):=\int_\Gamma \bH^V_\Gamma \varphi(x) \bH_\Gamma \phi(x) V(x)dx,\quad \varphi, \phi\in H^{1/2}(\Gamma).  
\]
Repeating the argument to prove \eqref{eq:4.8},  we can also represent $\check{\sE}^V$ as the `perturbation' of $\check{\sE}$ by the bilinear symmetric functional $\mathscr{U}_V$.  

\begin{proposition}
Assume that $0\notin \sigma(\sL_{V,\Omega})$.  Then for any $\varphi, \phi\in H^{1/2}(\Gamma)$, 
\[
	\check{\sE}^V(\varphi,\phi)=\check{\sE}(\varphi,\phi)+\mathscr U_V(\varphi\otimes \phi).  
\]
\end{proposition}

Let us turn to derive the probabilistic counterpart of $D_V$.  
When $V\geq 0$,  $(\check{\sE}^V,\check{\sF}^V)$ is clearly the trace Dirichlet form of $(\sE^V,\sF^V)$ on $L^2(\Gamma)$ and the probabilistic counterpart of $D_V$ can be easily figured out by applying Theorem~\ref{THM26}.  The more interesting case is that $V$ has a negative part.  Though it is hard to formulate the representation of $\check{\sE}^V$ like $\check{\sE}^\lambda$ in \S\ref{Revisiting},   the irreducibility of the $L^2$-semigroup associated with $D_V$ is obtained in \cite[Proposition~7.4]{AE20}.
That is the following. 

%As the analogue of the condition $\lambda>\lambda^\text{D}_1/2$ in \S\ref{Revisiting},  we introduce the following:
%\begin{itemize}
%\item[(H)] 
%\end{itemize}
%Clearly the positivity of $-\sL_{V,\Omega}$ amounts to that $\sE^V(u,u)\geq 0$ for all $u\in H^1_0(\Omega)$,  and (H) The following result obtained in  is crucial to our argument. 

\begin{lemma}\label{LM53-2}
Assume that $a_{ij}\in L^\infty(\Omega)$ for $1\leq i,j\leq d$,  $0\notin \sigma(\sL_{V,\Omega})$ and $-\sL_{V,\Omega}$ is positive,  i.e.  $(-\int_\Omega \sL_{V,\Omega} u, u)_m\geq 0$ for any $u\in \cD(\sL_{V,\Omega})$.  Then the $L^2$-semigroup $\check{T}^V_t:=\re^{-t D_V}$ associated with $-D_V$ is irreducible. 
\end{lemma}
\begin{remark}
In view of Lemma~\ref{LM311},  the positivity of $\sL_{V,\Omega}$ amounts to that $\sE^V(u,u)\geq 0$ for all $u\in H^1_0(\Omega)$. Together with the positivity of $\sL_{V,\Omega}$,  $0\notin \sigma(\sL_{V,\Omega})$ implies  additionally that $\sE^V(u,u)=0$ for $u\in H^1_0(\Omega)$ if and only if $u=0$. 
\end{remark}

Mimicking the proof of Theorem~\ref{THM411},  one can conclude that $\lambda^V_1:=\text{min }\sigma(D_V)$ is a simple eigenvalue admitting a strictly positive eigenfunction,  i.e.  there exists $h_V\in L^2(\Gamma)$ such that $h_V>0$,  $\sigma$-a.e.,  and $D_Vh_V=\lambda^V_1 h_V$.  No other eigenvalues admit strictly positive eigenfunctions.  We call $h_V$ the \emph{ground state} of $D_V$.  

For $\alpha\in\bR$, denote by $\mathbf{E}^+_\alpha$ the totality of all strictly positive $\alpha$-excessive functions in $\check{\sF}^V$.  Particularly $h_V\in \mathbf{E}^+_{-\lambda_1^V}$.  For $h\in \mathbf{E}^+_\alpha$,  let $(\check{\sE}^{V,h}_\alpha, \check{\sF}^{V,h})$ be the $h$-transform of $(\check{\sE}^V_\alpha,\check{\sF}^V)$.  As an extension of Theorem~\ref{THM412},  one can obtain the following result.  The proof can be completed by repeating the argument in the proof of Theorem~\ref{THM412}, and we omit it. 

\begin{theorem}\label{THM55}
Adopt the assumptions of Lemma~\ref{LM53-2}.  Let $\alpha\in \bR$ and $h\in \mathbf{E}^+_\alpha$.  Then the following hold:
\begin{itemize}
\item[(1)]  $(\check{\sE}^{V,h}_\alpha, \check{\sF}^{V,h})$ is a quasi-regular and irreducible lower bounded symmetric Dirichlet form on $L^2(\Gamma, h^2\cdot \sigma)$.  Particularly,  the $(\alpha,h)$-associated Markov process of $D_V$ is irreducible. 
\item[(2)]  $(\check{\sE}^{V,h}_\alpha, \check{\sF}^{V,h})$ is non-negative if and only if $\alpha\geq -\lambda^V_1$.  
\item[(3)] $(\check{\sE}^{V,h}_\alpha, \check{\sF}^{V,h})$ is recurrent,  if and only if $\alpha=-\lambda^V_1$ and $h=c\cdot h_\lambda$ for some constant $c>0$,  where $h_V$ is the ground state of $D_V$.  
\end{itemize}
\end{theorem}

\section{DN operators for perturbations supported on boundary}\label{SEC7}

In this section we pay our attention to the special perturbation supported on the boundary.  Let $(\sE,\sF)$ be a regular and irreducible Dirichlet form on $L^2(E,m)$ and $\mu\in \mathring{\bS}$ with $\text{qsupp}[\mu]=F$ of positive $\sE$-capacity.  Denote by $(\check{\sE},\check{\sF})$ the trace Dirichlet form of $(\sE,\sF)$ on $L^2(F,\mu)$ and by $\check{\bS}$ the family of positive smooth measures with respect to $\check{\sE}$.  Take $\kappa:=\kappa^+-\kappa^-\in \bS-\bS$ such that $|\kappa|(G)=0$ where $G:=E\setminus F$.  

\subsection{Characterization of trace forms}\label{SEC42}

We first summarize some facts about $(\sE^\kappa,\sF^\kappa)$ and its trace form $(\check{\sE}^\kappa,\check{\sF}^\kappa)$ defined as \eqref{eq:45} and \eqref{eq:46}.  
%In what follows we will treat a special case that $\text{qsupp}\left[|\kappa|\right]\subset F=\text{qsupp}[\mu]$.   %Denote them by  and $\check{\bS}_{EK}$ respectively.  %The main result of this section is as follows.

\begin{proposition}\label{THM33}
Assume that $|\kappa|(G)=0$.   Then the following hold:
\begin{itemize}
\item[(1)]  $\sF^\kappa_\re=\sF^{\kappa,G}_\re\oplus \cH^\kappa_F$,  i.e.  \eqref{eq:42-3} holds true.   %Particularly $\sF^{\kappa,G}_\re\cap \cH^\kappa_F=\{0\}$. 
\item[(2)] $\kappa|_F \in \check{\bS}-\check{\bS}$.
\item[(3)] $(\check{\sE}^\kappa,\check{\sF}^\kappa)$  is identified with the perturbation of $(\check{\sE},\check{\sF})$ by $\kappa|_F$. 
% $\varphi\in\cD(\sN_\kappa)$,  if and only if $\varphi\in \check{\sF}^\kappa$ and for some $f\in L^2(F,\mu)$,  
%\begin{equation}\label{eq:31}
%\check{\sE}^\kappa(\varphi,\phi)=\int_F f \phi d\mu,\quad \forall \phi\in \check{\sF}^\kappa, 
%\end{equation}
%where $(\check{\sE}^\kappa,\check{\sF}^\kappa)$ is the perturbation of $(\check{\sE},\check{\sF})$  by $\kappa|_F$.  Meanwhile $f=\sN_\kappa \varphi$.  
%\item[(4)] Assume further that $\kappa^-|_F$ is $\check{\sE}^{\kappa^+}$-form bounded.  Then  Particularly,  $\sN_\kappa$ is a lower semi-bounded and self-adjoint operator on $L^2(F,\mu)$.  
\end{itemize} 
\end{proposition}
\begin{proof}
\begin{itemize}
\item[(1)] Note that $\sF^{\kappa,G}_{\re}=\sF_{\re}^G$ because of $|\kappa|(G)=0$.  This implies $\cH^\kappa_F=\cH_F\cap L^2(E,|\kappa|)$.    Hence for $u\in \sF^\kappa_{\re}\subset \sF_\re$, $u-\bH_F u\in \sF^{\kappa,G}_\re$ and $\bH_F u\in \cH^\kappa_F$.  Particularly $\sF^\kappa_\re=\sF^{\kappa,G}_\re\oplus \cH^\kappa_F$ holds.  
\item[(2)] Write $\kappa$ for $\kappa|_F$ for convenience.  In view of \cite[Theorems 5.2.6 and 5.2.8]{CF12} and $|\kappa|(G)=0$,  we get $\kappa^\pm\in \check{\bS}$.  Hence $\kappa\in \check{\bS}-\check{\bS}$.  
\item[(3)] Note that $\check{\sF}^\kappa=\sF_\re|_F\cap L^2(F,|\kappa|  +\mu)=\check{\sF}\cap L^2(F,|\kappa|)$.  For $\varphi \in \check{\sF}^\kappa$,  it follows from the proof of the first assertion that $\bH_F^\kappa  \varphi=\bH_F \varphi\in \cH^\kappa_F$.  Hence for $\varphi, \phi\in \check{\sF}^\kappa$,  using $|\kappa|(G)=0$, we get
\[
	\check{\sE}^\kappa(\varphi, \phi)=\sE^\kappa(\bH^\kappa_F \varphi, \bH^\kappa_F \phi)=\sE^\kappa(\bH_F \varphi, \bH_F \phi)=\check{\sE}(\varphi, \phi)+\int_F \varphi \phi d\kappa.  
\]
Therefore $(\check{\sE}^\kappa,\check{\sF}^\kappa)$ is the perturbation of $(\check{\sE},\check{\sF})$ by $\kappa$.  
%\item[(4)]  These facts are obvious by the third assertion. 
\end{itemize}
That completes the proof.  
\end{proof}
%\begin{remark}
%Note that if $\kappa^-|_F(dx)=V^-(x)\mu(x)$ with $V^-\in L^\infty(F,\mu)$,  then $\kappa^-|_F$ is clearly $\check{\sE}$-form bounded (hence $\check{\sE}^{\kappa^+}$-form bounded).  Meanwhile \eqref{eq:47} holds true for $C_{\delta_0}=\|V^-\|_{L^\infty(F,\mu)}$.  
%\end{remark}

From now on we would write $\kappa$ for $\kappa|_F$ if no confusions caused.  Let $\sN_\kappa$ be the DN operator for $(\sE^\kappa,\sF^\kappa)$ on $L^2(F,\mu)$. On account of Proposition~\ref{THM33}~(3) and Lemma~\ref{LM43},  if $\kappa^-$ is $\check{\sE}^{\kappa^+}$-form bounded (for example,  $\kappa^-(dx)=\beta^-(x)\mu(x)$ with $\beta^-\in pL^\infty(F,\mu)$),  then $(\check{\sE}^\kappa,\check{\sF}^\kappa)$ is a lower bounded symmetric closed form on $L^2(F,\mu)$ whose generator is $-\sN_\kappa$.  
As a corollary,  we can obtain the following exchangeability of killing transformation and time change transformation.   

\begin{corollary}
Let $\kappa\in \bS$ such that $\kappa(G)=0$.  Then the trace Dirichlet form of $(\sE^\kappa,\sF^\kappa)$ on $L^2(F,\mu)$ is identified with the perturbed Dirichlet form of $(\check{\sE},\check{\sF})$ by $\kappa$.  
\end{corollary}

\subsection{Markov processes $h$-associated with DN operators}

Next we figure out the probabilistic counterpart of $\sN_\kappa$ under the assumption  that $\kappa^-$ is $\check{\sE}^{\kappa^+}$-form bounded.  
%The second case is the perturbation on the boundary,  i.e.  $\text{qsupp}[|\kappa|]\subset \Gamma$.  This amounts to $|\kappa|(\Omega)=0$.  In this case,  as mentioned in Proposition~\ref{THM33},   \eqref{eq:42-3} trivially holds true and if $\kappa^-$ is $\check{\sE}^{\kappa^+}$-form bounded,  then $\sN_\kappa$ is lower semi-bounded and self-adjoint.  

\begin{theorem}
Assume that $\kappa^-$ is $\check{\sE}^{\kappa^+}$-form bounded.  Then $(\check{\sE}^\kappa,\check{\sF}^\kappa)$ is a quasi-regular lower bounded positivity preserving (symmetric) coercive form on $L^2(F,\mu)$. 
\end{theorem}
\begin{proof}
The assumption implies that $(\check{\sE}^\kappa,\check{\sF}^\kappa)$ is a lower bounded closed form.  Take a constant $\alpha_0\geq 0$ such that $(\sA,\sG):=(\check{\sE}^\kappa_{\alpha_0},\check{\sF}^\kappa)$ is non-negative.  On account of Proposition~\ref{THM33}~(3),  one can obtain that $(\sA,\sG)$ is a non-negative positivity preserving (symmetric) coercive form.  To prove the quasi-regularity,  set
\[
	\tilde \sG:=\check{\sF}^{\kappa^+},\quad \tilde{\sA}(\varphi,\varphi):=\check{\sE}^{\kappa^+}_{\alpha_0}(\varphi,\varphi),\; \varphi \in \tilde\sG. 
\]
Then $(\tilde{\sA},\tilde{\sG})$ is the perturbed Dirichlet form of $(\check{\sE},\check{\sF})$ by $\tilde{\kappa}:=\kappa^++\alpha_0\cdot \mu\in \bS$.  Particularly $(\tilde{\sA},\tilde{\sG})$ is a quasi-regular Dirichlet form on  $L^2(F,\mu)$.  
Note that $\sG=\check{\sF}^\kappa=\check{\sF}^{\kappa^+}=\tilde{\sG}$ and for any $\varphi\in \sG=\tilde{\sG}$,
\[
	\sA_1(\varphi,\varphi)\leq \tilde{\sA}_1(\varphi,\varphi).  
\]
Therefore $(\sA,\sG)$ is quasi-regular by means of Lemma~\ref{LMD3}.  That completes the proof. 
\end{proof}

As discussed in \S\ref{SEC34},  for $\alpha\in \bR$ and $h\in \mathbf{E}^+_\alpha$, the $h$-transform $(\check{\sE}^{\kappa, h}_\alpha,\check{\sF}^{\kappa,h})$ of $(\check{\sE}^\kappa_\alpha, \check{\sF}^\kappa)$ is a quasi-regular lower bounded symmetric Dirichlet form on $L^2(F,h^2\cdot \mu)$ whose associated Markov process is called $(\alpha, h)$-associated with $\sN_\kappa$.  In case $(\check{\sE},\check{\sF})$ is irreducible,  $(\check{\sE}^{\kappa, h}_\alpha,\check{\sF}^{\kappa,h})$ is also irreducible. 

%We call $h$ the \emph{ground state} of $\sN_\kappa$.  Define the $h$-transform of $(\check{\sE}^\kappa,\check{\sF}^\kappa)$ (with $h:=h_\kappa$) as
%\begin{equation}\label{eq:5.3}
%	\begin{aligned}
%		&\check{\sF}^{\kappa, h}:=\{\varphi\in L^2(F, h_\kappa^2\cdot \mu): \varphi h_\kappa\in \check{\sF}^\kappa \},\\
%	&\check{\sE}^{\kappa,h}(\varphi,\varphi):=\check{\sE}^\kappa(\varphi h_\kappa, \varphi h_\kappa)-
%		\lambda^\kappa_1\int_F \varphi^2 h_\kappa^2 d\mu,\quad \varphi\in \check{\sF}^{\kappa,h}. 
%	\end{aligned}
%\end{equation}
%The following result figures out the probabilistic counterpart of $\sN_\kappa$.  

\subsection{Absolutely continuous case}

Consider a special case that $\kappa^-=\beta^-\cdot\mu$ with $\beta^-\in p L^\infty(F,\mu)$.  Clearly $\kappa^-$ is $\check{\sE}$-form bounded.  
The following corollary presents a simple way to obtain another Markov process related to $\sN_\kappa$,  which has been proposed in  \cite[\S4]{BV17}.  In a word,  for any constant $\alpha_0\geq \|\beta^-\|_{L^\infty(F,\mu)}$,  $-\sN_\kappa-\alpha_0$ is the $L^2$-generator of a certain Markov process.  

%In fact,  we have the following.
%\begin{lemma}\label{LM55}
%Assume that $|\kappa|(\Omega)=0$ and $\kappa^-=\beta^-\cdot \sigma$ with $\beta^-\in pL^\infty(F)$.  Then $\kappa^-$ is simultaneously $\sE$-form bounded and $\check{\sE}$-form bounded.  Particularly,  both $(\sE^\kappa,\sF^\kappa)$ and $(\check{\sE}^\kappa,\check{\sF}^\kappa)$ are lower bounded closed forms.  
%\end{lemma}
%\begin{proof}
%Clearly,  $\kappa^-$ is $\check{\sE}$-form bounded.  Note that $H^1(\Omega)$ is compactly embedded in $L^2(F)$.  Applying Lemma~\ref{LM45} to $(\sE,\sF)$ (with $\kappa^+=\mu=dx, \kappa^-=\|\beta^-\|_{L^\infty(F)}\cdot \sigma$),  we get that for any $1/2<\delta<1$,  there exists $C_\delta>0$ such that
%\[
%	\int_{F} (u|_F)^2d\kappa^-\leq \|\beta^-\|_{L^\infty(F)} \int_F (u|_F)^2d\sigma\leq \frac{\delta}{1-\delta} \sE_1(u,u)+C_\delta \int_\Omega u^2(x)dx.  
%\]
%Therefore $\kappa^-$ is $\sE$-form bounded.  That completes the proof. 
%\end{proof}

%In this subsection,  we explore the perturbation on the boundary.  Let $(\check{\sE},\check{\sF})$ be the trace Dirichlet form of $(\sE,\sF)$ on $L^2(F)$ and denote by $\check{\bS}_{EK}$ the extended Kato class with respect to $(\check{\sE},\check{\sF})$.   Thanks to Proposition~\ref{THM33},  we have the following.

\begin{corollary}\label{COR49}
Let $\kappa^-=\beta^-\cdot\mu$ with $\beta^-\in p L^\infty(F,\mu)$.  For any $\alpha_0\geq \|\beta^-\|_{L^\infty(F,\mu)}$,  $(\check{\sE}^\kappa_{\alpha_0},\check{\sF}^\kappa)$ is a quasi-regular Dirichlet form on $L^2(F,\mu)$,  whose generator is $-\sN_\kappa-\alpha_0$.  
%Then $-\sN_\kappa$ is the generator of the lower bounded symmetric closed form on $L^2(\Gamma)$:
%\[
%\begin{aligned}
%	\check{\sF}^\kappa&=H^{1/2}(\Gamma)\cap L^2(\Gamma, \kappa^+), \\
%	\check{\sE}^\kappa(\varphi, \phi)&=\check\sE(\varphi,\phi)+\int_\Gamma \varphi \phi d\kappa^+ -\int_\Gamma \varphi \phi \beta^- d\mu,\quad \varphi, \phi\in \check{\sF}^\kappa.
%\end{aligned}
%\]Particularly,   
\end{corollary}
\begin{proof}
On account of Proposition~\ref{THM33}~(3),  $(\check{\sE}^\kappa_{\alpha_0},\check{\sF}^\kappa)$ is the perturbed Dirichlet form of $(\check{\sE},\check{\sF})$ by $\tilde{\kappa}:=\kappa+\alpha_0\cdot \mu\in \bS$.  Hence $(\check{\sE}^\kappa_{\alpha_0},\check{\sF}^\kappa)$ is a quasi-regular Dirichlet form on $L^2(F,\mu)$.  That completes the proof. 
%Clearly,  $\kappa^-:=\beta^-\cdot \mu \in \check{\bS}_{EK}$.  
%It suffices to figure out the expression of $(\check{\sE}^\kappa,\check{\sF}^\kappa)$ and prove that $(\check{\sE}^\kappa_{\beta_0},\check{\sF}^\kappa)$ is a regular Dirichlet form.  Since $\check{\sF}=H^1(\Omega)|_\Gamma\cap L^2(\Gamma)=H^{1/2}(\Gamma)$,  we get $\check{\sF}^\kappa=\check{\sF}\cap L^2(\Gamma, |\kappa|)=H^{1/2}(\Gamma)\cap L^2(\Gamma, \kappa^+)$.  The expression of $\check{\sE}^\kappa$ is clear by that of $\check{\sE}$ in \eqref{eq:traceDirichletform}.  Note that $$\check{\sE}^\kappa_{\beta_0}(\varphi,\varphi)=\sE(\bH_F\varphi, \bH_F\varphi)+\int_\Gamma \varphi^2 d\left(\kappa^+ +(\beta_0-\beta^-)\mu\right)\geq 0$$ and $(\check{\sE}^\kappa_{\beta_0},\check{\sF}^\kappa)$ is in fact the perturbed Dirichlet form of $(\check{\sE},\check{\sF})$ by the positive Radon smooth measure $\kappa^++(\beta_0-\beta^-)\cdot \mu$.  Hence $(\check{\sE}^\kappa_{\beta_0},\check{\sF}^\kappa)$  is a regular Dirichlet form on $L^2(\Gamma)$.  
\end{proof}

%A concrete example for Corollary~\ref{COR49}  is $\kappa=\beta^+\cdot\mu-\beta^-\cdot\mu\in \bS-\bS$,  where $\beta^\pm$ are non-negative and $\beta^+\in L^1(\Gamma)$,  $\beta^-\in L^\infty(\Gamma)$.   
The following example recovering  \cite[\S4]{BV17} is related to the classical Robin boundary condition for the Laplacian on $\Omega$.  

\begin{example}\label{EXA65}
Adopt the notations in \S\ref{SEC5} and consider $\sF=H^1(\Omega)$ and $\sE(u,v):=\frac{1}{2}\bD(u,v)$ for $u,v\in \sF$ on $L^2(\bar{\Omega})$.  Take $\kappa:=\beta\cdot \sigma$ with $\beta\in L^\infty(\Gamma)$. 

It is easy to verify that $(\sE^\kappa,\sF^\kappa)$ is a lower bounded closed form on $L^2(\Omega)$ and its generator on $L^2(\Omega)$ is the so-called Robin Laplacian on $\Omega$:
\begin{equation}\label{eq:49}
\begin{aligned}
	&\cD(\Delta_\beta):=\{u\in H^1(\Omega):\Delta u\in L^2(\Omega) \text{ and }\frac{1}{2} \partial_\bn u +\beta \cdot u|_\Gamma=0\},\\
	&\Delta_\beta u:=\frac{1}{2}\Delta u,\quad u\in \cD(\Delta_\beta).  
\end{aligned}
\end{equation}
In addition,  the trace form $(\check{\sE}^\kappa,\check{\sF}^\kappa)$ is a lower bounded closed form on $L^2(\Gamma)$ whose generator is $-\sN_\kappa$.  One can verify that $\sN_\kappa$ is identified with
\[
	\begin{aligned}
		&\cD(\sN_\beta):=\cD(D),\\
		&\sN_\beta\varphi:=\frac{1}{2}\partial_\bn u +\beta\cdot \varphi,\quad \varphi\in \cD(\sN_\beta), 
		\end{aligned}
\]
where $D$ is the classical DN operator and $u$ is the harmonic extension of $\varphi$ appearing in Definition~\ref{DEF31}.  Note that $\sN_\beta$ is what is concerned with in \cite[\S4]{BV17}. 

The probabilistic counterpart of $\sN_\beta$ obtained in Corollary~\ref{COR49} can be described as follows.  Take a constant $\alpha_0\geq \|\beta^-\|_{L^\infty(\Gamma)}$ where $\beta^-:=-(0\wedge \beta)$.  Set $\tilde{\kappa}:=\kappa+\alpha_0\cdot \sigma\in \mathring{\bS}$ and let $X^{\tilde{\kappa}}$ be the subprocess of $X$ killed by the PCAF corresponding to $\tilde{\kappa}$.   Then $-\sN_\beta-\alpha_0$ is the $L^2$-generator of the time changed process $\check{X}^{\tilde \kappa}$ of $X^{\tilde{\kappa}}$ by the PCAF corresponding to $\sigma$.  

Note incidentally that the $L^2$-semigroup associated with $\sN_\beta$ is irreducible due to the irreducibility of $(\check{\sE},\check{\sF})$.  Mimicking Theorem~\ref{THM411},  one can obtain that $\sN_\beta$ admits a strictly positive ground state,  i.e.  $\sN_\beta h_\beta=\lambda^\beta_1 h_\beta$ where $\lambda^\beta_1=\text{min }\sigma(\sN_\beta)$ is a simple eigenvalue and $h_\beta\in L^2(\Gamma)$,  $h_\beta>0$,  $\sigma$-a.e.  In addition the analogical results of Theorem~\ref{THM412} holds true for $\sN_\beta$.  Particularly $\lambda^\beta_1+\alpha_0\geq 0$ and $h_\beta\in \mathbf{E}^+_{-\lambda^\beta_1}\subset \mathbf{E}^+_{\alpha_0}$.  Let $(\check{\sE}^{\kappa, h_\beta}_{\alpha_0},\check{\sF}^{\kappa,h_\beta})$ be the $h$-transform of $(\check{\sE}^\kappa_{\alpha_0}, \check{\sF}^\kappa)$ with $h=h_\beta$.  
Then $(\check{\sE}^{\kappa, h_\beta}_{\alpha_0},\check{\sF}^{\kappa,h_\beta})$ is a quasi-regular and irreducible symmetric Dirichlet form on $L^2(\Gamma, h^2_\beta\cdot \sigma)$,  whose associated Markov process is the $h$-transform of $\check{X}^{\tilde{\kappa}}$ with $h=h_\beta$ as studied in, e.g.,  \cite{Y98}. 
\end{example}

\subsection{Calder\'on's problem}\label{SEC43-2}
Finally,  we turn to give some remarks on the Calder\'on's problem related to $(\sE,\sF)$.  It is an inverse problem considering whether one can determine the perturbation if knowing the trace information on certain boundary;  see Appendix~\ref{SEC6}.  
%We show the following uniqueness result for 
% Take $\kappa_i\in \bS-\bS$ for $i=1,2$.  Let $C_{\kappa_i}$ be the Cauchy data and $\sN_{\kappa_i}$ be the DN operator (if $\sF^{\kappa_i,G}_\re\cap \cH^{\kappa_i}_F=\{0\}$) for $(\sE^{\kappa_i},\sF^{\kappa_i})$;  see Definition~\ref{DEF41}.  Then the \emph{Calder\'on's problem} (related to $\sE,  F$ and $\mu$) is to ask if $C_{\kappa_1}=C_{\kappa_2}$ or $\sN_{\kappa_1}=\sN_{\kappa_2}$ leads to $\kappa_1=\kappa_2$. 
 %$\sN_{\kappa_i}$ be the DN operator of $(\sE^{\kappa_i},\sF^{\kappa_i})$ relative to $F$ and $\mu$ for $i=1,2$  (see Definition~\ref{DEF41}). 
%\begin{remark}\label{RM51}
%When $\kappa_i\in \mathring{\bS}$,  the Calder\'on's problem can be restated in a more probabilistic manner: Let $X$ be the Markov process   associated with $(\sE,\sF)$ and $X^{\kappa_i}$ be the subprocess of $X$ obtained by killing by the PCAF corresponding to $\kappa_i$.  Then it is to ask if we can get $\kappa_1=\kappa_2$ by supposing that the time changed processes of $X^{\kappa_1}$ and $X^{\kappa_2}$ by the PCAFs corresponding to the same $\mu$ are identified.  
%\end{remark}
The following result states that if the perturbation is supported on boundary, then the uniqueness for the Calder\'on's problem holds true.  %Recall that the extended Kato class $\check{\bS}_{EK}$ appeared in Proposition~\ref{THM33}.  Particularly if $\kappa_i^-|_F=V_i^-\cdot \mu$ with $V^-_i\in L^\infty(F,\mu)$,  then $\kappa^-_i|_F\in \check{\bS}_{EK}$.  
%and the second condition below guarantees that $\sN_{\kappa_i}$ is self-adjoint.  

\begin{theorem}
Take $\kappa_i\in \bS-\bS$ such that $|\kappa_i|(G)=0$ for $i=1,2$ and let $(\check{\sE}^{\kappa_i},\check{\sF}^{\kappa_i})$ be the trace form of $(\sE^{\kappa_i},\sF^{\kappa_i})$.   Then $(\check{\sE}^{\kappa_1},\check{\sF}^{\kappa_1})=(\check{\sE}^{\kappa_2},\check{\sF}^{\kappa_2})$ implies $\kappa_1=\kappa_2$.  %If in addition $\kappa_i^-|_F\in \check{\bS}_{EK}$,  then $\sN_{\kappa_1}=\sN_{\kappa_2}$ implies $\kappa_1=\kappa_2$.
\end{theorem}
\begin{proof}
%We would write $\kappa_i,  |\kappa_i|,  \kappa_i^\pm$ for their restrictions to $F$ if no confusions caused.  
Note that $\check{\sF}^{\kappa_1}=\check{\sF}^{\kappa_2}$ and Proposition~\ref{THM33}~(3) imply that $$\sG:=\check{\sF}\cap L^2(F,|\kappa_1|)=\check{\sF}\cap L^2(F,|\kappa_2|)=\check{\sF}\cap L^2(F,|\kappa_1|+|\kappa_2|).$$ 
 Take $\varphi, \phi\in \sG$.  It follows from $\check{\sE}^{\kappa_1}(\varphi, \phi)=\check{\sE}^{\kappa_2}(\varphi, \phi)$ that 
\begin{equation}\label{eq:44}
\check{\sE}(\varphi, \phi)+\int_F \varphi \phi d(\kappa^+_1+\kappa^-_2)=\check{\sE}(\varphi, \phi)+\int_F \varphi \phi d(\kappa^-_1+\kappa^+_2).  
\end{equation}
Set $\tilde{\kappa}_1:=\kappa_1^++\kappa^-_2$ and $\tilde{\kappa}_2:=\kappa^-_1+\kappa^+_2$.  Then $\tilde{\kappa}_1,\tilde{\kappa}_2\in \bS$.  Denote by $(\check{\sE}^{\tilde{\kappa}_i},\check{\sF}^{\tilde{\kappa}_i})$ the perturbed Dirichlet form of $(\check{\sE},\check{\sF})$ by $\tilde{\kappa}_i$ for $i=1,2$.  Clearly,  $\sG\subset \check{\sF}^{\tilde{\kappa}_1} \cap \check{\sF}^{\tilde{\kappa}_2}$.
We claim that $\sG$ is $\check{\sE}^{\tilde{\kappa}_i}_1$-dense in $\check{\sF}^{\tilde{\kappa}_i}$.  To accomplish this,  using \cite[Theorem~2.2.4]{FOT11}  we can take an $\check{\sE}$-nest $\{F_n:n\geq 1\}$ such that $|\kappa_1|(F_n)+|\kappa_2|(F_n)<\infty$.  Then
\[
	\tilde{\sG}:=\bigcup_{n\geq 1} b\check\sF_{F_n} \subset \sG,
\]
where $b\check\sF_{F_n}$ is the family of all bounded functions in $$\check\sF_{F_n}=\{u\in \check\sF: u=0, \check\sE\text{-q.e. on }F^c_n\}.$$
On account of \cite[Theorem~5.1.4]{CF12},  $\{F_n:n\geq 1\}$ is also an $\check\sE^{\tilde{\kappa}_i}$-nest.  Hence $\tilde\sG$ as well as $\sG$ is $\check{\sE}^{\tilde{\kappa}_i}_1$-dense in $\check{\sF}^{\tilde{\kappa}_i}$.  This assertion,  together with \eqref{eq:44}, tells us that $(\check{\sE}^{\tilde{\kappa}_1},\check{\sF}^{\tilde{\kappa}_1})=(\check{\sE}^{\tilde{\kappa}_2},\check{\sF}^{\tilde{\kappa}_2})$.  Eventually  applying the Beurling-Deny formula to them,  we get $\tilde{\kappa}_1=\tilde{\kappa}_2$.  Therefore $\kappa_1=\kappa_2$.  That completes the proof. 
%Now assume further that $\kappa_i^-|_F\in \check{\bS}_{EK}$.  
%On account of Proposition~\ref{THM33},  $-\sN_{\kappa_i}$ is the generator of the lower bounded symmetric closed form $(\check{\sE}^{\kappa_i},\check{\sF}^{\kappa_i})$.  Particularly,  $\sN_{\kappa_1}=\sN_{\kappa_2}$ amounts to $(\check{\sE}^{\kappa_1},\check{\sF}^{\kappa_1})=(\check{\sE}^{\kappa_2},\check{\sF}^{\kappa_2})$ leading to $\kappa_1=\kappa_2$.  
\end{proof}

\appendix

\section{Perturbations of Dirichlet forms}\label{APPB}

Let $(\sE,\sF)$ be a regular Dirichlet form on $L^2(E,m)$ and $\kappa:=\kappa^+-\kappa^-\in \bS-\bS$.  Denote by $A^{\kappa^\pm}:=(A^{\kappa^\pm}_t)_{t\geq 0}$ the PCAF corresponding to $\kappa^\pm$ and set $A^\kappa:=A^{\kappa^+}-A^{\kappa^-}$.  

\begin{definition}\label{DEFB1}
Let $\kappa=\kappa^+-\kappa^-\in \bS-\bS$.  The following quadratic form
\[
\begin{aligned}
	\sF^\kappa&:=\sF\cap L^2(E,|\kappa|), \\
	\sE^\kappa(u,v)&:=\sE(u,v)+\int_E uvd\kappa^+-\int_E uvd\kappa^-,\quad u,v\in \sF^\kappa
\end{aligned}
\]
is called \emph{the perturbation of $(\sE,\sF)$ by $\kappa$}.  
\end{definition}

When $\kappa\in \bS$ (i.e. $\kappa^-=0$),  $(\sE^\kappa,\sF^\kappa)$ is a quasi-regular Dirichlet form on $L^2(E,m)$ associated with the subprocess of $X$ obtained by killing by $A^\kappa$;  see,  e.g.,  \cite[\S5.1]{CF12}.  Meanwhile it is also called the \emph{perturbed Dirichlet form of $(\sE,\sF)$ by $\kappa$}.  % Particularly if $\kappa\in \mathring{\bS}$, then $(\sE^\kappa,\sF^\kappa)$ is regular on $L^2(E,m)$.  
In general if $\kappa^-\neq 0$,  then $(\sE^\kappa,\sF^\kappa)$ is not necessarily closed.  The following definition gives a well-known sufficient condition for the closedness of $(\sE^\kappa,\sF^\kappa)$.  

\begin{definition}\label{DEFA2}
Let $\kappa=\kappa^+-\kappa^-\in \bS-\bS$.  Then $\kappa^-$ is called \emph{$\sE^{\kappa^+}$-form bounded} (with $\sE^{\kappa^+}$-form bound less than $1$)  if $L^2(E,\kappa^-)\subset \sF^{\kappa^+}$ and there exist constants $0\leq \delta<1$ and $C_\delta\geq 0$ such that 
\begin{equation}\label{eq:A1}
	\int_E u^2 d\kappa^-\leq \delta\cdot \sE^{\kappa^+}(u,u)+C_\delta \int_E u^2dm,\quad \forall u\in \sF^{\kappa^+}.  
\end{equation}
The infimum of all possible constants $\delta$ for which the inequality holds is called the \emph{$\sE^{\kappa^+}$-bound} of $\kappa^-$.  
\end{definition}
\begin{remark}
A famous family of smooth measures that are $\sE$-form bounded (clearly also $\sE^{\kappa^+}$-form bounded for any $\kappa^+\in \bS$) is the so-called extended Kato class containing $\left\{V\cdot m: V\in pL^\infty(E,m)\right\}$; see, e.g.,  \cite{AM92, SV96}. 
\end{remark}

Clearly if $\kappa^-$ is $\sE^{\kappa^+}$-form bounded,  then $(\sE^\kappa,\sF^\kappa)$ is a lower bounded symmetric closed form in the sense that $(\sE^\kappa_{\alpha_0}, \sF^\kappa)$ for $\alpha_0=C_\delta$ appearing in \eqref{eq:A1} is a non-negative symmetric closed form on $L^2(E,m)$.  Meanwhile $\sF^\kappa=\sF\cap L^2(E,\kappa^+)$ and the $L^2$-semigroup of $(\sE^\kappa,\sF^\kappa)$ is the so-called Feymann-Kac semigroup $$P^\kappa_t f(x)=\mathbf{E}_x\left[e^{-A^\kappa_t}f(X_t) \right]$$ satisfying $\|P^\kappa_t f\|_{L^2(E,m)}\leq e^{\alpha_0 t}\|f\|_{L^2(E,m)}$,  whose generator is upper semi-bounded and self-adjoint on $L^2(E,m)$.

%Thanks to Lemma~\ref{LM31},  it is easy to verify that for $\kappa\in \bS-\bS_{EK}$,  $(\sE^\kappa,\sF^\kappa)$ is a lower bounded symmetric closed form in the sense that there exists a constant $\beta_0>0$ such that $(\sE^\kappa_{\beta_0}, \sF^\kappa)$ is a non-negative symmetric closed form.  Meanwhile $\sF^\kappa=\sF\cap L^2(E,\kappa^+)$ and $(\sE^\kappa,\sF^\kappa)$ corresponds to the so-called Feymann-Kac semigroup $$P^\kappa_t f(x)=\mathbf{E}_x\left[e^{-A^\kappa_t}f(X_t) \right]$$ satisfying $\|P^\kappa_t f\|_{L^2(E,m)}\leq e^{\beta_0 t}\|f\|_{L^2(E,m)}$,  whose generator, denoted by $\sL_\kappa$ (or $\sL+\kappa$ in abuse of notation) is an upper semi-bounded and self-adjoint operator on $L^2(E,m)$.  %Meanwhile,  $(\sE^\kappa,\sF^\kappa)$ is called the \emph{perturbation of $(\sE,\sF)$ by $\kappa$}.  

\section{Direct sum decomposition for perturbation of Dirichlet forms}\label{APP2}

Adopt the same notations as in Appendix~\ref{APPB}.  Assume that $\kappa^-$ is $\sE^{\kappa^+}$-form bounded,  so that $(\sE^\kappa,\sF^\kappa)$ is lower bounded and closed on $L^2(E,m)$.  
In this appendix we set up the analogue of \eqref{eq:Fe} for $(\sE^\kappa,\sF^\kappa)$.  %From now on let $(\sE,\sF)$ be a regular Dirichlet form on $L^2(E,m)$. Fix $\kappa\in \bS-\bS$ and denote by $(\sE^\kappa,\sF^\kappa)$ the perturbation of $(\sE,\sF)$ by $\kappa$.  

Take a quasi closed set $F\subset E$ with respect to $\sE$ and set $G=E\setminus F$.  Set further
\[
	\sF^\kappa_\re:=\sF_\re\cap L^2(E,|\kappa|),\quad \sF^{\kappa,G}_{\re}:=\{u\in \sF^\kappa_\re:  u=0,  \sE\text{-q.e.  on }F\}
\]
and 
\[
	\cH^\kappa_F:=\{u\in \sF^\kappa_\re: \sE^\kappa(u,v)=0,\forall v\in \sF^{\kappa,G}_{\re}\}.  
\] 
Write $\sF^{\kappa,G}:=\sF^{\kappa,G}_{\re}\cap L^2(G,m|_G)$ and $\sE^{\kappa,G}(u,v):=\sE^\kappa(u,v)$ for $u,v\in \sF^{\kappa, G}$.  Note that $(\sE^{\kappa,G}, \sF^{\kappa,G})$ is lower bounded and closed on $L^2(G,m|_G)$.  Denote the $L^2$-generator of $(\sE^{\kappa,G}, \sF^{\kappa,G})$ by $\sL_{\kappa,G}$.  The spectrum of $\sL_{\kappa,G}$ is denoted by $\sigma(\sL_{\kappa,G})$.  

\begin{lemma}\label{LMB1}
Assume that $0\notin \sigma(\sL_{\kappa, G})$.  Then it holds 
\begin{equation}
\sF^\kappa=\sF^{\kappa,G}\oplus \tilde\cH^\kappa_F,
\end{equation}
where $\tilde{\cH}^\kappa_F:=\{u\in \sF^\kappa: \sE^\kappa(u,v)=0,\forall v\in \sF^{\kappa,G}\}$.  
Particularly,  if $\sF^{\kappa,G}_{\re}\subset L^2(G,m|_G)$,  then 
\begin{equation}\label{eq:31-2}
	\sF^\kappa_\re=\sF^{\kappa,G}_\re\oplus \cH^\kappa_F.
\end{equation}
\end{lemma}
\begin{proof}
Note that $\sF^{\kappa,G}$ is a Hilbert space under the inner produce $\sE^{|\kappa|}_1(u,v):=\sE(u,v)+\int_E uv d(m+|\kappa|)$.  %In fact,  for $u\in \sF^{\kappa,G}_{\re}$, $\sE^{|\kappa|}(u,u)=0$ implies $\sE(u,u)=0$ and hence $u=0$ because of the irreducibility of $(\sE,\sF)$ and $u=0$,  $\sE$-q.e. on $F$.  
Denote by $\left( \sF^{\kappa,G}\right)'$ the family of all continuous linear functionals on $\sF^{\kappa,G}$.  Consider an operator $$\sA:\cD(\sA)= \sF^{\kappa,G}\left(\subset \left( \sF^{\kappa,G}\right)'\right)\rightarrow \left( \sF^{\kappa,G}\right)'$$ defined as follows:  For any $u\in \sF^{\kappa,G}$, 
\[
	\langle \sA u,  v\rangle:=\sE^\kappa(u,v),\quad \forall v\in  \sF^{\kappa,G}.    
\]
It is easy to verify that $L^2(G,m|_G)$ is continuously embedded in $\left( \sF^{\kappa,G}\right)'$,  i.e.  for any $f\in L^2(G,m|_G)$,  $I_f: g\mapsto \int_G fgdm$ defines a continuous linear functional on $\sF^{\kappa,G}$ and $\|I_f\|_{\left( \sF^{\kappa,G}\right)'}\lesssim \|f\|_{L^2(G,m|_G)}$.  In addition,  the part $\sA_0$ of $\sA$ in $L^2(G,m|_G)$,  i.e.  $\cD(\sA_0):=\{u\in \cD(\sA)\cap L^2(G,m|_G): \sA u\in L^2(G,m|_G)\}$,  $\sA_0 u:=\sA u$,  is identified with $-\sL_{\kappa, G}$.  Applying \cite[Proposition~3.10.3]{AB01},  we get that $\sigma(\sA)=\sigma(-\sL_{\kappa,G})$.  Since $0\notin \sigma(-\sL_{\kappa,G})=\sigma(\sA)$,  it follows that $\sA$ is invertible.  

Take $u\in \sF^\kappa$.  Set $F_u(v):=\sE^\kappa(u,v)$ for any $v\in \sF^{\kappa,G}$.  It is easy to verify that $F_u\in \left( \sF^{\kappa,G}\right)'$.  Since $\sA$ is invertible,  there exists $u_1\in \sF^{\kappa,G}$ such that $F_u=\sE^\kappa(u_1,\cdot)$.  Let $u_2:=u-u_1\in \sF^\kappa$.  Then $\sE^\kappa(u_2,v)=F_u(v)-\sE^\kappa(u_1,v)=0$ for any $v\in \sF^{\kappa,G}$.  This implies $u_2\in \tilde{\cH}^\kappa_F$.  In addition,  $0\notin \sigma(\sL_{\kappa, G})$ also leads to $\sF^{\kappa,G}\cap \tilde{\cH}^\kappa_F=\{0\}$,  because any $u\in \sF^{\kappa,G}\cap \tilde{\cH}^\kappa_F$ satisfies $\sL_{\kappa,G}u=0$.  

If $\sF^{\kappa,G}_{\re}\subset L^2(G,m|_G)$,  then $\sF^{\kappa,G}_\re=\sF^{\kappa,G}$.  For $u\in \sF^\kappa_\re$,  we still have that $F_u=\sE^\kappa(u,\cdot)\in \left( \sF^{\kappa,G}\right)'$.  Hence there exists $u_1\in \sF^{\kappa,G}=\sF^{\kappa,G}_\re$ such that $F_u=\sE^\kappa(u_1,\cdot)$.  It follows that $u_2:=u-u_1\in \sF^\kappa_\re$ and $\sE^\kappa(u_2,v)=0$ for any $v\in \sF^{\kappa,G}=\sF^{\kappa,G}_\re$.  Particularly,  $u_2\in \cH^\kappa_F$.  On the other hand,  
\[
	\sF^{\kappa,G}_\re\cap \cH^\kappa_F=\sF^{\kappa,G}\cap \cH^\kappa_F=\sF^{\kappa,G}\cap \tilde\cH^\kappa_F=\{0\}.
\]
That completes the proof.  
\end{proof}
%\begin{remark}
%Since $\kappa^-$ is $\sE^{\kappa^+}$-form bounded,  there exists a constant $\beta_0>0$ such that $(\sE^\kappa_{\beta_0}, \sF^\kappa)$ is a closed form on $L^2(E,m)$.  Particularly,  $\sF^\kappa$ is a Hilbert space under the inner product $\sE^\kappa_{\beta}$ for any $\beta>\beta_0$.  Note that $\sF^{\kappa,G}$ is a closed subspace of $\sF^\kappa$.  Hence the following direct sum decomposition is true:
%\[
%	\sF^\kappa=\sF^{\kappa,G}\oplus \cH_F^{\kappa,\beta},
%\]
%where $\cH_F^{\kappa,\beta}:=\{u\in \sF^\kappa: \sE^\kappa_\beta(u,v)=0,\forall v\in \sF^{\kappa,G}\}$.  

%If $(\sE,\sF)$ is irreducible and $\kappa\in \bS$,  then \eqref{eq:31-2} trivially holds due to the irreducibility of $(\sE^\kappa,\sF^\kappa)$.  
%\end{remark}

\section{Irreducibility of lower bounded closed forms}

Let $(\sE,\sF)$ be a lower bounded symmetric closed form on $L^2(E,m)$,  i.e.  there exists a constant $\alpha_0\geq 0$ such that $(\sE_{\alpha_0},\sF)$ is a non-negative symmetric closed form on $L^2(E,m)$.  Let $(T_t)_{t\geq 0}$ be the $L^2$-semigroup associated with $(\sE,\sF)$.  %Then $\|T_t\|\leq \re^{\beta_0 t}$ for any $t\geq 0$.  
%Write $L^2_+(E,m):=\{f\in L^2(E,m):f\geq 0\}$. 

\begin{definition}\label{DEFC1}
\begin{itemize}
\item[(1)] The semigroup $(T_t)_{t\geq 0}$ is \emph{positive} if $T_t f\in pL^2(E,m)$ for all $t\geq 0$ and $f\in pL^2(E,m)$.  It is called \emph{irreducible} if for every $t>0$ and non-zero function $f\in pL^2(E,m)$,  we have $T_t f(x)>0$ for $m$-a.e.  $x\in E$.  
\item[(2)] A subset $A\subset E$ is called \emph{invariant} if $1_{A^c}\cdot T_t(f1_A)=0$,  $m$-a.e.  for any $t\geq 0$ and $f\in L^2(E,m)$.  
\end{itemize}
\end{definition}

There are several equivalent conditions on $(\sE,\sF)$ to the positivity of $(T_t)_{t\geq 0}$.  For example,  it is equivalent to
\begin{equation}\label{eq:C1}
	f\in \sF \Longrightarrow f^+, f^-\in \sF \text{ and }\sE(f^+,f^-)\leq 0,
\end{equation}
where $f^+:=f\vee 0$ and $f^-:=f^+-f$;  see,  e.g.,  \cite{MR95} for other equivalent conditions.   

\begin{theorem}\label{THMC2}
Assume that $(T_t)_{t\geq 0}$ is positive.  The following are equivalent:
\begin{itemize}
\item[(1)] $A\subset E$ is invariant;
\item[(2)] $f 1_A\in \sF$ and $\sE( f1_A,   f1_{A^c})\geq 0$ for all $f\in \sF$;
\item[(3)] $f 1_A\in \sF$ and $\sE(f1_A,  f1_{A^c})=0$ for all $f\in \sF$;
\item[(4)] $f 1_A\in \sF$ and $\sE(f,f)=\sE(f1_A, f1_A)+\sE(f1_{A^c}, f1_{A^c})$ for all $f\in \sF$. 
\end{itemize}
Further,  $(T_t)_{t\geq 0}$ is irreducible,  if and only if $m(A)=0$ or $m(A^c)=0$ for any invariant set $A\subset E$. 
\end{theorem}
\begin{proof}
The equivalence between (1) and (2) and the characterization for the irreducibility of $(T_t)_{t\geq 0}$ are the consequences of \cite[Theorems~2.9 and 2.10]{O05}.  Clearly,  (3) and (4) are equivalent,  and (3) implies (2).  Now suppose that (2) holds.  Then for $f\in\sF$,   $u:=f1_A-f1_{A^c}\in \sF$ and $\sE(u1_A, u 1_{A^c})\geq 0$,  which amounts to
\[
	\sE(f1_A, -f1_{A^c})\geq 0. 
\]
Therefore $\sE(f1_A, f1_{A^c})=0$ and (3) holds true.  That completes the proof. 
\end{proof}

In view of this theorem,  when $(\sE,\sF)$ is a Dirichlet form,  the irreducibility of $(T_t)_{t\geq 0}$ coincides with that of $(\sE,\sF)$;  see,  e.g.,  \cite[\S1.6]{FOT11}. 

\section{Quasi-regular positivity preserving (symmetric) coercive forms}\label{APPD}

In this appendix we review the basic conceptions about quasi-regular positivity preserving (symmetric) coercive forms raised in \cite{MR95}.  Note that only symmetric cases are under consideration in our paper.

Let $(\sE,\sF)$ be a non-negative symmetric closed form on $L^2(E,m)$,  whose associated $L^2$-semigroup $(T_t)_{t\geq 0}$ is positive in the sense of Definition~\ref{DEFC1}.  This positivity preserving property amounts to \eqref{eq:C1}.  Meanwhile $(\sE,\sF)$ is also called a \emph{positivity preserving (symmetric) coercive form}.  For a closed set $F\subset E$,  we set
\[
	\sF_F:=\{u\in \sF: u=0,  m\text{-a.e.  on }E\setminus F\}.
\]
Then an increasing sequence $\{F_n:n\geq 1\}$ of closed subsets of $E$ is called an \emph{$\sE$-nest} if $\cup_{n\geq }\sF_{F_n}$ is $\sE_1$-dense in $\sF$.  A subset $N\subset E$ is called \emph{$\sE$-polar} if $N\subset \cap_{n\geq 1}(E\setminus F_n)$ for some $\sE$-nest $\{F_n:n\geq 1\}$.  We say that a property of points in $E$ holds $\sE$-quasi-everywhere (abbreviated $\sE$-q.e.),  if the property holds outside some $\sE$-polar set.  Given an $\sE$-nest $\{F_n:n\geq 1\}$,  define
\[
	C(\{F_n\}):=\{f:A\rightarrow \bR: \cup_{n\geq 1}F_n\subset A\subset E, f|_{F_n}\text{ is continuous for each }n\}.
\]
Then an $\sE$-q.e. defined function $u$ on $E$ is called \emph{$\sE$-quasi-continuous} if there exists an $\sE$-nest $\{F_n:n\geq 1\}$ such that $u\in C(\{F_n\})$.  For $\alpha\geq 0$,  $u\in L^2(E,m)$ is called \emph{$\alpha$-excessive} if $u\geq 0$,  $m$-a.e.  and $\re^{-\alpha t}T_tu\leq u$,  $m$-a.e. for all $t>0$.  

\begin{definition}\label{DEFD1}
A positivity preserving (symmetric) coercive form $(\sE,\sF)$ is called \emph{quasi-regular} if:
\begin{itemize}
\item[(i)] There exists an $\sE$-nest $\{K_n:n\geq 1\}$ consisting of compact sets.
\item[(ii)] There exists an $\sE_1$-dense subset of $\sF$ whose elements have $\sE$-quasi-continuous $m$-versions.
\item[(iii)] There exist $u_n\in\sF$,  $n\in \mathbb{N}$,  having $\sE$-quasi-continuous $m$-versions $\tilde{u}_n$,  $n\in \mathbb{N}$, and an $\sE$-polar set $N\subset E$ such that $\{\tilde{u}_n:n\in\mathbb{N}\}$ separates the point of $E\setminus  N$.
\item[(iv)] There exists an $\sE$-q.e. strictly positive $\sE$-quasi-continuous $m$-version $h$ of an $\alpha$-excessive function in $\sF$ for some $\alpha\in (0,\infty)$.
\end{itemize}
\end{definition}

Let $h\in L^2(E,m)$,  $h>0$,  $m$-a.e.  Define
\[
\begin{aligned}
		&\sF^h:=\{u\in L^2(E,h^2\cdot m): uh\in \sF\}, \\
		&\sE^h(u,v):=\sE(uh,vh),\quad u,v\in \sF^h,
\end{aligned}
\]
called the \emph{$h$-transform of $(\sE,\sF)$}. 
The following theorem is taken from \cite[Theorem~4.14]{MR95}.  The Markov process associated with the $h$-transform $(\sE^h,\sF^h)$ in this theorem is called \emph{$h$-associated with $(\sE,\sF)$.} %It is worth noting that the quasi-regularity is also the necessary condition for $(\sE,\sF)$ to be $h$-associated with a certain Markov process;  see \cite[Theorem~5.2]{MR95}. 

\begin{theorem}\label{THMD2}
A positivity preserving (symmetric) coercive form $(\sE,\sF)$ on $L^2(E,m)$ is quasi-regular if and only if for one (hence every) $m$-a.e. strictly positive function $h\in \sF$, which is $\alpha$-excessive for some $\alpha>0$,  there exists an $\sE$-q.e. strictly positive $\sE$-quasi-continuous $m$-version $\tilde{h}$ and the corresponding $h$-transform $(\sE^h,\sF^h)$ is a quasi-regular symmetric Dirichlet form.  
\end{theorem}

We give a simple but useful sufficient condition for the quasi-regularity of $(\sE,\sF)$.

\begin{lemma}\label{LMD3}
Let $(\sE',\sF')$ be a quasi-regular symmetric Dirichlet form on $L^2(E,m)$.  Assume that 
\begin{equation}\label{eq:D1}
	\sF'=\sF,\quad \sE_1(u,u)\leq C \sE'_1(u,u),\;\forall u\in \sF,
\end{equation}
for some  constant $C>0$. 
Then $(\sE,\sF)$ is quasi-regular. 
\end{lemma}
\begin{proof}
The condition \eqref{eq:D1} implies that an $\sE'$-nest is also an $\sE$-nest.  Hence an $\sE'$-polar set is also $\sE$-polar and an $\sE'$-quasi-continuous function is also $\sE$-quasi-continuous.  On account of the quasi-regularity of $(\sE',\sF')$,  one can easily verify that Definition~\ref{DEFD1}(i-iii) and \cite[(4.8)]{MR95} are satisfied by $(\sE,\sF)$.   Take  $\varphi\in L^2(E,m)$ with $\varphi>0$, $m$-a.e., and set $u:=\int_0^\infty \re^{-\alpha t}T_t\varphi dt$.  Then $u$ is $\alpha$-excessive and $u>0$,  $m$-a.e. due to \cite[Lemma~3.6]{MR95}.  Note that $u\in \sF=\sF'$.  Thus $u$ admits an $\sE'$-quasi-continuous $m$-version $\tilde{u}$,  which is also $\sE$-quasi-continuous.  Clearly $u\leq \tilde{u}$,  $m$-a.e.  Consequently,  \cite[Assumption~4.6]{MR95} is satisfied by $(\sE,\sF)$.  In view of \cite[Proposition~4.11]{MR95},  we can eventually conclude the quasi-regularity of $(\sE,\sF)$.  
\end{proof}

\section{Classical Calder\'on's problem}\label{SEC6}

%The classical Calder\'on's problem reads as follows:   Let $V_i\in L^\infty(\Omega)$ and $D_{V_i}$ be the DN operator for the Sch\"odinger operator $\frac{1}{2}\Delta +V_i$ on $L^2(\Gamma)$  for $i=1,2$ (see Definition~\ref{DEF45}).  Then it is to ask if $D_{V_1}=D_{V_2}$ implies $V_1=V_2$.  

%Usually self-adjoint DN operators will be under consideration.  Meanwhile $\sN_{\kappa_1}=\sN_{\kappa_2}$ amounts to that their associated closed forms on $L^2(F,\mu)$ are identified.  

In this appendix we review some results for the classical Calder\'on problem (see \cite{U12, SU87, C80}).  Let $\Omega$ be a bounded domain with smooth boundary and consider the Dirichlet form $(\sE,\sF)=(\frac{1}{2}\bD, H^1(\Omega))$ on $L^2(\bar{\Omega})$.  Set $G=\Omega$, $F=\Gamma:=\partial \Omega$ and $\mu=\sigma$.   Take $\kappa_i(dx)=V_i(x)dx$ with $V_i\in L^\infty(\Omega)$ for $i=1,2$.  
We do not assume \eqref{eq:42-3} and let 
\begin{equation}\label{eq:E0}
\begin{aligned}
	C_{V_i}:=&\big\{(\varphi, f)\in L^2(\Gamma)\times L^2(\Gamma):  \exists\,u\in \cH^{\kappa_i}_\Gamma\text{ such that }u|_\Gamma=\varphi, \\
	&\qquad\qquad \sE^{\kappa_i}(u,v)=\int_{\Gamma} f v|_\Gamma d\sigma \text{ for any }v\in \sF^{\kappa_i}_\re\text{ with }v|_\Gamma\in L^2(\Gamma)\big\},
\end{aligned}\end{equation}
called the \emph{Cauchy data} for $(\sE^{\kappa_i},\sF^{\kappa_i})$ on $L^2(\Gamma)$.  Note that if \eqref{eq:42-3} holds true,  then $C_{V_i}$ is the graph of $\sN_{\kappa_i}$. The uniqueness problem for classical Calder\'on problem considers whether $C_{V_1}=C_{V_2}$ implies $V_1=V_2$.  

%Take
%\[
%	\sF^{V_i}_\re:=\sF^{\kappa_i}_\re,\quad \sF^{V_i,\Omega}_\re:=\sF^{\kappa_i,\Omega}_\re,\quad \cH^{V_i}_\Gamma:=\cH^{\kappa_i}_\Gamma,  \quad C_{V_i}:=C_{\kappa_i}.  
%\]
%Clearly $\sF^{V_i}_\re=H^1(\Omega)$ and $\sF^{V_i,\Omega}_\re=H^1_0(\Omega)$.  
%Denote by $D_{V_i}:=\sN_{\kappa_i}$ the DN operator for the Sch\"odinger operator $\frac{1}{2}\Delta +V_i$ on $L^2(\Gamma)$ whenever $\sF^{V_i,\Omega}_\re\cap \cH^{V_i}_\Gamma=\{0\}$.  Note that if $0\notin \sigma(\sL_{V_i,\Omega})$,  where $\sL_{V_i,\Omega} u=\frac{1}{2}\Delta u- V_i\cdot u$ with $$\cD(\sL_{V_i,\Omega})=\{u\in H^1_0(\Omega): \Delta u\in L^2(\Omega)\}, \quad$$  then $\sF^{V_i,\Omega}_\re\cap \cH^{V_i}_\Gamma=\{0\}$ holds and $D_{V_i}$ is a lower semi-bounded and self-adjoint operator on $L^2(\Gamma)$.  

%consider    $\mu=\sigma$ and $\kappa_i=V_i(x)dx$ with $V_i\in L^\infty(\Omega)$ for $i=1,2$ as in Remark~\ref{RM51}.  Note that $D_{V_i}=\sN_{\kappa_i}$ is lower bounded self-adjoint on $L^2(\Gamma)$; see Remark~\ref{RM46}. The following lemma characterizes the graph $C_{V_i}$ of $D_{V_i}$.  
%Clearly $D_{V_1}=D_{V_2}$ amounts to $C_{V_1}=C_{V_2}$.  

\begin{lemma}\label{LM53}
For $i=1,2$,  it holds that
\begin{equation}\label{eq:54}
\begin{aligned}
C_{V_i}=\left\{(u|_\Gamma,  \frac{1}{2}\partial_\bn u): u\in H^{3/2}(\Omega),  -\frac{1}{2}\Delta u +V_i u=0\right\}.  
\end{aligned}
\end{equation}
\end{lemma}
\begin{proof}
Denote the family on the right side of \eqref{eq:54} by $\sG$.  To prove $C_{V_i}\subset \sG$,  
take $(\varphi,f)\in C_{V_i}$ and let $u$ be the function appearing in \eqref{eq:E0}, i.e.  $u\in H^1(\Omega)$,  $u|_\Gamma=\varphi$,  $-\frac{1}{2}\Delta u + V_i u=0$ weakly in $\Omega$ and 
\begin{equation}\label{eq:53}
	\frac{1}{2}\int_\Omega \nabla u \nabla v dx+\int_\Omega u(x)v(x)V_i(x)dx=\int_\Gamma f v|_\Gamma d\sigma,\quad \forall v\in H^1(\Omega). 
\end{equation}
It follows that $\Delta u=2V_i u\in L^2(\Omega)$ and for any $v\in H^1(\Omega)$,  
\begin{equation}\label{eq:52}
	\int_\Omega  \nabla u \nabla v dx + \int_\Omega \Delta u vdx=2\int_\Gamma f v|_\Gamma d\sigma.  
\end{equation}
Hence $u$ has a weak normal derivative in $L^2(\Gamma)$ and $f=\frac{1}{2}\partial_\bn u$.  On account of \eqref{eq:35},  we get $u\in H^{3/2}(\Omega)$.  Therefore $C_{V_i}\subset \sG$ is obtained.  To the contrary,  take $u\in H^{3/2}(\Omega)$ such that $-\frac{1}{2}\Delta u +V_i u=0$ weakly in $\Omega$.  Then $\Delta u=2V_i u\in L^2(\Omega)$ and \eqref{eq:35} implies that $u$ has a weak normal derivative $\partial_\bn u$ in $L^2(\Gamma)$,  i.e.  \eqref{eq:52} holds for $f:=\frac{1}{2}\partial_\bn u$.  Since $-\frac{1}{2}\Delta u +V_i u=0$ weakly in $\Omega$,  it follows that \eqref{eq:53} is true.  Thus $(\varphi,f)\in C_{V_i}$ with $\varphi=u|_\Gamma \in L^2(\Gamma)$ and $f=\frac{1}{2}\partial_\bn u$.  Eventually we can conclude \eqref{eq:54}.  
\end{proof}

The classical Calder\'on problem of dimension higher than $3$ has been solved in a seminal paper \cite{SU87};  see also \cite{U12}.  That is the following.  

\begin{theorem}
Consider a bounded domain $\Omega\subset \bR^d$,  $d\geq 3$, with smooth boundary.  Let $V_i\in L^\infty(\Omega)$ for $i=1,2$.  If $C_{V_1}=C_{V_2}$,  then $V_1=V_2$.  
\end{theorem}
\begin{proof}
Take $(\varphi,f)\in C_{V_1}=C_{V_2}$.
For $i=1,2$,  let $u_i\in H^{3/2}(\Omega)$ such that $u_i|_\Gamma=\varphi$ and $\frac{1}{2}\Delta u_i-V_i u_i=0$ as in Lemma~\ref{LM53}.  We assert that 
\begin{equation}\label{eq:51}
	\int_\Omega (V_1-V_2)u_1u_2dx=0.  
\end{equation}
In fact,  it follows from \eqref{eq:E0} that for any $v\in H^1(\Omega)$,  
\[
	\frac{1}{2}\int_\Omega \nabla u_i \nabla v dx+\int_\Omega u_i(x) v(x) V_i(x)dx=\int_\Gamma f\Tr(v)d\sigma.  
\]
Taking $v=u_2$ for $i=1$ and $v=u_1$ for $i=2$,  we get
\[
\frac{1}{2}\int_\Omega \nabla u_1 \nabla u_2 dx+\int_\Omega u_1(x) u_2(x) V_1(x)dx=\int_\Gamma f\varphi d\sigma
\]
and
\[
\frac{1}{2}\int_\Omega \nabla u_2 \nabla u_1 dx+\int_\Omega u_2(x) u_1(x) V_2(x)dx=\int_\Gamma f\varphi d\sigma.  
\]
Hence \eqref{eq:51} holds.  

The conclusion $V_1=V_2$ can be obtained as follows:  Note that the restrictions of the functions appearing in \cite[(18)]{U12} (with $q_i=V_i$) to $\Omega$ belong to $H^2(\Omega)\subset H^{3/2}(\Omega)$.  Particularly,  by means of \eqref{eq:54},   \eqref{eq:51} holds for $u_1,u_2$ of the form \cite[(18)]{U12}.  Hence $V_1=V_2$ follows from the same argument as that in the proof of \cite[Theorem~2.5]{U12}.  
\end{proof}

Regarding the two-dimensional case,  Bukhgeim \cite{B08} obtains $V_1=V_2$ under a slightly different condition to $C_{V_1}=C_{V_2}$;  see \cite[Theorem~2.1]{B08}.  

\begin{theorem}
Consider $d=2$ and $\Omega=\mathbb{D}=\{x: |x|<1\}$.  Let $V_i\in L^\infty(\Omega)$ for $i=1,2$.   Set for $i=1,2$ and $2<p\leq \infty$,
\[
	\tilde{C}_{V_i, p}=\left\{(u|_\Gamma,  \frac{1}{2}\partial_\bn u): u\in W^{2,p}(\Omega),  -\frac{1}{2}\Delta u +V_i u=0\right\}.
\]
If $\tilde C_{V_1, p}=\tilde C_{V_2,p}$ for some $2<p\leq \infty$,   then $V_1=V_2$.  
\end{theorem}
\begin{remark}
Note that $W^{2,p}(\Omega)\subset H^{3/2}(\Omega)$ for $p>2$.  Hence $\tilde{C}_{V_i,p}\subset C_{V_i}$.  
\end{remark}

%\section{Calder\'on's problem for perturbation on the boundary}

%For convenience,  we assume from now on $$\mathbf{P}_x(X^0_{\tau_\Omega}\in \partial \Omega)=0,$$  where $\tau_\Omega:=\inf\{t>0:X^0_t\notin \Omega\}$.  This is satisfied 

%\[
%\begin{aligned}
%	\sE(u,u)&=\frac{1}{2}\int_{\Omega\times \Omega\setminus \mathsf{d}}(u(x)-u(y))^2k(x,y)dxdy+\int_{\Omega\times \bar{\Omega}^c}(u(x)-u(y))^2k(x,y)dxdy  \\
%	&=\frac{1}{2}\| u(x)-u(y)\|^2_{L^2(\Omega\times \Omega\setminus \mathsf{d},  k(x,y)dxdy)}+\| u(x)-u(y)\|^2_{L^2(\Omega\times \bar\Omega^c,  k(x,y)dxdy)}.
%\end{aligned}\]

%\section*{Acknowledgement}
%The author would like to thank Dr. Wenjie Sun for some helpful discussion.

\bibliographystyle{abbrv}
\bibliography{DtN4}

\end{document}